%% file: dc-arxiv.tex
\documentclass{article}
\usepackage {amsmath, amsfonts, amssymb, mathrsfs, amsthm,stmaryrd, sectsty,hyphenat,etoolbox,tabu,calc,graphicx,url}
\usepackage[usenames]{color}
\usepackage{palatino}
\usepackage[all]{xy}
\usepackage[toc,page]{appendix}
\usepackage{enumitem}
\usepackage{showlabels}

\input{"macros.tex"}

\newcommand{\qedwhite}{\hfill\ensuremath{\Box}}

\theoremstyle{definition}
\newtheorem{thm}{Theorem}[section]
\newtheorem{lem}[thm]{Lemma}
\newtheorem{pro}[thm]{Proposition}
\newtheorem{cor}[thm]{Corollary}
\newtheorem{fct}[thm]{Fact}

\newtheorem{cnj}[thm]{Conjecture}

\newtheorem{dfn}[thm]{Definition}
\newtheorem{rem}[thm]{Remark}
\newtheorem{exa}[thm]{Example}
\newtheorem{cns}[thm]{Construction}
\newtheorem{asm}[thm]{Assumption}

\AtEndEnvironment{thm}{\qedwhite}
\AtEndEnvironment{lem}{\qedwhite}
\AtEndEnvironment{pro}{\qedwhite}
\AtEndEnvironment{cor}{\qedwhite}
\AtEndEnvironment{fct}{\qedwhite}
\AtEndEnvironment{clm}{\qedwhite}
\AtEndEnvironment{cnd}{\qedwhite}
\AtEndEnvironment{cnj}{\qedwhite}
\AtEndEnvironment{exn}{\qedwhite}
\AtEndEnvironment{dfn}{\qedwhite}
\AtEndEnvironment{rem}{\qedwhite}
\AtEndEnvironment{exa}{\qedwhite}
\AtEndEnvironment{cns}{\qedwhite}
\AtEndEnvironment{asm}{\qedwhite}

\sectionfont{\center\sc\normalsize}
\subsectionfont{\bf\normalsize}

\numberwithin{equation}{section}

\newlength{\sumcorr}
\def\sprod#1{\setlength{\sumcorr}{(\widthof{$\displaystyle\prod_{#1}$}-\widthof{$\displaystyle\prod$})/2} \hspace{-\sumcorr}\prod_{#1}\hspace{-\sumcorr} }

\begin{document}

\begin{mytitle}
On $L$-embeddings and double covers of tori over local fields
\end{mytitle}

\begin{center}
Tasho Kaletha
\end{center}

\begin{abstract}
To a torus $T$ over a local field $F$ and a subset of its character module subject to certain properties, we associate a canonical double cover $T(F)_\pm$ of the topological group $T(F)$ of $F$-rational points of $T$. We further associate an $L$-group $^LT_\pm$ to this double cover and establish a natural bijection between $L$-parameters valued in $^LT_\pm$ and genuine characters of $T(F)_\pm$. When $T$ is a maximal torus of a connected reductive group $G$, we show that there is a canonical $L$-embedding $^LT_\pm \to {^LG}$. This leads to a canonical factorization of Langlands parameters. We associate to a genuine character of $T(F)_\pm$ subject to certain conditions a Harish-Chandra character formula and use it to give a conjectural characterization of the supercuspidal local Langlands correspondence for $G$, subject to a certain condition on $p$. This generalizes previous work of Adams-Vogan \cite{AV92} for $F=\R$, and reinterprets computations of Langlands-Shelstad \cite{LS87}.
\end{abstract}

{\let\thefootnote\relax\footnotetext{This research is supported in part by NSF grant DMS-1801687 and a Sloan Fellowship.}}

\tableofcontents

\section{Introduction}

Let $G$ be a connected reductive group defined over a local field $F$ and let $T \subset G$ be a maximal torus. Let $\hat T$ and $\hat G$ be the corresponding dual groups, defined over $\C$, and equipped with algebraic actions of the absolute Galois group $\Gamma$ of $F$. There is a canonical $\hat G$-conjugacy class of embeddings $\hat T \to \hat G$. This conjugacy class is $\Gamma$-invariant, but need not contain a $\Gamma$-invariant member. A natural question is whether a member of this conjugacy class can be extended to an $L$-embedding $\hat T \rtimes \Gamma \to \hat G \rtimes \Gamma$ between the Galois forms of the corresponding $L$-groups, and whether a canonical such extension exists. Interest in this question comes from the local Langlands correspondence, which describes representations of $G(F)$ in terms of objects associated to $\hat G \rtimes \Gamma$.

In some cases the answer to the existence question is negative. For example, when $F=\R$, $G=\tx{PGL}_2$, and $T$ is anisotropic, such an embedding would have to take image in the normalizer of a maximal torus of $\tx{SL}_2(\C)$, which is a non-split extension of $\C^\times$ by $\Z/2\Z$, while $\hat T \rtimes \Gamma$ is the split such extension. In other cases the answer to the existence question is positive, but there is no natural choice.

Langlands observed that if one replaces $\hat T \rtimes \Gamma$ by $\hat T \rtimes W_F$, where $W_F$ is the absolute Weil group of $F$, then the answer to the existence question is always positive \cite[Lemma 4]{Lan79}. This existence result was later strenghtened by Langlands-Shelstad, who gave an explicit construction of an $L$-embedding $\hat T \rtimes W_F \to \hat G \rtimes W_F$ based on their concept of $\chi$-data \cite[\S2.5,\S2.6]{LS87}. This construction was initially motivated by the study of endoscopic transfer factors \cite{LS87}, but has since found other uses as well, for example in explicit constructions of the local Langlands correspondence and in the description of the Harish-Chandra characters of supercuspidal representations \cite{KalRSP}. The group $\hat T \rtimes W_F$ is called the Weil-form of the $L$-group of $T$, and the usage of the Weil form of the $L$-group has now become commonplace due to this phenomenon, since replacing $\hat G \rtimes \Gamma$ by $\hat G \rtimes W_F$ has no deterimental effects on other parts of the theory.

While $\chi$-data played a merely technical (albeit rather important) role in \cite{LS87}, its applications in \cite{KalRSP} raise the question of finding a better understanding of its nature. Indeed, an $L$-embedding $^LT \to {^LG}$ (now between the Weil forms of the $L$-groups) embodies functorial transfer of representations from $T(F)$ to $G(F)$. All discrete series representations of real groups and most supercuspidal representations of $p$-adic groups arise from such transfer. Since the transfer depends on the particular choice of $L$-embedding, this makes understanding the nature of such embeddings important.

Over the real numbers Shelstad studied in \cite{She81} $L$-embeddings $^LT \to {^LG}$, and more generally $^LH \to {^LG}$ for endoscopic groups $H$. She observed that to specify an $L$-embedding $^LT \to {^LG}$ it is enough to choose a Borel subgroup defined over $\C$ containing $T$. From the point of view of $\chi$-data this is expressed by the existence of a natural choice, which Shelstad calls ``based'' in \cite[\S9]{SheTE1}, associated to such a Borel subgroup. Unfortunately, over non-archimedean fields such a simple choice does not exist.

A different point of view emerges from the work of Adams and Vogan \cite{AV92}, \cite{Adm91}, \cite{AV16}, which introduces the concept of the $\rho$-cover of a maximal torus of a real reductive group. There is an associated $L$-group, which is a (usually) non-split extension of the dual torus by the Galois group, and therefore not isomorphic to the $L$-group of the torus itself. Emphasizing this fact, Adams and Vogan call it an ``$E$-group''. The $\rho$-cover and its $L$-group make the statement of the Harish-Chandara character formula for discrete series representations, as well as the construction of the local Langlands correspondence, more transparent. More recently, a slight variation of the $\rho$-cover has emerged, the so called $\rho_i$-cover, which makes constructions even more natural and has become the preferred device \cite{ALTV}. Its relationship to the $\rho$-cover is simple and explored in \cite[\S3]{Adm11}.

In this paper we introduce a double cover of the group of rational points of any algebraic torus defined over a local field and equipped with a certain finite subset of its character module (an essential example being the root system of an ambient reductive group). Unlike the $\rho_i$-cover of Adams-Vogan, which is of algebraic nature, the cover introduced here is of Galois-theoretic nature. Nevertheless, over the real numbers it recovers the $\rho_i$-cover.

We further define an $L$-group for our double cover -- it is a (usually) non-split extension of the absolute Galois group of $F$ (in fact already of the Galois group of the splitting extension) by the dual torus -- and prove a local Langlands correspondence, i.e. a bijection between the set of $L$-parameters valued in this $L$-group and the set of genuine characters of the double cover. The computations involved in the proof are all contained in \cite{LS87}. Therefore, this part of the paper can be seen as a reinterpretation of computations made in \cite{LS87}. We hope that the point of view introduced here illuminates the conceptual significance of these computations. 

Let us return to the setting of a maximal torus $T$ in a connected reductive group $G$ over $F$. The root system of $G$ leads by our general construction to a double cover $T(F)_\pm$ of $T(F)$ with a corresponding $L$-group $^LT_\pm$. It turns out that there is a canonical $\hat G$-conjugacy class of $L$-embeddings $^LT_\pm \to {^LG}$. Here one can take either Galois forms or Weil forms for the $L$-groups; it makes no difference. For varying elliptic maximal tori $T$ in $G$ these embeddings can be used to canonically factor $L$-parameters for $G$ that are discrete when $F$ is archimedean or supercuspidal and torally wild when $F$ is non-archimedean (the torally wild condition is automatic when $p$ does not divide the order of the Weyl group) to $L$-parameters for $T(F)_\pm$. In turn, one can associate to the resulting genuine character of $T(F)_\pm$ a natural stable distribution on $G(F)$. This gives a spectral characterization of the Langlands correspondence for discrete parameters for $F$ is archimedean and supercuspidal torally wild parameters when $F$ is non-archimedean. In this paper we state it only for non-archimedean $F$, cf. Conjecture \ref{cnj:stabchar}, since the archimedean case is discussed in \cite{AV16}. In the non-archimedean case, one needs to assume that $p$ is not too small, since the character formula uses the exponential map. This approach provides a different point of view on the explicit constructions of the local Langlands correspondence of \cite{KalRSP} and \cite{KalSLP}. 

How does $\chi$-data and the Langlands-Shelstad construction of an $L$\-embedding $^LT \to {^LG}$ fit into this point of view? First, $\chi$-data is used as a technical device in the proof of the local Langlands correspondence for the double cover $T(F)_\pm$. 
Second, while the Galois form of $^LT_\pm$ can be a non-split extension of $\Gamma$ by $\hat T$, its pull-back along $W_F \to \Gamma$, i.e. the Weil form of $^LT_\pm$, is always split, albeit non-canonically. A choice of $\chi$-data provides a splitting, and hence an $L$-isomorphism $^LT \to {^LT}_\pm$. Composed with the canonical $L$-embedding $^LT_\pm \to {^LG}$ this recovers the Langlands-Shelstad $L$-embedding associated to this $\chi$-data.

We generalize our investigation from maximal tori to twisted Levi subgroups. These are connected reductive subgroups of $G$ defined over $F$ that are Levi components of parabolic subgroups defined over some field extension of $F$. In particular, any maximal torus is a twisted Levi subgroup. We show that the Langlands\-Shelstad construction can be used to obtain an $L$-embedding $^LM \to {^LG}$ for any twisted Levi subgroup in terms of $\chi$-data for a suitable finite set derived from the root system of $G$. We then show that $M(F)$ possesses a canonical double cover $M(F)_\pm$, construct an $L$-group $^LM_\pm$ that is again a (usually) non-split extension of $\Gamma$ by $\hat M$, show how an assumed local Langlands correspondence for $M(F)$ easily implies such a correspondence for genuine representations of $M(F)_\pm$, and show that $^LM_\pm$ has a canonical embedding into $^LG$. These considerations are motivated by the study of supercuspidal representations of $p$-adic groups, which naturally arise in an inductive fashion from twisted Levi subgroups \cite{Yu01}. We hope they lay the groundwork for the study of general functorial transfer from twisted Levi subgroups.

The local Langlands correspondence for tori is functorial. It is therefore natural to ask if the correspondence for our double covers is also functorial. Even formulating this question appropriately is not completely straightforward, since the double cover depends on a subset of the character module that satisfies certain properties. In this paper we formulate and prove one instance of such functoriality. It is needed in the study of how an $L$-embedding $^LT \to {^LG}$ factors through the $L$-group of a twisted Levi subgroup. 

In 2008 Benedict Gross circulated privately a set of hand-written notes titled ``Groups of type $L$'', in which he made some observations about covers of tori over local fields and their $L$-groups. These observations work well for tori $S$ for which the norm map $S(E) \to S(F)$ is surjective, where $E/F$ is the splitting extension of $S$. While this condition is quite restrictive, it is fulfilled in a very important special case -- one-dimensional anisotropic tori. In Appendix \ref{app:glt} we have included some material from these notes and some further discussion, in the hope of shedding new light on some of the formulas appearing in \cite{LS87} and on our definition of the $L$-group of the double cover and the corresponding Langlands parameters.\\

\textbf{Acknowledgements:} The author thanks Jeff Adams for drawing attention to the problem of embedding the $L$-group of a maximal torus into the $L$-group of the ambient reductive group and to its elegant solution in the real case via the $\rho$-cover.

\section{Notation}

Throughout the paper $F$ will denote a local field of arbitrary characteristic, $F^s$ a separable closure, $\Gamma$ or $\Gamma_F$ the absolute Galois group of $F^s/F$, $W$ or $W_F$ the Weil group, and $\Sigma = \Gamma \times \{\pm 1\}$. We let $\Sigma$ act on $F^s$ with $\{\pm 1\}$ acting trivially.

Consider a set $R$ endowed with an action of $\Sigma$, which we will call a $\Sigma$-set for short. Following \cite{LS87}, a $\Sigma$-orbit $O \subset R$ is called symmetric if it is a single $\Gamma$-orbit, and asymmetric if it is the disjoint union of two $\Gamma$-orbits. For any $\alpha \in R$ we define $\Gamma_\alpha=\tx{Stab}(\alpha,\Gamma)$, $\Gamma_{\pm\alpha}=\tx{Stab}(\{\pm\alpha\},\Gamma)$, $F_\alpha=(F^s)^{\Gamma_\alpha}$, $F_{\pm\alpha}=(F^s)^{\Gamma_{\pm\alpha}}$. We introduce $\Sigma_\alpha$ and $\Sigma_{\pm\alpha}$ analogously. Thus $O$ is symmetric if and only if for one, hence any, $\alpha \in O$, one has $[F_\alpha:F_{\pm\alpha}]=2$. We may then call each $\alpha \in O$ symmetric. When $F$ is archimedean a symmetric $\alpha$ is traditionally called ``imaginary''. When $F$ is non-archimedean a symmetric $\alpha$ is called ramified or unramified depending on whether the quadratic extension $F_\alpha/F_{\pm\alpha}$ is ramified or unramified. We have the $\Sigma$-invariant decomposition $R=R_\tx{sym} \cup R_\tx{asym}$ into the subsets of symmetric and asymmetric elements, respectively. In the non-archimedean case we further have the $\Sigma$-invariant decomposition $R_\tx{sym}=R_\tx{sym,ram} \cup R_\tx{sym,unram}$.

For $\alpha \in R_\tx{sym}$ we denote by $\kappa_\alpha$ the non-trivial homomorphisms $\Gamma_{\pm\alpha}/\Gamma_\alpha \to \{\pm 1\}$ and $F_{\pm\alpha}^\times/N_{F_\alpha/F_{\pm\alpha}}(F_\alpha^\times) \to \{\pm 1\}$. We denote by $\tau_\alpha$ the non-trivial element of $\Gamma_{\pm\alpha}/\Gamma_\alpha$. We denote by $\Z_{(-1)}$ the abelian group $\Z$ on which $\Gamma_{\pm\alpha}/\Gamma_\alpha$ acts by multiplication by $\kappa_\alpha$. We denote by $\C_{(-1)}^\times$ the abelian group $\C^\times$ with $\Gamma_{\pm\alpha}/\Gamma_\alpha$-action where $\tau_\alpha$ acts by inversion. We denote by $F_\alpha^1$ the kernel of the norm homomorphism $N_{F_\alpha/F_{\pm\alpha}} : F_\alpha^\times \to F_{\pm\alpha}^\times$.

We will say that $R$ is \emph{admissible} if the action of $-1 \in \Sigma$ does not fix any element of $R$. Given an algebraic torus $S$ defined over $F$ we shall denote by $X^*(S)$ and $X_*(S)$ its character and co-character modules. These are $\Sigma$-modules, with $\Gamma$ acting in accordance with the $F$-structure of $S$, and $\{\pm 1\}$ acting by multiplication. A $\Sigma$-invariant subset $R \subset X^*(S)$ is admissible if and only if it does not contain $0$.

If $R$ is a $\Sigma$-set equipped with a $\Sigma$-invariant map $R \to X^*(S)$, the image of an element $\alpha \in R$ in $X^*(S)$ will be denoted by $\bar\alpha$.

\section{The double cover}

\subsection{Construction} \label{sub:const}

Let $S$ be an algebraic torus defined over $F$. Assume given a finite admissible $\Sigma$-set $R$ and a $\Sigma$-equivariant map $R \to X^*(S)$. The most important case is when $R$ is a $\Sigma$-invariant subset of $X^*(S)$ and the map is just the inclusion, but the slightly more general setting will become useful in \S\ref{sub:func} and \S\ref{sec:tl}, where it will give us more flexibility and will lead to a more natural statement of functoriality. Given this data we will define a double cover
\begin{equation} \label{eq:dc}
1 \to \{\pm 1\} \to S(F)_\pm \to S(F) \to 1. 	
\end{equation} 
When we want to emphasize the datum $R$ we shall write $S(F)_R$ in place of $S(F)_\pm$.

We begin by defining for any $\Sigma$-orbit $O \subset R$ an algebraic torus $J_O$ and a morphism $S \to J_O$, both defined over $F$. In fact, we are simply recalling a construction present in the proofs of Lemmas 3.2.D and 3.3.C of \cite{LS87}. 

First, let $O$ be asymmetric. Choose $\alpha \in O$ and let $J_\alpha=\tx{Res}_{F_\alpha/F}\mb{G}_m$, so that $X^*(J_\alpha)=\tx{Ind}_{\Gamma_\alpha}^\Gamma \Z$. We have $J_\alpha(F)=F_\alpha^\times$. The $\Gamma_\alpha$-equivariant morphism $\Z \to X^*(S)$ given by $n \mapsto n\bar\alpha$ specifies, by Frobenius reciprocity as reviewed in Appendix \ref{app:frob}, a $\Gamma$-equivariant morphism $X^*(J_\alpha) \to X^*(S)$, given explicitly by $f \mapsto \sum_{\tau \in \Gamma_\alpha\lmod\Gamma} f(\tau)\tau^{-1}\bar\alpha$. Another $\beta \in O$ is of the form $\beta=\epsilon\tau\alpha$ for some $\tau\in\Gamma$, whose coset in $\Gamma/\Gamma_\alpha$ is unique, and a unique $\epsilon \in \{\pm 1\}$. The map $f_\alpha \mapsto f_\beta$ given by $f_\beta(x)=\epsilon f_\alpha(\tau^{-1}x)$ is an isomorphism $X^*(J_\alpha) \to X^*(J_\beta)$, depending only on $\alpha$ and $\beta$, and identifying the two morphisms to $X^*(S)$. Taking the inverse limit over all $\alpha \in O$ we obtain the torus $J_O$ and the morphism $S \to J_O$. An element of $J_O(F)$ is a collection $(\gamma_\alpha)_{\alpha \in O}$, with $\gamma_\alpha \in F_\alpha^\times$, subject to $\gamma_{\epsilon\tau\alpha}=\tau(\gamma_\alpha)^\epsilon$.

Now let $O$ by symmetric. Choose $\alpha \in O$. Let $\tx{Res}^1_{F_\alpha/F_{\pm\alpha}}\mb{G}_m$ be the unique one-dimensional anisotropic torus defined over $F_{\pm\alpha}$ and split over $F_\alpha$. Its character module is the $\Gamma_{\pm\alpha}$-module $\Z_{(-1)}$. Let $J_\alpha=\tx{Res}_{F_{\pm\alpha}/F} \tx{Res}^1_{F_\alpha/F_{\pm\alpha}}\mb{G}_m$, so that
\[ X^*(J_\alpha)=\tx{Ind}_{\Gamma_{\pm\alpha}}^\Gamma \Z_{(-1)} = \{f : \Gamma \to \Z| \forall \sigma \in \Gamma_{\pm\alpha}: f(\sigma\tau)=\kappa_\alpha(\sigma)f(\tau)\}. \]
In particular, an element $f$ is constant on $\Gamma_\alpha$-cosets. We have $J_\alpha(F)=F_\alpha^1$.

The $\Gamma_{\pm\alpha}$-equivariant morphism $\Z_{(-1)} \to X^*(S)$ given by $n \mapsto n\bar\alpha$ specifies, again by Frobenius reciprocity, a $\Gamma$-equivariant morphism $X^*(J_\alpha) \to X^*(S)$, given again explicitly by $f \mapsto \sum_{\tau \in \Gamma_{\pm\alpha} \lmod \Gamma} f(\tau)\tau^{-1}\bar\alpha$. Another $\beta \in O$ is of the form $\beta=\tau\alpha$ for $\tau \in \Gamma$ whose coset in $\Gamma/\Gamma_\alpha$ is unique. Then $\Gamma_{\pm\beta}=\tau\Gamma_{\pm\alpha}\tau^{-1}$ and $\Gamma_{\beta}=\tau\Gamma_{\alpha}\tau^{-1}$, hence $\kappa_\beta=\kappa_\alpha\circ\tx{Ad}(\tau)$ and thus $f_\alpha \mapsto f_\beta$ with $f_\beta(x)=f_\alpha(\tau^{-1}x)$ provides again an isomorphism $X^*(J_\alpha) \to X^*(J_\beta)$ that depends only on $\alpha$ and $\beta$ and identifies the two morphisms to $X^*(S)$. Taking the inverse limit over all $\alpha \in O$ we again obtain a torus $J_O$ together with a morphism $S \to J_O$ defined over $F$. An element of $J_O(F)$ is a collection $(\gamma_\alpha)_{\alpha \in O}$ of $\gamma_\alpha \in F_\alpha^1$ satisfying $\gamma_{\sigma\alpha}=\sigma(\gamma_\alpha)$.

We have thus defined, for each $\Sigma$-orbit $O \subset R$, a torus $J_O$ and a morphism $S \to J_O$, both defined over $F$. Next we define a double cover $J_O(F)_\pm$ of the topological group $J_O(F)$. When $O$ is asymmetric we take the canonical split cover $J_O(F)_\pm = J_O(F) \times \{\pm 1\}$. When $O$ is symmetric we choose again $\alpha \in O$. The map $F_\alpha^\times \to F_\alpha^1$ sending $x$ to $x/\tau_\alpha(x)$ leads to the double cover
\[ 1 \to \frac{F_{\pm\alpha}^\times}{N_{F_\alpha^\times/F_{\pm\alpha}^\times}(F_\alpha^\times)} \to \frac{F_\alpha^\times}{N_{F_\alpha^\times/F_{\pm\alpha}^\times}(F_\alpha^\times)} \to F_\alpha^1 \to 1.  \]
Using the identification $J_\alpha(F)=F_\alpha^1$ and pusing out by the sign character $\kappa_\alpha : F_{\pm\alpha}^\times/N_{F_\alpha^\times/F_{\pm\alpha}^\times}(F_\alpha^\times) \to \{\pm 1\}$ we obtain a double cover
\[ 1 \to \{\pm 1\} \to J_\alpha(F)_\pm \to J_\alpha(F) \to 1. \]
It is easy to see that the transition isomorphism $J_\alpha \to J_\beta$ for any pair $\alpha,\beta \in O$ lifts naturally to a transition isomorphism $J_\alpha(F)_\pm \to J_\beta(F)_\pm$; it is simply given by $\tau : F_\alpha^\times \to F_\beta^\times$ for the unique $\tau \in \Gamma/\Gamma_\alpha$ satisfying $\tau\alpha=\beta$. Taking the limit we therefore obtain a double cover
\[ 1 \to \{\pm 1\} \to J_O(F)_\pm \to J_O(F) \to 1. \]

\begin{rem} \label{rem:sosplit}
Assume that $O$ is symmetric. The double cover $J_O(F)_\pm$ is split if and only if the canonical character $\kappa_\alpha : F_{\pm\alpha}^\times/N_{F_\alpha^\times/F_{\pm\alpha}^\times}(F_\alpha^\times) \to \{\pm 1\}$ can be extended to a character $F_{\alpha}^\times/N_{F_\alpha^\times/F_{\pm\alpha}^\times}(F_\alpha^\times) \to \{\pm 1\}$. When $F$ is non-archimedean and $F_\alpha/F_{\pm\alpha}$ is unramified there is a canonical such extension, so the cover $J_O(F)_\pm$ splits canonically. When $F$ is non-archimedean, $F_\alpha/F_{\pm\alpha}$ is ramified, and the size $q_\alpha$ of the residue field of $F_\alpha$ is congruent to $1$ modulo $4$, then such an extension does exist, but there is no canonical choice, so the cover $J_O(F)_\pm$ splits non-canonically. When $F$ is non-archimedean, $F_\alpha/F_{\pm\alpha}$ is ramified, and $q_\alpha$ is congruent to $3$ modulo $4$, or when $F=\R$, then there is no such extension, so the cover $J_O(F)_\pm$ does not split.
\end{rem}

\begin{rem} \label{rem:so2}
The pull-back of the double cover $J_O(F)_\pm \to J_O(F)$ along the squaring map $J_O(F) \to J_O(F)$ is canonically split. Indeed, it is enough to treat the case of symmetric $O$, in which case we choose $\alpha \in O$ and present this pull-back as
\[ \{(x,y)|x \in F_\alpha^1,y \in F_\alpha^\times/N_{F_\alpha/F_{\pm\alpha}}(F_\alpha^\times), x^2=y/\tau(y)\}. \]
The splitting of the map $(x,y) \mapsto x$ is given by $x \mapsto (x,x)$.
\end{rem}

We define an extension of $S(F)$ by $\prod\{\pm 1\}$, where the product runs over the set of $\Sigma$-orbits in $R$, as the pull-back of $S(F) \to \prod J_O(F) \from \prod J_O(F)_\pm$. We then push out this extension along the multiplication morphism $\prod \{\pm 1\} \to \{\pm 1\}$ to obtain the double cover \eqref{eq:dc}.

\begin{rem} \label{rem:split}
There are some situations in which the cover $S(F)_\pm$ splits canonically, i.e. there is a canonical identification $S(F)_\pm = S(F) \times \{\pm 1\}$. 
\begin{enumerate}
	\item If all elements of $R$ are asymmetric, then $S(F)_\pm$ splits canonically. This is because each $J_O(F)_\pm$ is by definition the canonical split extension.
	\item More generally, if the image in $X^*(S)$ of every symmetric element of $R$ is divisible by $2$, then $S(F)_\pm$ splits canonically. This follows from Remark \ref{rem:so2} and the fact that the morphism $S \to J_O$ factors through the squaring map $J_O \to J_O$ for each symmetric $O$.
	\item If $F$ is non-archimedean and all symmetric elements of $R$ are unramified, then $S(F)_\pm$ splits canonically. This follows from Remark \ref{rem:sosplit}.
\end{enumerate}
\end{rem}

\begin{rem} \label{rem:elements}
An element of $S(F)_\pm$ can be represented non\-uniquely by a tuple $(\gamma,\epsilon,(\delta_\alpha)_{\alpha \in R_\tx{sym}})$, where $\gamma \in S(F)$, $\epsilon \in \{\pm 1\}$, and $\delta_\alpha \in F_\alpha^\times$ satisfy $\delta_{\sigma\alpha}=\sigma(\delta_\alpha)$ and $\delta_\alpha/\delta_{-\alpha}=\bar\alpha(\gamma)$ for all $\sigma \in \Gamma$ and $\alpha \in R_\tx{sym}$. Note that $\delta_{-\alpha}=\tau_\alpha(\delta_\alpha)$. Another such collection representing the same element is of the form $(\gamma,\eta\cdot\epsilon,(\eta_\alpha \cdot\delta_\alpha)_{\alpha \in R_\tx{sym}})$, for $\eta_\alpha \in F_{\pm\alpha}^\times$ with $\eta_{\sigma\alpha}=\sigma(\eta_\alpha)$, and $\eta=\prod_{\alpha \in R_\tx{sym}/\Gamma} \kappa_\alpha(\eta_\alpha)$. The unique non-trivial element $\epsilon=-1$ in the kernel of $S(F)_\pm \to S(F)$ is represented by $(1,-1,(1))$.

The above presentation is valid without assuming that $R$ has symmetric elements. If we do make this assumption, we can make the more economical presentation of elements of $S(F)_\pm$ by tuples $(\gamma,(\delta_\alpha)_{\alpha \in R_\tx{sym}})$, where $\gamma \in S(F)$ and $\delta_\alpha \in F_\alpha^\times$ satisfy $\delta_{\sigma\alpha}=\sigma(\delta_\alpha)$ and $\delta_\alpha/\delta_{-\alpha}=\bar\alpha(\gamma)$ for all $\sigma \in \Gamma$ and $\alpha \in R_\tx{sym}$.  Another such collection represents the same element if and only if it is of the form $(\gamma,(\eta_\alpha\cdot\delta_\alpha)_{\alpha \in R_\tx{sym}})$ for $\eta_\alpha \in F_{\pm\alpha}^\times$ with $\eta_{\sigma\alpha}=\sigma(\eta_\alpha)$ and $\prod_{\alpha \in R_\tx{sym}/\Gamma} \kappa_\alpha(\eta_\alpha)=1$. The unique non-trivial element $\epsilon = -1$ in the kernel of $S(F)_\pm \to S(F)$ is now represented by $(1,(\eta_\alpha)_{\alpha \in R_\tx{sym}})$, where $\eta_\alpha \in F_{\pm\alpha}^\times$ satisfies $\eta_{\sigma\alpha}=\sigma(\eta_\alpha)$ and $\prod_{\alpha\in R_\tx{sym}/\Gamma}\kappa_\alpha(\eta_\alpha)=-1$.

Given an economical presentation $(\gamma,(\delta_\alpha)_{\alpha \in R_\tx{sym}})$ we obtain the non\-economical presentation $(\gamma,1,(\delta_\alpha)_{\alpha \in R_\tx{sym}})$.
\end{rem}

\begin{rem}
In what follows, we will quote some arguments and computations from \cite{LS87}. Strictly speaking, the arguments of \cite{LS87} apply to a finite $\Sigma$-invariant subset $0 \notin R \subset X^*(S)$, which appears slightly less general than having a $\Sigma$-equivariant map $R \to X^*(S)$ from a finite admissible $\Sigma$-set $R$. But for all arguments that we will quote, one can reduce to the more special case by the following simple construction. Let $\tilde M$ be the finitely generated free abelian group generated by the set $R$. Thus $\tilde M=\bigoplus_{\alpha \in R} \Z$. The group $\Sigma$ operates on $\tilde M$ linearly. On the other hand, $\Z$ operates on $\tilde M$ via the abelian group structure. We let $M$ be the quotient of $\tilde M$ under the action of the group $\{\pm 1\}$ embedded diagonally in $\Sigma \times \Z$. The admissibility of $\Sigma$ implies that $M$ is a finitely generated free abelian group. By construction it carries a $\Sigma$-action where $-1 \in \Sigma$ acts by multiplication by $-1 \in \Z$. The natural map $R \to \tilde M \to M$ is injective. The $\Sigma$-equivariant map $R \to X^*(S)$ extends canonically to a $\Sigma$-equivariant map $M \to X^*(S)$. If $T$ is the $F$-torus with $X^*(T)=M$, we obtain a natural map $S \to T$. The arguments of \cite{LS87} can be applied to the torus $T$ and then deduced for the torus $S$.
\end{rem}

\subsection{$\chi$-data and genuine characters} \label{sub:chichar}

We continue with a torus $S$ defined over $F$, a finite admissible $\Sigma$-set $R$, and a $\Sigma$-equivariant map $R \to X^*(S)$. We recall here the notions of $\chi$-data and $\zeta$-data introduced in \cite{LS87}; see \cite[\S2.5]{LS87} and \cite[Definition 4.6.4]{KalRSP}. A set of $\chi$-data consists of a character $\chi_\alpha : F_\alpha^\times \to \C^\times$ for each $\alpha \in R$. This set satisfies 
\begin{enumerate}
	\item $\chi_{\sigma\alpha}=\chi_\alpha\circ\sigma^{-1}$ for all $\sigma \in \Gamma$;
	\item $\chi_{-\alpha}=\chi_\alpha^{-1}$;
	\item $\chi_\alpha|_{F_{\pm\alpha}^\times}=\kappa_\alpha$ when $\alpha$ is symmetric.
\end{enumerate}
A set of $\zeta$-data consists of a character $\zeta_\alpha : F_\alpha^\times \to \C^\times$ for each $\alpha$ in $R$ and this set satisfies conditions 1 and 2 above, but in place of $3$ it satisfieds $\zeta_\alpha|_{F_{\pm\alpha}^\times}=1$ when $\alpha$ is symmetric.

We now recall from the proof of \cite[Lemma 3.5.A]{LS87} how a set of $\zeta$-data leads to a character $\zeta_S : S(F) \to \C^\times$ and use this as a motivation to show how a set of $\chi$-data leads to a genuine character of the double cover $S(F)_\pm$. Let $(\zeta_\alpha)_\alpha$ be a set of $\zeta$-data. For each $\Sigma$-orbit $O$ in $R$ define a character $\zeta_O : J_O(F) \rw \C^\times$ as follows. Choosing $\alpha \in O$ we obtain an identification $J_O = J_\alpha$. If $O$ is asymmetric let $\zeta_O = \zeta_\alpha$. If $O$ is symmetric $\zeta_O$ be the composition of $\zeta_\alpha$ with the isomorphism $F_\alpha^\times/F_{\pm\alpha}^\times \cong F_\alpha^1=J_O(F)$ sending $x \in F_\alpha^\times$ to $x/\tau_\alpha(x)$. In both cases it is straightforward to check that $\zeta_O$ depends only on $O$ and not on the choice of $\alpha$. Define $\zeta_S$ to be the pull-back of $\prod_O \zeta_O$ under the map $S(F) \to \prod_O J_O(F)$, where the products run over the set of $\Sigma$-orbits in $R$. This is the character of \cite[Definition 4.6.5]{KalRSP}, where it was defined in the slightly less general setting $R \subset X^*(S)$.

Let now $(\chi_\alpha)_\alpha$ be a set of $\chi$-data. We construct a genuine character $\chi_S : S(F)_\pm \to \C^\times$ as follows. Let $O \subset R$ be a $\Sigma$-orbit. Choose $\alpha \in O$ and obtain an identification $J_O = J_\alpha$. If $O$ is asymmetric then define $\chi_O = \chi_\alpha$ as a character of $J_O(F)=F_\alpha^\times$ and extend it to a genuine character of $J_O(F)_\pm$ in the unique way. If $O$ is symmetric then we have the commutative diagram
\[ \xymatrix{
1\ar[d]&1\ar[d]&&1\ar[d]\\
\{\pm 1\}\ar@{=}[r]\ar[d]&\{\pm 1\}\ar[d]&&F_{\pm\alpha}^\times/N_{F_\alpha/F_{\pm\alpha}}(F_\alpha^\times)\ar[d]\ar[ll]^{\kappa_\alpha}\\
J_O(F)_\pm\ar@{=}[r]\ar[d]&J_\alpha(F)_\pm\ar[d]&&F_\alpha^\times/N_{F_\alpha/F_{\pm\alpha}}(F_{\alpha}^\times)\ar[ll]\ar[d]\\
J_O(F)\ar@{=}[r]\ar[d]&J_\alpha(F)\ar@{=}[r]\ar[d]&F_\alpha^1&F_\alpha^\times/F_{\pm\alpha}^\times\ar[l]\ar[d]\\
1&1&&1
}
\]
in which the columns are exact and the horizontal maps are isomorphisms. We let $\chi_O$ be the character of $J_O(F)_\pm$ that is the composition of $\chi_\alpha$ with the middle horizontal isomorphism. It is a genuine character of the double cover $J_O(F)_\pm$ of $J_O(F)$. The product $\prod_O \chi_O$ is a character of $\prod_O J_O(F)_\pm$ whose restriction to $\prod_O \{\pm 1\}$ kills the kernel of the multiplication map $\prod_O \{\pm 1\} \to \{\pm1\}$ and therefore pulls back to a genunine character of $S(F)_\pm$.

\begin{dfn} \label{dfn:chis}
Let $\chi_S : S(F)_\pm \to \C^\times$ denote the character just constructed.
\end{dfn}

\begin{fct} \label{fct:charmult}
If $(\chi_\alpha)$ and $(\zeta_\alpha)$ are sets of $\chi$-data and $\zeta$-data for $R$, respectively, and $\chi_S : S(F)_\pm \to \C^\times$ and $\zeta_S : S(F) \to \C^\times$ are the corresponding characters, then $(\chi_\alpha \cdot \zeta_\alpha)$ is another set of $\chi$-data and we have
\[ (\chi\cdot\zeta)_S = \chi_S \cdot \zeta_S. \]
\end{fct}
\begin{proof}
Left to the reader.
\end{proof}

\subsection{The $L$-group} \label{sub:lgrp}

We continue with a torus $S$ defined over $F$, a finite admissible $\Sigma$-set, and a $\Sigma$-equivariant map $R \to X^*(S)$. Let $E/F$ be a finite Galois extension such that $\Gamma_E$ acts trivially on $R$ and $X^*(S)$. In the case when the map $R \to X^*(S)$ is injective, we can take $E/F$ to be the splitting extension of $S$. We shall define the $L$-group of the double cover $S(F)_\pm$, first in its Galois form. This will be an extension of $\Gamma_{E/F}$ by $\hat S$, usually not split, and unique up to isomorphism that is unique up to conjugation by $\hat S$. 

Recall the concept of a gauge from \cite{LS87}. It is a function $p : R \to \{\pm 1\}$ having the property $p(-\alpha)=-p(\alpha)$. 

\begin{rem}
A gauge need not be $\Gamma$-invariant. In fact, a $\Gamma$-invariant gauge exists if and only if all elements of $R$ are asymmetric.	
\end{rem}

\begin{dfn} \label{dfn:tp}
The \emph{Tits cocycle} corresponding to $p$ is $t_p \in Z^2(\Gamma_{E/F},\hat S)$ defined by $t_p(\sigma,\tau)=(-1)^{\lambda_p(\sigma,\tau)}$, where $\lambda_p(\sigma,\tau) \in X^*(S)=X_*(\hat S)$ is the sum of $\bar\alpha$ for $\alpha \in \Lambda_p(\sigma,\tau) \subset R$ defined as
\[  \{ \alpha \in R|p(\alpha)=+1,p(\sigma^{-1}\alpha)=-1,p((\sigma\tau)^{-1}\alpha)=+1 \}. \]
\end{dfn}
It is shown in \cite[Lemmas 2.1.B]{LS87} that $t_p$ is indeed a $2$-cocycle. According to \cite[Lemma 2.1.C]{LS87} the cohomology class of $t_p$ is independent of the choice of gauge $p$. In fact, according to \cite[Lemma 2.4.A]{LS87} there is a canonical co-chain $s_{p/q} \in C^1(\Gamma_{E/F},\hat S)$ for two gauges $p,q : R \to \{\pm1\}$ with the property $t_p/t_q = \partial s_{p/q}$. We recall its definition: $s_{p/q}(\sigma)=(-1)^{\mu_{p/q}(\sigma)}$, where $\mu_{p/q}(\sigma) \in X^*(S)$ is the sum of $\bar\alpha$ for $\alpha \in M_{p/q}(\sigma) \subset R$ defined as the disjoint union of
\[ \{\alpha \in R|p(\alpha)=+1,p(\sigma^{-1}\alpha)=-1,q(\alpha)=q(\sigma^{-1}\alpha)=+1\} \]
and
\[ \{\alpha \in R|p(\alpha)=p(\sigma^{-1}\alpha)=+1,q(\alpha)=-1,q(\sigma^{-1}\alpha)=+1\}. \]
Finally, \cite[Lemma 2.4.B]{LS87} shows that for three gauges $p,q,r$ the 1-cochains $s_{p/q}$ and $s_{q/p}$ are cohomologous, and the 1-cochains $s_{p/q}\cdot s_{q/r}$ and $s_{p/r}$ are cohomologous.

After these preparations we now define the $L$-group of $S(F)_\pm$ as the twisted product $\hat S \boxtimes_{t_p} \Gamma_{E/F}$ with respect to an arbitrary choice of gauge $p : R \to \{\pm 1\}$ and note that modulo conjugation by $\hat S$ this is independent of the choice of $p$. More formally we note that to specify an extension of $\Gamma_{E/F}$ by $\hat S$ that is unique up to isomorphism that in turn is unique up to conjugation by $\hat S$ is the same as to specify an object, unique up to unique isomorphism, in the category whose objects are extensions of $\Gamma_{E/F}$ by complex tori, and the set of morphisms from the object $1 \to \hat T_1 \to \mc{T}_1 \to \Gamma_{E/F} \to 1$ to the object $1 \to \hat T_2 \to \mc{T}_2 \to \Gamma_{E/F} \to 1$ is the set of $\hat T_2$-conjugacy classes of morphisms of extensions.

\begin{dfn} \label{dfn:lgrp}
The \emph{Galois form of the $L$-group} $^LS_\pm$ is the inverse limit, in the above category, of the system of extensions $^LS_\pm^{(p)}= \hat S \boxtimes_{t_p} \Gamma_{E/F}$, one for each gauge $p : R \to \{\pm 1\}$, with transition isomorphisms defined for any pair of gauges $p,q$ by
	\[ \xi_{p,q} : {^LS_\pm^{(p)}} \to {^LS_\pm^{(q)}},\quad s \boxtimes \sigma \mapsto s \cdot s_{p/q} \boxtimes \sigma. \]
The \emph{Weil form of the $L$-group} is the pull-back along $W_{E/F} \to \Gamma_{E/F}$ of the Galois form.
\end{dfn}

\begin{rem} We can also consider the absolute Galois and Weil forms of the $L$-group, by pulling back along the projections $\Gamma \to \Gamma_{E/F}$ and $W \to W_{E/F}$.

The Galois form of the $L$-group can often be a non-split extension of $\Gamma$ by $\hat S$. The Weil form, on the other hand, is always split, but not canonically. In the non-archimedean case, the absolute Galois form is also split, but in general not canonically. We will discuss this in the next subsection.
\end{rem}

\begin{rem}
The formula for $s_{p/q}$, and the definition of $^LS_\pm$, may seem a bit ad-hoc. We provide more motivation in Appendix \ref{app:glt}, especially Remark \ref{rem:spq} and \S\ref{sub:gentori}.
\end{rem}

\begin{rem} \label{rem:dsplit}
Assume that all elements of $R$ are asymmetric, so that a $\Gamma$-invariant gauge $p : R \to \{\pm 1\}$ exists. Then $t_p=1$ and hence we have the canonical identification $^LS_\pm = {^LS}$. More generally, the same argument works when all symmetric elements of $R$ are divisible by $2$ in $X^*(S)$. In that case we choose the gauge $p$ to have $\Gamma$-invariant restriction to all asymmetric orbits, so that the only elements of $\Lambda_p(\sigma,\tau)$ are symmetric, and again see that $t_p=1$. This is the dual incarnation of the canonical isomorphism $S(F)_\pm=S(F) \times \{\pm1\}$ of Remark \ref{rem:split}.

When $F$ is non-archimedean and all symmetric elements are unramified, the situation is slightly more subtle. While the cover $S(F)_\pm$ is still canonically split, its dual group $^LS_\pm$ need not be a split extension of $\Gamma_{E/F}$ by $\hat S$. However, its inflation to $\Gamma$ does split canonically, see Remark \ref{rem:dsplituf} below.
\end{rem}

\subsection{The local Langlands correspondence}

We continue with a torus $S$ defined over $F$ a finite admissible $\Sigma$-set, and a $\Sigma$-equivariant map $R \to X^*(S)$. We can apply the usual definition of a Langlands parameter to the $L$-group $^LS_\pm$:

\begin{dfn} \label{dfn:lp}
A \emph{Langlands parameter} for $S(F)_\pm$ (equivalently for $^LS_\pm$) is a continuous homomorphism $W_F \to {^LS}_\pm$ that commutes with the maps of both sides to $\Gamma$. Two such parameters are considered \emph{equivalent} if they are conjugate under $\hat S$.
\end{dfn}

\begin{rem} \label{rem:lp}
Since $^LS$ is a split extension endowed with a splitting, a Langlands parameter for $S$ is the same as a continuous 1-cocycle $W_F \to \hat S$, and two parameters are equivalent if the corresponding 1-cocycles are cohomologous.

We can describe Langlands parameters for $^LS_\pm$ in a similar manner. If we fix a gauge $p : R \to \{\pm 1\}$ then a Langlands parameter for $^LS_\pm$ is the same as a 1-cochain $W_F \to \hat S$ whose differential equals the Tits cocycle $t_p$. If we prefer not to fix a gauge, then we can describe an equivalence class of Langlands parameters for $^LS_\pm$ as a collection of cohomology classes of 1-cochains $u_p$, one for each gauge $p$, s.t. $u_p \cdot s_{q/p} = u_q$ as cohomology classes, and the differential of $u_p$ equals $t_p$ for each $p$.

Note that if $(u_p)$ is such a collection and $z : W_F \to \hat S$ is a 1-cocycle, then $(z \cdot u_p)$ is another such collection. Conversely if $(u_p)$ and $(v_p)$ are two such collections, then $u_p \cdot v_p$ is a 1-cocycle and independent of $p$.
\end{rem}

\begin{thm} \label{thm:lldc}
\begin{enumerate}
	\item There is a natural bijection between the set of equivalence classes of Langlands parameters for $S(F)_\pm$ and the set of continuous genuine characters of $S(F)_\pm$.
	\item Given a genuine character of $S(F)_\pm$ and a character of $S(F)$, the parameter of the product of these characters is the product of their parameters.
	\item Given two genuine characters of $S(F)_\pm$, the product of their parameters is the 1-cocycle $W_F \to \hat S$ that is the parameter of their product.
\end{enumerate}
\end{thm}

The word ``natural'' in the first part of this theorem has the naive meaning of naturally occurring and independent of choices, and does not (yet) signify that it is functorial. We do prove certain cases of functoriality in \S\ref{sec:func}.

The proof of this theorem is implicit in the computations of \cite[\S2.5]{LS87}. Before we explain that we offer some motivation for the resulting formulas. In \S\ref{sub:chichar} we recalled the construction of a character $\zeta_S : S(F) \to \C^\times$ associated to a set $(\zeta_\alpha)$ of $\zeta$-data. 

\begin{lem} \label{lem:zsp}
The element of $H^1(W_F,\hat S)$ corresponding to the character $\zeta_S$ can be represented by the 1-cocycle
\[ z(w)=z_\zeta(w)= \prod_{\alpha \in R/\Sigma} \prod_{i=1}^n \zeta_\alpha(v_0(u_i(w)))^{w_i^{-1}\bar\alpha}. \]
Here for each $\alpha \in R$ we have fixed representatives $w_1,\dots,w_n \in W$ for $W_{\pm\alpha} \lmod W$, $v_0 \in W_\alpha$ is arbitrary, and $v_1 \in W_{\pm\alpha} \sm W_\alpha$ is arbitrary when $\alpha$ is symmetric. For $w \in W$ and $1 \leq i \leq n$ the element $u_i(w) \in W_{\pm\alpha}$ is defined by $w_i \cdot w = u_i(w) \cdot w_j$, where $1 \leq j \leq n$ is also uniquely determined. For $u \in W_{\pm\alpha}$ and $0 \leq i \leq 1$ the element $v_i(u) \in W_\alpha$ is defined by $v_i \cdot u = v_i(u) \cdot v_j$, where $0 \leq j \leq 1$ is also uniquely determined.
\end{lem}
\begin{proof}
The arguments for this are contained in the proof of \cite[Lemma 3.5.A]{LS87}. We repeat them here, in a slightly different form, for the sake of illustration.

Since $\zeta_S$ was the pull-back to $S(F)$ of the product of characters $\zeta_O : J_O(F) \to \C^\times$ over the set of $\Sigma$-orbits $O$ in $R$, it will be enough to compute the corresponding 1-cocycle for each $O$. Choose $\alpha \in O$ and obtain identification $J_O=J_\alpha$. When $O$ is asymmetric the character $\zeta_O$ is the composition of $\zeta_\alpha : F_\alpha^\times \to \C^\times$ with the identification $J_\alpha(F)=F_\alpha^\times$. Dually we have $\hat J_O=X^*(J_O) \otimes_\Z \C^\times = \tx{Ind}_{\Gamma_\alpha}^\Gamma \C^\times$ and the parameter of $\zeta_O$ is the image of $\zeta_\alpha \in Z^1(W_{F_\alpha},\C^\times)$ under the Shapiro map $Z^1(W_{F_\alpha},\C^\times) \to Z^1(W_F,\tx{Ind}_{\Gamma_\alpha}^\Gamma \C^\times)$. Tracing through the computation we see that the value at $w \in W$ of this cocycle as an element of $\tx{Ind}_{W_\alpha}^W \C^\times$ whose value at $w_i$ is $\zeta_\alpha(u_i(w))$. For the sake of uniformity with the symmetric case we note that this equals $\zeta_\alpha(v_0(u_i(w)))$, since $v_0(u_i(w))=v_0 \cdot u_i(w) \cdot v_0^{-1}$.

When $O$ is symmetric the character $\zeta_O$ is the composition of $\zeta_\alpha : F_\alpha^\times/F_{\pm\alpha}^\times \to \C^\times$, the isomorphism $F_\alpha^\times/F_{\pm\alpha}^\times \cong F_\alpha^1$, and the identification $J_\alpha(F)=F_\alpha^1$. When dualizing this procedure we interpret $F_\alpha^1$ as the $F_{\pm\alpha}$-points of the torus $\tx{Res}^1_{F_\alpha/F_{\pm\alpha}}\mb{G}_m$, and $F_\alpha^\times/F_{\pm\alpha}^\times$ as the $F_{\pm\alpha}$-points of the torus $(\tx{Res}_{F_\alpha/F_{\pm\alpha}}\mb{G}_m)/\mb{G}_m$. The dual torus of $(\tx{Res}_{F_\alpha/F_{\pm\alpha}}\mb{G}_m)/\mb{G}_m$ is the kernel of the norm homomorphism $\tx{Ind}_{\Gamma_\alpha}^{\Gamma_{\pm\alpha}} \C^\times \to \C^\times$. The dual torus of $\tx{Res}^1_{F_\alpha/F_{\pm\alpha}}\mb{G}_m$ is $\C^\times_{(-1)}$. We have the isomorphism of algebraic tori $(\tx{id}-\tau_\alpha) : (\tx{Res}_{F_\alpha/F_{\pm\alpha}}\mb{G}_m)/\mb{G}_m \to \tx{Res}^1_{F_\alpha/F_{\pm\alpha}}\mb{G}_m$ which on rational points recovers the isomorphism $F_\alpha^\times/F_{\pm\alpha}^\times \to F_\alpha^1$. Its dual is given by the anti-diagonal embedding of $\C^\times_{(-1)}$ into $\tx{Ind}_{\Gamma_\alpha}^{\Gamma_{\pm\alpha}}\C^\times$.

Under the Shapiro map $Z^1(W_\alpha,\C^\times) \to Z^1(W_{\pm\alpha},\tx{Ind}_{W_\alpha}^{W_{\pm\alpha}} \C^\times)$ the character $\zeta_\alpha$ maps to the cocycle whose value at $u \in W_{\pm\alpha}$ is the element of $\tx{Ind}_{W_\alpha}^{W_{\pm\alpha}} \C^\times$ whose value at $v_i$ is $\zeta_\alpha(v_i(u))$. The assumption that $\zeta_\alpha|_{F_{\pm\alpha}^\times}=1$ is equivalent to the claim that this element of $Z^1(W_{\pm\alpha},\tx{Ind}_{W_\alpha}^{W_{\pm\alpha}} \C^\times)$ maps trivially to $Z^1(W_{\pm\alpha},\C^\times)$ under the norm map. In other words, the value of this cocycle at each $u \in W_{\pm\alpha}$ lies in the kernel of the norm map $\tx{Ind}_{\Gamma_\alpha}^{\Gamma_{\pm\alpha}} \C^\times \to \C^\times$, hence in the image of the anti-diagonal embedding $\C^\times_{(-1)} \to \tx{Ind}_{\Gamma_\alpha}^{\Gamma_{\pm\alpha}}\C^\times$. This gives the element of $Z^1(W_{\pm\alpha},\C^\times_{(-1)})$ whose value at $u \in W_{\pm\alpha}$ is $\zeta_\alpha(v_0(u))$. Now $\hat J_O=\tx{Ind}_{\Gamma_{\pm\alpha}}^\Gamma \C^\times_{(-1)}$. Under the Shapiro map $Z^1(W_{\pm\alpha},\C^\times_{(-1)}) \to Z^1(W_F,\hat J_O)$ we obtain the 1-cocycle whose value at $w \in W$ is the element of $\tx{Ind}_{\Gamma_{\pm\alpha}}^\Gamma \C^\times_{(-1)}$ whose value at $w_i$ is $\zeta_\alpha(v_0(u_i(w)))$.

It remains to see that, in both cases, the image of the parameter for $\zeta_O$ under the map $\hat J_O \to \hat S$ is given by the product over $i$ in the statement of the lemma. This is immediate from the description of the map $S \to J_O$ given in \S\ref{sub:const}, which implies that its dual, $\tx{Ind}_{\Gamma_\alpha}^\Gamma \C^\times \to \hat S$ when $\alpha$ is asymmetric and $\tx{Ind}_{\Gamma_{\pm\alpha}}^\Gamma \C_{(-1)}^\times \to \hat S$ when $\alpha$ is symmetric, is given by the formula $f \mapsto \prod_{\tau \in \Gamma_{\pm\alpha}\lmod \Gamma} f(\tau)^{\tau^{-1}\bar\alpha}$.
\end{proof}

This Lemma motivates the following definition. Let $(\chi_\alpha)$ be a set of $\chi$-data for $R$ and let $\chi_S : S(F)_\pm \to \C^\times$ be the character of Definition \ref{dfn:chis}.

\begin{dfn} \label{dfn:rp}
The \emph{Langlands parameter of $\chi_S$} is the 1-cochain
\[ r_p(w) = r_{p,\chi}(w) = \prod_{\alpha \in R/\Sigma} \prod_{i=1}^n \chi_\alpha(v_0(u_i(w)))^{w_i^{-1}\bar\alpha}, \]
where the notation is as in Lemma \ref{lem:zsp} and $p : R \to \{\pm 1\}$ is the gauge specified by $p(w_i^{-1}\alpha)=+1$. For an arbitrary gauge $q$ the 1-cochain $r_q$ is defined by $s_{q/p}r_p$.
\end{dfn}

\begin{rem} This is the formula occurring before the statement of Lemma 2.5.A in \cite{LS87} and we have kept the name $r_p$ of loc. cit. to emphasize that. For a different expression, in terms of a section $s : W_{\pm\alpha} \lmod W \to W$, see \eqref{eq:rpalt}. For a further motivation about this formula, see Appendix \ref{app:glt}, especially \S\ref{sub:gentori}.
\end{rem}

\begin{lem} \label{lem:lparpm}
The collection $(r_p)$ is indeed a valid Langlands parameter for $S(F)_\pm$.
\end{lem}
\begin{proof}
Following the discussion in Remark \ref{rem:lp} we must show the following: 
\begin{enumerate}
	\item Up to equivalence the collection $(r_p)$ does not depend on the choices of $w_i$ and $v_i$ made in Lemma \ref{lem:zsp};
	\item The differential of $r_p$ is $t_p$.
\end{enumerate}
The second point is \cite[Lemma 2.5.A]{LS87} and the first point is discussed right after the proof of that lemma in loc. cit.
\end{proof}

\begin{lem} \label{lem:rpz}
Let $(\zeta_\alpha)$ be a set of $\zeta$-data for $R$ and let $\zeta_S : S(F) \to \C^\times$ be the associated character. The Langlands parameter for $(\chi\zeta)_S$ is the product of the Langlands parameters for $\chi_S$ and $\zeta_S$.
\end{lem}
\begin{proof}
This is stated without proof in \cite[(2.6.3)]{LS87} and here we also leave the proof to the reader.
\end{proof}

\begin{proof}[Proof of Theorem \ref{thm:lldc}]
Choose a set of $\chi$-data $(\chi_\alpha)$. Let $\chi_S$ be the corresponding genuine character of Definition \ref{dfn:chis} and let $(r_p)$ be its parameter as in Definition \ref{dfn:rp}. Every other genuine character is of the form $\theta \cdot \chi_S$ for a character $\theta : S(F) \to \C^\times$. Letting $t : W_F \to \hat S$ be the parameter for $\theta$ we assign to $\theta \cdot \chi_S$ the parameter $(t \cdot r_p)$. This gives the desired bijection.

If $(\chi'_\alpha)$ is another set of $\chi$-data with corresponding character $\chi'_S$ and parameter $r'_p$ we define $\zeta$-data $(\zeta_\alpha)$ by $\zeta_\alpha=\chi'_\alpha\chi_\alpha^{-1}$. Let $z : W \to \hat S$ be the parameter of $\zeta_S$ as in Lemma \ref{lem:zsp}. Then Fact \ref{fct:charmult} and Lemma \ref{lem:rpz} show that $\chi'_S=\chi_S \cdot \zeta_S$ and $r'_p = r_p \cdot z$, proving that the bijection is independent of the choice of $(\chi_\alpha)$.

The multiplicativity statements are immediate.
\end{proof}

We can reinterpet Theorem  \ref{thm:lldc} by considering the Weil form of $^LS_\pm$. As we remarked earlier, \cite[Lemma 4]{Lan79} implies that this is a split extension of $W$ by $\hat S$ but without a distinguished splitting. That is, the Weil forms of $^LS_\pm$ and $^LS$ are non-canonically isomorphic.

\begin{pro} Let $(\chi_\alpha)$ be a set of $\chi$-data. The map
\[ \hat S \rtimes W \to \hat S \boxtimes_{t_p} W,\qquad s \rtimes w \mapsto s \cdot r_{p,\chi}(w) \boxtimes w \]
is an isomorphism of extensions whose $\hat S$-conjugacy class depends only on $(\chi_\alpha)$. If $\theta : S(F) \to \C^\times$ is a character with parameter $t$, the composition of $t$ with this isomorphism is the parameter for $\theta \cdot \chi_S$.
\end{pro}
\begin{proof}
This is just a reformulation of Theorem \ref{thm:lldc}. 
\end{proof}

\begin{rem} \label{rem:dsplituf}
When $F$ is non-archimedean and all symmetric elements of $R$ are unramified, the cover $S(F)_\pm$ splits canonically. This is not true for the $L$-group when viewed as an extension of $\Gamma_{E/F}$ by $\hat S$. Indeed, in the most basic case where $E/F$ is an unramified quadratic extension and $S$ is the 1-dimensional anisotropic torus defined over $F$ and split over $E$ and $R \subset X^*(S)$ consists of the two generators of this free $\Z$-module of rank $1$, then choosing a gauge on $R$ is equivalent to choosing a generator, and then $t_p$ is the unique non-trivial element of $H^2(\Gamma_{E/F},\C^\times_{(-1)})$, and so $^LS_\pm$ is not a split extension. 

However, there is a canonical choice of $\chi$-data, namely $\chi_\alpha$ is the unramified quadratic character for every symmetric $\alpha$. The corresponding Langlands parameter $r_{p,\chi}$ of Definition \ref{dfn:rp} extends to $\Gamma$ and provides a 1-cochain splitting $t_p$. Thus $^LS_\pm$, when viewed as an extension of $\Gamma$ by $\hat S$, splits canonically.

More generally, when $F$ is non-archimedean, the absolute Galois form of the $L$-group $^LS_\pm$ always splits, but not canonically in general. The reason for the splitting is that there always exist $\chi$-data of finite order, and then $r_{p,\chi}$ extends to $\Gamma$.
\end{rem}

\subsection{$a$-data as a geniune function on the double cover}

We continue with a torus $S$ defined over $F$, a finite admissible $\Sigma$-set $R$, and a $\Sigma$-equivariant map $R \to X^*(S)$. Recall from \cite[\S2.2]{LS87} the notion of $a$-data: A set of $a$-data consists of an element $a_\alpha \in F_\alpha^\times$ for each $\alpha \in R$ s.t. $a_{\sigma\alpha}=\sigma(a_\alpha)$ and $a_{-\alpha}=-a_\alpha$ for all $\alpha \in R$ and $\sigma \in \Gamma$. Given such a set we obtain the function 
\begin{equation} \label{eq:adata}
a_S : S(F)_\pm \to \{1\}\qquad (\gamma,(\delta_\alpha)_\alpha) \mapsto \prod_{\substack{\alpha \in R_\tx{sym}/\Gamma\\ \bar\alpha(\gamma)\neq 1}}\kappa_\alpha\left(\frac{\delta_\alpha-\delta_{-\alpha}}{a_\alpha}\right).
\end{equation}
Here elements of $S(F)_\pm$ are represented as in Remark \ref{rem:elements}. Both the numerator and the denominator of the argument of $\kappa_\alpha$ are non-zero elements of $F_\alpha$ of trace zero, so their quotient is a non-zero element of $F_{\pm\alpha}$. The values $a_\alpha$ for $\alpha \in R_\tx{asym}$ are irrelevant to this function. 

\begin{fct} \label{fct:asgen}
The function $a_S$ is genuine: For $\tilde\gamma \in S(F)_\pm$ and $\epsilon=-1 \in S(F)_\pm$ one has $a_S(\epsilon\tilde\gamma)=-a_S(\tilde\gamma)$.
\end{fct}

\begin{rem} \label{rem:moda}
Let $F$ be non-archimedean. Recall from \cite[Definition 4.6.8]{KalRSP} the concept of mod-$a$-data: It consists of a real number $r_\alpha$ and a non-zero element $\bar a_\alpha \in [F_\alpha]_{r_\alpha}/[F_\alpha]_{r_\alpha+}$ for each $\alpha \in R$ s.t. $r_{\sigma\alpha}=r_\alpha=r_{-\alpha}$, $\bar a_{\sigma\alpha}=\sigma(\bar a_\alpha)$, and $\bar a_{-\alpha}=-\bar a_\alpha$, for all $\alpha \in R$ and $\sigma \in \Gamma$. 

If $p\neq 2$, the function $a_S$ can be defined by the same formula with mod-$a$-data in place of $a$-data. Indeed, note first that mod-$a$-data can be lifted to $a$-data. For this it is enough to lift $\bar a_\alpha$ to an element $a_\alpha \in F_\alpha$ of trace zero. An arbitrary lift $\dot a_\alpha$ satisfies $z_\alpha := \dot a_\alpha+\tau_\alpha(\dot a_\alpha) \in [F_{\pm\alpha}]_{r_\alpha+}$ and $a_\alpha=\dot a_\alpha-z_\alpha/2$ works. Next, two lifts $a_\alpha$ and $a'_\alpha$ of $\bar a_\alpha$ to $F_\alpha$ are related by $a'_\alpha=a_\alpha \cdot u_\alpha$ for some $u_\alpha \in [F_{\pm\alpha}^\times]_{0+}$, but $\kappa_\alpha$ is trivial on $[F_{\pm\alpha}^\times]_{0+}$ when $p \neq 2$.
\end{rem}

\section{Applications}

\subsection{The canonical $L$-embedding} \label{sub:canemb}

Let now $G$ be a connected reductive group defined over $F$ with complex dual $\hat G$. Let $S \subset G$ be a maximal torus defined over $F$. It is well known that there exists a canonical $\hat G$-conjugacy class of embeddings $\hat S \to \hat G$, see e.g. \cite[\S5.1]{KalRSP}. We will show that this conjugacy class extends to a canonical $\hat G$-conjugacy class of $L$-embeddings $^LS_\pm \to {^LG}$, where $S(F)_\pm$ is constructed with respect to the root system $R \subset X^*(S)$ of $G$.

Choose a $\Gamma$-invariant pinning $(\hat T,\hat B,\{X_{\alpha}\})$ of $\hat G$ and let $\Omega=N(\hat T,\hat G)/\hat T$. Recall the Tits section
\[ n : \Omega \to N(\hat T,\hat G). \]
It sends a simple reflection $s_\alpha \in \Omega$ to the image of the standard Weyl element of $\tx{SL}_2$ under the unique morphism $\tx{SL}_2 \to \hat G$ mapping the standard nilpotent element to the element $X_\alpha$ of the pinning. Explicitly, $s_\alpha$ is sent to $\exp(X_\alpha)\exp(-X_{-\alpha})\exp(X_\alpha)$, where $X_{-\alpha}$ is the unique vector in the root space for the root $-\alpha$ that satisfies $[X_\alpha,X_{-\alpha}]=H_\alpha=d\alpha^\vee(1)$. It sends a general $\omega \in \Omega$ to $n(s_1)\dots n(s_k)$, where $\omega=s_1\dots s_k$ is any reduced expression. It is easy to see that $n$ is $\Gamma$-equivariant, hence extends to a section
\[ n : \Omega \rtimes \Gamma \to N(\hat T,\hat G) \rtimes \Gamma, \qquad \omega \rtimes \sigma \mapsto n(\omega) \rtimes \sigma. \]

Choose a member of the conjugacy class of embeddings $\hat S \to \hat G$ with image $\hat T$; it is an isomorphism $\hat j : \hat S \to \hat T$. This isomorphism is usually not $\Gamma$-equivariant, but since its $\hat G$-conjugacy class is, the difference between $\hat j$ and $\sigma \circ \hat j \circ \sigma^{-1}$ is realized by an element $\sigma_{S,G} \in \Omega$ for any $\sigma \in \Gamma$. Thus the transport of the $\Gamma$-action on $\hat S$ via $\hat j$ is given by the group homomorphism $\Gamma \to \Omega \rtimes \Gamma$ sending $\sigma$ to $\sigma_{S,G} \rtimes \sigma$. Pulling back along this homomorphism the extension with splitting

\[ \xymatrix{
1\ar[r]&\hat T\ar[r]&N(\hat T,\hat G)\rtimes\Gamma\ar[r]&\Omega \rtimes\Gamma\ar[r]\ar@/_1pc/[l]_n&1\\
}
\]
and using $\hat j$ to identify $\hat S$ with $\hat T$ we obtain an extension of $\Gamma$ by $\hat S$ with a splitting. According to \cite[Lemma 2.1.A]{LS87} the 2-cocycle associated to this splitting is the Tits cocycle $t_p$ of Definition \ref{dfn:tp}, where $p$ is the gauge that assigns $+1$ to the $\hat B$-positive roots, and $R \subset X^*(S)$ is the root system of $G$. In other words, we have just extended $\hat j : \hat S \to \hat T$ to an $L$-embedding $^LS_\pm \to {^LG}$ taking image in $N(\hat T,\hat G) \rtimes \Gamma$. We record explicitly that it is given by 
\begin{equation} \label{eq:lembt}
\hat S \boxtimes_{t_p} \Gamma \to \hat G \rtimes \Gamma, s \boxtimes \sigma \mapsto \hat j(s) \cdot n(\sigma_{S,G}) \rtimes \sigma.
\end{equation}
The only choice involved in the construction of this $L$-embedding $^LS_\pm \to {^LG}$ was the $\Gamma$-invariant pinning. Since any two such are conjugate under $\hat G$, in fact even under $\hat G^\Gamma$ \cite[Corollary 1.7]{Kot84}, it is immediate to check that the two resulting $L$-embeddings are also conjugate.

\subsection{Factorization of Langlands parameters} \label{sub:par}

Let $S$ be an algebraic torus defined over $F$ and equipped with a stable class of embeddings $S \to G$ defined over $F$ whose images are maximal tori. This endows $S$ with various structures \cite[\S5.1]{KalRSP}, in particular a finite $\Sigma$-invariant subset $0 \notin R \subset X^*(S)$. Therefore we have the double cover $S(F)_\pm$. 

Consider two data consisting of an algebraic torus $S_i$, a stable class of embeddings $S_i \to G$ as above, and a genuine character $\theta_i$ of $S_i(F)_\pm$. We say that they are equivalent if there exists an isomorphism $S_1 \to S_2$ that identifies the two stable classes of embeddings and whose natural lift $S_1(F)_\pm \to S_2(F)_\pm$ identifies the two genuine characters.

\begin{pro} \label{pro:parfac}
There is a canonical bijection between the set of equivalence classes of Langlands parameters $W \to {^LG}$ whose image normalizes a maximal torus, and the set of equivalence classes of data consising of an algebraic torus $S$ defined over $F$, a stable class of embeddings $S \to G$ defined over $F$ whose images are maximal tori, and a genuine character $\theta$ of $S(F)_\pm$. Given such a datum, the bijection sends it to the composition of the Langlands parameter of $\theta$ with the canonical embedding $^LS_\pm \to {^LG}$ of \S\ref{sub:canemb}.
\end{pro}

\begin{rem}
It is well known that discrete series parameters when $F=\R$ have the above property, and it is proved in \cite[\S5.2]{KalRSP} that this is also the case for regular supercuspidal parameters. 	
\end{rem}

\begin{proof}
Such constructions have been explained in detail in various places. In the real case we refer to \cite[\S3.4]{She83} for the statement without covers and to \cite[\S6]{AV16} for the statment with covers. In the $p$-adic case we refer to \cite[\S5.2]{KalRSP} for the statement without covers. We limit ourselves to a brief sketch here, in order to illustrate the simplicity of the situation once the double cover has been introduced and $\chi$-data is no longer required. 

Let $\varphi : W \to {^LG}$ be a Langlands parameter whose image is contained in the normalizer of a maximal torus. Let $\hat S$ denote this torus together with the action of $W$ by conjugation by $\varphi$. It is immediate that this action extends to $\Gamma$. Let $S$ be the torus over $F$ with dual $\hat S$. Then $S$ comes equipped with a stable class of embeddings $S \to G$ \cite[\S5.1]{KalRSP}. We have the canonical $\hat G$-conjugacy class of $L$-embeddings $^LS_\pm \to {^LG}$ of \S\ref{sub:canemb}. Since $\hat S$ is already (tautologically) embedded into $\hat G$, inside of this $\hat G$-conjugacy class there is a unique $\hat S$-conjugacy class of $L$-embeddings $^LS_\pm \to {^LG}$ that extends the tautological embedding $\hat S \to \hat G$. The parameter $\varphi$ factors through this embedding and leads in this way to a canonical parameter $\varphi_\pm$ for $S(F)_\pm$, hence to a canonical genuine character $\theta_\pm$ of $S(F)_\pm$ via Theorem \ref{thm:lldc}.
\end{proof}

\subsection{The character formula associated to a genuine character} \label{sub:char}

In this section we assume that in the case of a non-archimedean base field $F$ the residual characteristic is not $2$.

Let $S \subset G$ be an elliptic maximal torus and $\theta : S(F)_\pm \to \C^\times$ a genuine character. Fix a non-trivial character $\Lambda : F \to \C^\times$. Then $\theta$ and $\Lambda$ determine a set of $a$-data in the archimedean case, and a set of mod-$a$-data in the non-archimedean case, that may happen to be degenerate (i.e. $a_\alpha=0$), as follows. The character $\theta^2$ descends to $S(F)$. In the archimedean case we define $a_\alpha$ by the equation
\[ \theta^2(N_{F_\alpha/F}(\alpha^\vee(\exp(z)))) = \Lambda(\tx{tr}_{F_\alpha/F}(2a_\alpha z)), \]
where $z$ is a variable  in $F_\alpha$. In the non-archimedean case we use the equation
\[ \theta^2(N_{F_\alpha/F}(\alpha^\vee(z+1))) = \Lambda(\tx{tr}_{F_\alpha/F}(2\bar a_\alpha z)), \]
where $z$ is a variable in $[F_\alpha]_{r_{\theta,\alpha}}/[F_\alpha]_{r_{\theta,\alpha}+}$ and $r_{\theta,\alpha}$ is the largest real number for which the left hand side is non-zero for such $z$. Let $r_{\Lambda,\alpha}$ be the depth of $\Lambda \circ \tx{tr}_{F_\alpha/F}$ and let $r_\alpha=r_{\Lambda,\alpha}-r_{\theta,\alpha}$.

It is quite obvious that $a_{-\alpha}=-a_\alpha$ and $a_{\sigma\alpha}=\sigma(a_\alpha)$ for $\alpha \in R$ and $\sigma \in \Gamma$. But it may well happen that $a_\alpha=0$. We'll be particularly interested in characters $\theta$ for which this does not happen. For $F=\R$ one checks easily that $a_\alpha = \<\alpha^\vee,d\theta\>/(2\pi i)$, where $d\theta \in [X^*(S) \otimes \C]^\Gamma$ is the differential of $\theta$. Thus $a_\alpha \neq 0$ for all $\alpha$ is equivalent to $d\theta$ being regular. This is the case for $\theta$ that come, via Proposition \ref{pro:parfac}, from a discrete series parameter. In that case, let $(\chi_\alpha)$ be the based $\chi$-data specified by the Weyl chamber determined by $d\theta$. Then $\theta \cdot \chi_S^{-1}$ is the character denoted by $\theta$ in \cite[\S4.11]{KalRSP} and the $a$-data computed here coincides with the one computed there.

For a non-archimedean $F$ we will only apply this construction when the depth of $\theta^2 \circ N_{F_\alpha/F} \circ \alpha^\vee$ is positive, which guarantees $a_\alpha \neq 0$. For the remaining $\alpha$ we shall take $a_\alpha$ to be units, chosen arbitrarily. When $\theta$ comes, via Proposition \ref{pro:parfac}, from a regular supercuspidal parameter, let $(\chi_\alpha)$ be arbitrary minimally ramified $\chi$-data \cite[Definition 4.6.1]{KalRSP}. Then $\theta \cdot \chi_S^{-1}$ satisfies \cite[Definition 3.7.5]{KalRSP} and the $a$-data computed here coincides with the one computed from $\theta\cdot\chi_S^{-1}$ by \cite[(4.10.1)]{KalRSP}. Indeed, it is easy to see that the double cover $S(F)_\pm$ splits canonically over $S(F)_{0+}$, and $\chi_S^{-1}$ restricts trivially to $S(F)_{0+}$. This relies on our assumption $p \neq 2$.

This set of $a$-data, respectively mod-$a$-data, leads to the genuine function $a_S$ of \eqref{eq:adata}. The product $a_S \cdot \theta$ is a function on $S(F)_\pm$ that descends to $S(F)$. Let $T$ be the minimal Levi subgroup in the quasi-split inner form of $G$. Consider the function $S(F) \to \C$ sending $s \in S(F)$ to
\begin{equation} \label{eq:char}
 e(G)\epsilon(X^*(T)_\C-X^*(S)_\C,\Lambda) \sum_{w \in N(S,G)(F)/S(F)} [a_S \cdot \theta](ws).
\end{equation}

This function does not depend on the choice of $\Lambda$. When $F=\R$ there is a unique irreducible discrete series representation of $G(F)$ whose Harish-Chandra character restricted to $S(F)_\tx{reg}$ is given by this function \cite[\S4.11]{KalRSP}. When $F$ is non-archimedean, $G$ splits over a tame extension of $F$, and $p$ is not a bad prime for $G$, there exists a natural regular supercuspidal representation whose Harish-Chandra character restricted to the set of regular topologically semi-simple elements of $S(F)$ is given by this function. This representation is obtained by letting $(\chi_\alpha)$ be the $\chi$-data obtained from $(a_\alpha)$ by \cite[(4.7.2)]{KalRSP} and applying \cite[Proposition 3.7.14]{KalRSP} to the character $\theta\cdot\chi_S^{-1}\cdot \epsilon_{f,\tx{ram}} \cdot \epsilon^\tx{ram}$ and the torus $S$. Here $\epsilon_{f,\tx{ram}}$ and $\epsilon^\tx{ram}$ are defined in \cite[Definition 4.7.3 and Equation (4.3.3)]{KalRSP}. We have had to replace ``unique'' by ``natural'' as there may not be sufficiently many such elements in order to determine the representation uniquely.

\subsection{The stable character of a supercuspidal Langlands parameter} \label{sub:stabchar}

Let $F$ be non-archimedean. Assume that $G$ splits over a tame extension of $F$. Assume further that $p$ is sufficiently large so that it does not divide the order of the Weyl group and the exponential function $\tx{Lie}(G(F)) \to G(F)$ converges on the set of topologically nilpotent elements. Fix a non-trivial additive character $\Lambda : F \to \C^\times$.

We will now combine the discussions of \S\S\ref{sub:canemb},\ref{sub:par},\ref{sub:char} to obtain a spectral characterization of the local Langlands correspondence for supercuspidal parameters, i.e. discrete Langlands parameters $\varphi : W_F \to {^LG}$. For this, we will present a formula for the stable character $S\Theta_\varphi$ of the $L$-packet associated to $\varphi$. The theory of endoscopy produces from this stable character, and the stable characters on all endoscopic groups, the characters of the individual members of the $L$-packet, and in this way determines the $L$-packet completely.

It is shown in \cite[Lemma 4.1.3]{KalSLP} and \cite[Lemma 5.2.2]{KalRSP} that for a supercuspidal parameter $\varphi$, $\hat S = \tx{Cent}(\tx{Cent}(\varphi(I_F),\hat G)^\circ,\hat G)$ is a maximal torus of $\hat G$. It is clearly normalized by $\varphi$. Proposition \ref{pro:parfac} produces the pair $(S,\theta)$ of a torus defined over $F$ and equipped with a stable class of embeddings $S \to G$ whose images are elliptic maximal tori, and a genuine character $\theta$ of $S(F)_\pm$.

The prime $p$ is odd, which implies that the double cover $S(F)_\pm$ splits canonically over the pro-p-Sylow subgroup $S(F)_{0+} \subset S(F)$. The restriction of $\theta$ to $S(F)_{0+}$ is given by $\Lambda\<X^*,\log(-)\>$ for an element $X^* \in \tx{Lie}^*(S)(F)$. We also obtain $a$-data as in \S\ref{sub:char}.

\begin{cnj} \label{cnj:stabchar}
For any strongly regular semi-simple $\gamma \in G(F)$ with topological Jordan decomposition $\gamma=\gamma_0 \cdot \gamma_{>0}$,
\[ S\Theta_\varphi^G(\gamma) = e(J)\epsilon(T_G-T_J)\sum_{j : S \to J}[a_S \cdot \theta](\gamma_0) \cdot \widehat{SO}^J_{^jX^*}(\log(\gamma_{>0})). \]
Here $J$ is the connected centralizer of $\gamma_0$ in $G$, $e(J)$ is the Kottwitz sign of $J$, $\epsilon(T_G-T_J)$ is the root number of the virtual Galois representation $X^*(T_G)_\C-X^*(T_J)_\C$ for $T_G$ and $T_J$ the minimal Levi subgroups in the quasi-split inner forms of $G$ and $J$, and the sum runs over the set of stable classes of admissible embeddings $S \to J$.
\end{cnj}

This conjecture, together with the endoscopic character identities of \cite[\S5.4]{KalRI}, characterize completely the local Langlands correspondence for supercuspidal parameters. Indeed, given a supercuspidal parameter $\varphi$, this conjecture characterizes the stable characters of the associated $L$-packet for $G$, as well as all endoscopic groups of $G$ through which $\varphi$ factors. The endoscopic character identities then give a formula for the Harish-Chandra character of each constituent of the $L$-packet for $G$.

This conjecture is proved in \cite[\S4.4]{FKS} for the construction of \cite{KalSLP}. We have stated it here as a conjecture, because the notion of ``the'' local Langlands correspondence is somewhat ambiguous.

\section{Functoriality of the local correspondence} \label{sec:func}

Consider two tori $S$ and $T$ defined over $F$, two finite admissible $\Sigma$-sets $R_S$ and $R_T$, and two $\Sigma$-equivariant maps $R_S \to X^*(S)$ and $R_T \to X^*(T)$. Let $f : S \to T$ be a morphism defined over $F$ and let $f^* : X^*(T) \to X^*(S)$ be the map induced from $f$. We would like to study what data induces a lift of $f$ to a homomorphism $S(F)_\pm \to T(F)_\pm$ of covers, as well as an extension of $\hat f : \hat T \to \hat S$ to an $L$-homomorphism $^LT_\pm \to {^LS_\pm}$, so that the local correspondences for the two double covers match. The most naive guess would be a map $R_T \to R_S$ commuting with $f^*$, but it turns out that this is not correct, and the right condition is slightly different. More precisely, we will see in Example \ref{exa:non-surj} that, when the map $R_T \to R_S$ is not surjective, functoriality cannot hold in general. It will turn out that requiring $R_T \to R_S$ to be surjective, as well as satisfy a compatibility condition with respect to $f^*$, which is \emph{not} the commutativity of the obvious square, is enough to produce the desired map between covers and their $L$-groups. This will be discussed in \S\ref{sub:func}.

\subsection{Initial observations} \label{sub:initial}

\begin{dfn}
A \emph{lift} of $f : S \to T$ is a homomorphism $S(F)_\pm \to T(F)_\pm$ whose restriction to $\{\pm1\}$ is the identity, and whose composition with the natural projection $T(F)_\pm \to T(F)$ equals the composition of the natural projection $S(F)_\pm \to S(F)$ with $f$.
\end{dfn}

\begin{fct} \label{fct:pullbackfunc}
\begin{enumerate}
	\item For any lift $f_\pm : S(F)_\pm \to T(F)_\pm$ of $f$, the commutative diagram
	\[ \xymatrix{
	S(F)_\pm\ar[d]\ar[r]^{f_\pm}&T(F)_\pm\ar[d]\\
	S(F)\ar[r]^f&T(F)
	}
	\]
	is Cartesian.
	\item The set of lifts of $f$ is a torsor for the abelian group of sign characters $S(F) \to \{\pm 1\}$.
\end{enumerate}
\end{fct}

We now introduce a basic construction that reduces the question of functoriality to the case of a single torus.

\begin{cns} \label{cns:func_t}
Assume that $R_T = R_S$ and denote this set by $R$. Assume further that the composition of the given $\Sigma$-equivariant map $R \to X^*(T)$ with $f^*$ equals the given $\Sigma$-equivariant map $R \to X^*(S)$.

For every $\Sigma$-orbit $O \subset R$ we have the torus $J_O$ and the double cover $J_O(F)_\pm$ of $J_O(F)$. The diagram $S(F) \to \prod J_O(F) \from \prod J_O(F)_\pm$ maps the the diagram $T(F) \to \prod J_O(F) \from \prod J_O(F)_\pm$, and hence the pull-back of the first diagram maps to the pull-back of the second. In this way we obtain a lift $S(F)_\pm \to T(F)_\pm$ of $f : S(F) \to T(F)$.

Dually, the map $\hat f : \hat T \to \hat S$ identifies the two Tits cocycles and hence extends canonically to an $L$-homomorphism $^LT_\pm \to {^LS_\pm}$. The fact that the local correspondences for $S(F)_\pm$ and $T(F)_\pm$ are compatible with these maps is immediate and left to the reader.
\end{cns}

Note that the flexibility of considering a map $R \to X^*(S)$ rather than a subset $R \subset X^*(S)$ is useful here: even if $R$ is a subset of $X^*(T)$, the restriction of $f^*$ to $R$ may not be injective.

This construction reduces the question of functoriality to the case when one considers a single torus $S$ and two finite admissible $\Sigma$-sets $R$ and $R'$ with $\Sigma$-equivariant maps $R \to X^*(S)$ and $R' \to X^*(S)$. We will now give two examples showing that having a map $R \to R'$ that makes the obvious triangle commute does not imply that the identity of $S(F)$ lifts to a homomorphism $S(F)_R \to S(F)_{R'}$.

\begin{exa} \label{exa:non-surj}
Let $S=\tx{Res}^1_{E/F}\mb{G}_m$ for a separable quadratic extension $E/F$. Then $X^*(S)=\Z_{(-1)}$. Let $R=\varnothing \subset X^*(S)$ and $R'=\{1,-1\} \subset X^*(S)$. The tautological injective map $R \to R'$ commutes with $R \to X^*(S)$ and $R' \to X^*(S)$. But $S(F)_R$ is the split extension $S(F) \times \{\pm1\}$, while $S(F)_{R'}$ may be non-split, cf. Remark \ref{rem:sosplit}. Therefore, there is no lift $S(F)_R \to S(F)_{R'}$ of the identity of $S(F)$.

The reader may object that $R=\varnothing$ is not a valid choice of a finite admissible $\Sigma$-set. One can easily take a product of the above example with, say $\mb{G}_m$, and add to $R$ an element of $\Z=X^*(\mb{G}_m)$.
\end{exa}

\begin{exa} Let $S=\tx{Res}_{E/F}^1\mb{G}_m$, $R'=\{1,-1\} \subset \Z_{(-1)}=X^*(S)$, $R=\{1_a,1_b,-1_a,-1_b\}$ and $R \to R'$ is the surjective map that sends $1_a,1_b$ to $1$ and $-1_a,-1_b$ to $-1$. The action of $\Sigma$ on $R$ descends to $\Gamma_{E/F} \times \{\pm 1\}$ and both the non-trivial element of $\Gamma_{E/F}$ and $-1 \in \{\pm 1\}$ send $1_a \mapsto -1_a$ and $1_b \mapsto -1_b$. As in the previous example, we know that $S(F)_{R'}$ may be non-split. On the other hand, we claim that $S(F)_R$ is canonically split. Indeed, this extension of $S(F)=E^1$ is obtained by taking the product of two copies of the extension 
\[ 1 \to F^\times/N_{E/F}(E^\times) \to E^\times/N_{E/F}(E^\times) \to E^1 \to 1 \]
pushing out via the product map $(F^\times/N_{E/F}(E^\times))^2=(\{\pm 1\})^2 \to \{\pm 1\}$, and then pulling back along the diagonal inclusion $E^1 \to (E^1)^2$. Therefore, there is no lift $S(F)_R \to S(F)_{R'}$ of the identity of $S(F)$.

The reader may now object that this pathological behavior comes from us allowing maps $R \to X^*(S)$ that are not injective. This is not the case. Indeed, we can consider a second torus $T=S \times S$ and let $f : S \to T$ be the the diagonal embedding. Thus $f^* : X^*(T)=\Z_{(-1)} \oplus \Z_{(-1)} \to \Z_{(-1)}=X^*(S)$ is the addition map. If we let $R_T=\{(1,0),(0,1),(-1,0),(0,-1)\} \subset X^*(T)$, then we now have a version of this example where $R_T$ is a subset of $X^*(T)$ and $R_S=R'$ is a subset of $X^*(S)$. This version reduces to the previous version by Construction \ref{cns:func_t}.
\end{exa}

We now consider a second special case of functoriality. It will be used in \S\ref{sub:av} to compare our double cover construction with that of Adams-Vogan. It will be then substantially generalized in \S\ref{sub:func}.

\begin{cns} \label{cns:simple}
Consider a single torus $S$ and a finite $\Sigma$-invariant subset $0 \notin R \subset X^*(S)$. Let $\alpha,\beta \in R_\tx{sym}$ be such that $\Sigma\cdot\alpha \neq \Sigma\cdot\beta$, $\Gamma_\alpha=\Gamma_\beta$, and $\Gamma_{\pm\alpha}=\Gamma_{\pm\beta}$. Set $\gamma=\alpha+\beta$. Let $R'=\Sigma\cdot\gamma\cup R \sm (\Sigma\cdot\alpha \cup \Sigma \cdot\beta)$. We will construct canonical isomorphisms $S(F)_{R} \to S(F)_{R'}$ and $^LS_{R'} \to {^LS_{R}}$.

We have $J_\alpha=J_\beta=J_\gamma$ canonically. The multiplication map $J_\alpha \times J_\beta \to J_\gamma$ fits into the commutative diagram
\[ \xymatrix{
	[J_\alpha(F)_{\pm} \times J_\beta(F)_\pm]/\{\pm 1\}\ar[r]\ar[d]&J_\gamma(F)_\pm\ar[d]\\
	J_\alpha(F) \times J_\beta(F)\ar[r]&J_\gamma(F)
}
\]
and this diagram is cartesian. The transitivity of pull-back now gives the isomorphism $S(F)_{R} \to S(F)_{R'}$.	

On the dual side, let $p$ be a gauge on $R$ s.t. $p(\sigma\alpha)=p(\sigma\beta)$ for all $\sigma\in\Sigma$. Define a gauge $p'$ on $R'$ by setting $p'(\sigma\gamma)=p(\sigma\alpha)=p(\sigma\beta)$ for all $\sigma \in \Sigma$ and $p'(\delta)=p(\delta)$ for $\delta \in R' \sm \Sigma\cdot\gamma$. Then the Tits cocycle $t_p$ with respect to $R$ is equal to the Tits cocycle $t_{p'}$ with respect to $R'$, so the $L$-groups $^LS_{R}$ and $^LS_{R'}$ are canonically identified.

Finally we observe that the identifications $S(F)_{R}=S(F)_{R'}$ and $^LS_{R}={^LS_{R'}}$ are compatible with the Langlands correspondence. For this, take a set of $\chi$-data for $R'$. Define $\chi_\alpha=\chi_\beta=\chi_\gamma$ and obtain in this way a set of $\chi$-data for $R$. The characters of $S(F)_{R}$ and $S(F)_{R'}$ obtained from these sets of $\chi$-data are readily seen to be equal under the identification $S(F)_{R}=S(F)_{R'}$. The same holds for the 1-cochains $r_p$.
\end{cns}

\subsection{Comparison with the $\rho_i$-cover of Adams-Vogan} \label{sub:av}

Let $F=\R$, $S$ an algebraic torus over $F$, and $\gamma \in \frac{1}{2}X^*(S)$. Adams and Vogan \cite[\S3]{AV16} define the diagonalizable group $S_\gamma=\{(s,z) \in S \times \mb{G}_m|[2\gamma](s)=z^2\}$. Thus $X^*(S_\gamma)$ is the push-out of 
\[ \Z \stackrel{2}{\llw} \Z \stackrel{2\gamma}{\lrw} X^*(S). \]
The natural map $X^*(S) \to X^*(S_\gamma)$ is injective and the quotient is $\Z/2\Z$, so we have the exact sequence
\[ 1 \to \{\pm 1\} \to S_\gamma \to S \to 1. \]
They define $S(\R)_\gamma$ to be the preimage in $S_\gamma(\C)$ of $S(\R)$. Thus $S(\R)_\gamma$ is a double cover of $S(\R)$. Let $\gamma \in X^*(S_\gamma)$ by the morphism $S_\gamma \to \mb{G}_m$ given by $\gamma(s,z)=z$. Then $\gamma$ is a geniune character of $S_\gamma$ whose square is $2\gamma \in X^*(S)$.

The extension $S_\gamma$ splits if and only if $X^*(S_\gamma)$ has torsion, which is the case if and only if $\gamma \in X^*(S)$. If on the other hand $\gamma \in \frac{1}{2}X^*(S) \sm X^*(S)$, then $S_\gamma$ is a torus, its character module is the submodule of $\frac{1}{2}X^*(S)$ generated by $X^*(S)$ and $\gamma$, and the map $S_\gamma \to S$ is a degree $2$ isogeny.

Adams and Vogan further define an $L$-group for $S(\R)_\gamma$. It is the group generated by $\hat S$ and an element $\hat\delta$ acting on $\hat S$ as the non-trivial element of $\Gamma_{\C/\R}$ and satisfying $\hat\delta^2 = \exp(2\pi i\gamma)$, where $\exp : X_*(\hat S) \otimes \C \to \hat S$. They prove a local Langlands correspondence for their double cover \cite[Lemma 3.3]{AV16}. In fact, these results were first proved in \cite{AV92}, but \cite{AV16} is sometimes a more convenient reference.

\begin{lem} \label{lem:avcmp}
Assume $\gamma$ is symmetric (i.e. imaginary) and let $R=\{2\gamma,-2\gamma\}$. There are natural isomorphisms $S(\R)_\gamma \to S(\R)_{\pm}$ and $\<\hat S,\hat\delta\> \to {^LS_\pm}$. These isomorphisms identify the local correspondence of Theorem \ref{thm:lldc} with the local correspondence of Adams-Vogan.
\end{lem}
\begin{proof}
The construction of $S(\R)_\gamma$ can be rephrased as the pull-back diagram
\[ \xymatrix{
	S(\R)_\gamma\ar[r]\ar[d]&S(\R)\ar[d]^{2\gamma}\\
	\C^\times\ar[r]^{(-)^2}&\C^\times
}\]
Since $2\gamma$ is imaginary the image of $S(\R)$ lies in $\mb{S}^1 \subset \C^\times$. The preimage of $\mb{S}^1$ under the squaring map is again $\mb{S}^1$. Therefore we may replace the bottom row in the above diagram by the squaring map $\mb{S}^1 \to \mb{S}^1$. The inclusion $\mb{S}^1 \to \C^\times$ induces an isomorphism $\mb{S}^1 \to \C^\times/\R_{>0}$. Composing its inverse with the squaring map on $\mb{S}^1$ gives the map $\C^\times/\R_{>0} \to \mb{S}^1$ sending $z$ to $z/\bar z$. Again we replace the bottom row of above diagram, now with the map $\C^\times/\R_{>0} \to \mb{S}^1$ sending $z$ to $z/\bar z$. The corresponding pull-back diagram is the definition of $S(\R)_\pm$.

Turning to the dual side we compute the Tits cocycle for $R$. Define a gauge by $p(2\gamma)=+1$. The Tits cocycle $t_p$ sends $(\sigma,\sigma)$ to $(-1)^{2\gamma}=\exp(2\pi i\gamma)$. Therefore the map $\<\hat S,\hat\delta\> \to \hat S \boxtimes_{t_p}\Gamma_{E/F}$ that is the identity on $\hat S$ and sends $\hat\delta$ to $1 \boxtimes \sigma$ is an isomorphism.

It is enough to compare the two versions of the local correspondence on one genuine character of $S(\R)_\pm$ since both are compatible with multiplication by non-genuine characters. We take the one sending $(s,\delta)$ with $\delta/\bar\delta = [2\gamma](s)$ to $\tx{sgn}(\delta)$, where $\tx{sgn} : \C \to \mb{S}^1$ is  the argument function. In terms of the Adams-Vogan cover $S(\R)_\gamma$ this is the canonical genuine character $\gamma$. Present the Weil group $W_{\C/\R}$ as the group generated by $\C^\times$ and $j$ with $jzj^{-1}=\bar z$ and $j^2=-1$. The $L$-parameter of $\gamma$ associated by \cite[Lemma 3.3]{AV16} sends $z$ to $(z/\bar z)^\gamma$ and $j$ to $\hat\delta$. The $L$\-parameter associated by Theorem \ref{thm:lldc} is $r_p(z)=[2\gamma](\tx{sgn}(z))$ and $r_p(zj)=2[\gamma](\tx{sgn}(z))$. These are manifestly the same parameter.
\end{proof}

The construction of Adams-Vogan can be applied in particular when $S$ is a maximal torus of a connected reductive group $G$ and $\gamma$ is either half the sum of all positive roots (with respect to some fixed order), or half the sum of the positive imaginary roots, denoted by $\rho$ and $\rho_i$, respectively.

\begin{cor} Let $R \subset X^*(S)$ be the root system of $G$. There is are natural isomorphisms $S(\R)_\pm \to S(\R)_{\rho_i}$ between the two covers, as well as between their $L$-groups, that identify the two versions of the local correspondence.
\end{cor}
\begin{proof}
We may discard all asymmetric elements from $R$, since they do not influence the cover $S(\R)_\pm$. All symmetric elements of $R \subset X^*(S)$ share the same $F_\alpha=\C$ and $F_{\pm\alpha}=\R$, which allows us to apply Construction \ref{cns:simple} inductively and replace $R$ with the sum of representatives for its $\Gamma$-orbits. We choose these representatives to be the positive (imaginary) roots with respect to a given order. Their sum is then $2\rho_i$. The claim follows from Lemma \ref{lem:avcmp}.
\end{proof}
 
\begin{rem} For the comparison of the Adams-Vogan cover and the cover defined in this paper the two essential points were the fact that the inclusion $\C^1 = \mb{S}^1 \to \C^\times/\R_{>0}=\C^\times/N_{\C/\R}(\C^\times)$ is an isomorphism and that all symmetric elements $\alpha \in X^*(S)$ share the same $F_\alpha=\C$ and $F_{\pm\alpha}=\R$. Neither of these statements is true in the $p$-adic case, which explains why a direct translation of their construction to the $p$-adic case was not possible.
\end{rem}

\subsection{Transport of $\chi$-data} \label{sub:chi-trans}

We return to the case of an arbitrary local field $F$. In the next subsection we will discuss a more general case of functoriality for the double cover $S(F)_\pm$ of a torus determined by a $\Sigma$-stable finite set $R \subset X^*(S)$. As a preparation, we will discuss in this section some basic results about $\chi$-data. For this, we consider two admissible sets $R$ and $R'$ with $\Sigma$-action and a surjective $\Sigma$-equivariant map $f : R \to R'$.

\begin{fct} \label{fct:reps}
Let $\alpha' \in R'$. If $\alpha_1,\dots,\alpha_k$ is a set of representatives for the $\Gamma_{\alpha'}$-orbits in the preimage of $\alpha'$, then it is also a set of representatives for the $\Gamma$-orbits in the preimage of $\Gamma \cdot \alpha'$.
\end{fct}

\begin{lem} \label{lem:symres}
Let $\alpha \in R_\tx{sym}$ and let $\alpha'=f(\alpha)$.
\begin{enumerate}
\item We have $\alpha' \in R'_\tx{sym}$. The diagram
\[ \xymatrix{
\Gamma_\alpha\ar@{^(->}[d]\ar@{^(->}[r]&\Gamma_{\alpha'}\ar@{^(->}[d]\\
\Gamma_{\pm\alpha}\ar@{^(->}[r]&\Gamma_{\pm\alpha'}
}
\]
is both cartesian and cocartesian. In other words, $\Gamma_\alpha = \Gamma_{\pm\alpha} \cap \Gamma_{\alpha'}$ and $\Gamma_{\pm\alpha'}=\Gamma_{\pm\alpha} \cdot \Gamma_{\alpha'}$.
\item Restricting the quadratic character $\Gamma_{\pm\alpha'}/\Gamma_{\alpha'} \to \{\pm 1\}$ to $\Gamma_{\pm\alpha}$ gives the quadratic character $\Gamma_{\pm\alpha}/\Gamma_\alpha \to \{\pm 1\}$.
\item The following diagram commutes
\[ \xymatrix{
\Gamma_\alpha^\tx{ab}\ar[r]&\Gamma_{\alpha'}^\tx{ab}\\
\Gamma_{\pm\alpha}^\tx{ab}\ar[r]\ar[u]&\Gamma_{\pm\alpha'}^\tx{ab}\ar[u]
}
\]
where the horizontal maps are induced by the inclusions, and
the vertical maps are the transfer homomorphisms.
\end{enumerate}	
\end{lem}
\begin{proof}
That $\alpha' \in R'_\tx{sym}$ is immediate. Consider the equalities $\Gamma_\alpha = \Gamma_{\pm\alpha} \cap \Gamma_{\alpha'}$ and $\Gamma_{\pm\alpha'}=\Gamma_{\pm\alpha} \cdot \Gamma_{\alpha'}$. For the second, notice that since $\Gamma_{\alpha'}$ is normal in $\Gamma_{\pm\alpha'}$ and $\Gamma_{\pm\alpha} \subset \Gamma_{\pm\alpha'}$, the product is a subgroup. The two equations are equivalent to the injectivity and surjectivity, respectively, of the natural map $\Gamma_{\pm\alpha}/\Gamma_\alpha \to \Gamma_{\pm\alpha'}/\Gamma_{\alpha'}$. These two quotients being of order two, the two statements are equivalent. For the first equality, the inclusion $\Gamma_\alpha \subset \Gamma_{\pm\alpha} \cap \Gamma_{\alpha'}$ is trivial, while the opposite inclusion follows from the admissibility of $R'$.

The second point is now immediate.

For the second diagram we choose $\tau \in \Gamma_{\pm\alpha} \sm \Gamma_\alpha$ and recall that the map $t : \Gamma_{\pm\alpha} \to \Gamma_\alpha$ defined by $t(\gamma)=\gamma \cdot \tau\gamma\tau^{-1}$ and $t(\gamma\tau)=\gamma \cdot \tau\gamma\tau^{-1} \cdot \tau^2$ for $\gamma \in \Gamma_\alpha$ describes the transfer homomorphism. The commutativity follows from the fact that the image of $\tau$ in $\Gamma_{\pm\alpha'}$ avoids $\Gamma_{\alpha'}$, which is implied by the first part of this lemma.
\end{proof}

\begin{cor} \label{cor:symres}
Let $\alpha \in R_\tx{sym}$ and let $\alpha'=f(\alpha)$. Then
\begin{enumerate}
	\item $F_\alpha = F_{\pm\alpha} \cdot F_{\alpha'}$ and $F_{\pm\alpha'}=F_{\pm\alpha} \cap F_{\alpha'}$.
	\item Restriction to $F_{\alpha'}$ provides an isomorphism 
	\[ \tx{Gal}(F_\alpha/F_{\pm\alpha}) \to \tx{Gal}(F_{\alpha'}/F_{\pm\alpha'}). \]
	\item The restriction of $N_{F_\alpha/F_{\alpha'}}$ to $F_{\pm\alpha}$ equals $N_{F_{\pm\alpha}/F_{\pm\alpha'}}$.
	\item $\kappa_{\alpha'}\circ N_{F_{\pm\alpha}/F_{\pm\alpha'}}=\kappa_\alpha$.
\end{enumerate}
\end{cor}
\begin{proof}
The composite $F_{\pm\alpha} \cdot F_{\alpha'}$ is the fixed field of $\Gamma_{\pm\alpha} \cap \Gamma_{\alpha'}$, which by Lemma \ref{lem:symres} equals $\Gamma_\alpha$, hence the first equality. The second follows from the fact that $F_{\pm\alpha} \cap F_{\alpha'}$ is the subfield of $F_{\alpha'}$ fixed by the group $\Gamma_{\pm\alpha}$, which according to Lemma \ref{lem:symres} generates $\Gamma_{\pm\alpha'}$ together with $\Gamma_{\alpha'}$. The second and third points are immediate. The fourth point follows again from Lemma \ref{lem:symres} and the fact that the Artin reciprocity map translates inclusion on the Galois side to the norm on the local field side.
\end{proof}

\begin{cor} \label{cor:chi-norm}
Let $\chi'=(\chi_{\alpha'})_{\alpha'}$ be a set of $\chi$-data for $R'$. For each $\alpha \in R$ define $\chi_\alpha = \chi_{\alpha'} \circ N_{F_\alpha/F_{\alpha'}}$, where $\alpha'=f(\alpha)$. Then $\chi=(\chi_\alpha)_\alpha$ is a set of $\chi$-data for $R$. The analogous statement holds for $\zeta$-data.
\end{cor}
\begin{proof}
The $\Sigma$-equivariance of $f$ implies $\chi_{-\alpha}=\chi_\alpha^{-1}$ and $\chi_{\sigma\alpha}=\chi_\alpha\circ\sigma^{-1}$. Corollary \ref{cor:symres} implies
\[ \chi_\alpha|_{F_{\pm\alpha}^\times} = \chi_{\alpha'}\circ N_{F_\alpha/F_{\alpha'}}|_{F_{\pm\alpha}^\times} = \chi_{\alpha'}|_{F_{\pm\alpha'}^\times} \circ N_{F_{\pm\alpha}/F_{\pm\alpha'}} = \kappa_{\alpha'}\circ N_{F_{\pm\alpha}/F_{\pm\alpha'}} = \kappa_\alpha. \]
\end{proof}

\begin{dfn} \label{dfn:inf}
Given $\chi$-data $\chi'=(\chi_{\alpha'})$ for $R'$, we denote by $\tx{inf}\,\chi'$ the $\chi$-data $(\chi_{f(\alpha)} \circ N_{F_\alpha/F_{f(\alpha)}})_\alpha$ for $R$. Analogously we define $\tx{inf}\,\zeta'$ for a collection of $\zeta$-data $\zeta'$ for $R'$.
\end{dfn}

\subsection{A general case of functoriality} \label{sub:func}

We continue with an arbitrary local field $F$. Consider two tori $S$ and $T$ defined over $F$, two finite admissible $\Sigma$-sets $R_S$ and $R_T$, and two $\Sigma$-equivariant maps $R_S \to X^*(S)$ and $R_T \to X^*(S)$. Let $f : S \to T$ be a map of tori and $f^* : X^*(T) \to X^*(S)$ its dual. Let $f : R_T \to R_S$ be a $\Sigma$-equivariant surjective map. We do \emph{not} assume that the obvious square commutes! Rather, we make the following assumption: for any $\alpha_S \in R_S$
\begin{equation} \label{eq:func_st}
\bar\alpha_S = \sum_{\substack{\alpha_T \in R_T\\ f(\alpha_T)=\alpha_S}} f^*(\bar\alpha_T).	
\end{equation}

In this setting we are going to construct a natural homomorphism $S(F)_\pm \to T(F)_\pm$ of extensions lifting $f$, as well as a natural homomorphism $^LT_\pm \to {^LS}_\pm$, and then prove that the Langlands correspondence for the two double covers $S(F)_\pm$ and $T(F)_\pm$ is compatible with these two maps.

Using Construction \ref{cns:func_t} we immediately reduce to the case $S=T$. So from now on we focus on a single torus $S$ and a surjective $\Sigma$-equivariant map $f : R \to R'$ between to finite admissible $\Sigma$-sets equipped with maps to $X^*(S)$. The above assumption takes the form: for any $\alpha' \in R'$
\begin{equation} \label{eq:func}
\bar\alpha' = \sum_{\substack{\alpha \in R\\ f(\alpha)=\alpha'}} \bar\alpha.
\end{equation}

\begin{rem}
Note that Construction \ref{cns:simple} is a special case of this. In that case, the map $f : R \to R'$ is the identity on all $\Sigma$-orbits in $R$ except those of $\alpha$ and $\beta$, and sends both $\alpha$ and $\beta$ to $\gamma$.
\end{rem}

\begin{rem}
The situation discussed here may arise in practice as follows. Consider a torus $S$ defined over $F$, a finite admissible $\Sigma$-set $R$ and a $\Sigma$-equivariant map $R \to X^*(S)$. Consider further a $\Sigma$-equivariant equivalence relation $\sim$ on $R$. That is $\alpha \sim \beta$ implies $\sigma\alpha \sim \sigma\beta$ for all $\sigma \in \Sigma$. Assume that there does not exist $\alpha \in R$ such that $-\alpha \sim \alpha$. Then the quotient $R'=R/\sim$ is again a finite admissible $\Sigma$-set. We equip it with the map $R' \to X^*(S)$ that sends $\alpha' \in R'$ to the element $\bar\alpha'$ that is the sum of $\bar\alpha$ for all $\alpha$ in the equivalence class $\alpha'$. We will see an important instance of this in the discussion of twisted Levi subgroups in \S\ref{sec:tl}.
\end{rem}

\begin{rem}
The flexibility of allowing a map $R \to X^*(S)$ rather than a subset $R \subset X^*(S)$ is essential here. Indeed, even if $R$ were a subset of $X^*(S)$, the map $R' \to X^*(S)$ that we just constructed may fail to be injective, because two different equivalence classes may happen to yield the same sum of elements. This does happen in practice when dealing with Levi subgroups, as we will show in \S\ref{sub:exa}.
\end{rem}

\begin{cns} \label{cns:homcover}
We now construct the homomorphism $S(F)_R \to S(F)_{R'}$ lifting the identity of  $S(F)$. Let $\tilde\gamma = [\gamma,\epsilon,(\delta_\alpha)_{\alpha \in R_\tx{sym}}] \in S(F)_R$ as in Remark \ref{rem:elements}. We map it to the element $\tilde\gamma'=[\gamma,\epsilon,(\delta_{\alpha'})_{\alpha' \in R'_\tx{sym}}] \in S(F)_{R'}$, where $\delta_{\alpha'}$ is defined as follows. Let $\alpha_1,\dots,\alpha_k$ be the representatives for the $\Gamma_{\alpha'}$-orbits in the preimage of $\alpha'$. A given $\alpha_i$ may be symmetric or asymmetric. Assume first that it is symmetric. Write $\Gamma_i$ instead of $\Gamma_\alpha$. Then any $\tau_i \in \Gamma_{\pm i} \sm \Gamma_i$ lies in $\Gamma_{\pm\alpha'} \sm \Gamma_{\alpha'}$ by Lemma \ref{lem:symres}. It normalizes both $\Gamma_i$ and $\Gamma_{\alpha'}$ and hence commutes with $N_i := N_{F_{\alpha_i}/F_{\alpha'}}$. Define $\delta'_i=N_i(\delta_i)$. Then $\delta'_i/\tau_i(\delta'_i)=N_i(\delta_i/\tau_i(\delta_i))=N_i(\bar\alpha_i(\gamma))$. Assume now that $\alpha_i$ is asymmetric. Thus $-\Gamma\alpha_i \neq \Gamma\alpha_i$ so there exists $\tau_i \in \Gamma_{\pm\alpha'} \sm \Gamma_{\alpha'}$ s.t. $-\tau_i\alpha_i = \alpha_j$ for some $j \neq i$. Define $\delta'_i=N_i(\bar\alpha_i(\gamma))$. Then $\delta'_i/\tau_i(\delta'_i)=N_i(\bar\alpha_i(\gamma))N_j(\bar\alpha_j(\gamma))$. Now define $\delta_{\alpha'}$ to be the product of $\delta'_i$, where $i$ runs over those $i$ for which $\alpha_i$ is symmetric, as well as one representative for each pair of indices $i,j$ s.t. $\alpha_i,\alpha_j$ are asymmetric and $-\Gamma\alpha_i=\Gamma\alpha_j$. Then $\delta_{\alpha'}/\tau_{\pm\alpha'}(\delta_{\alpha'})=\prod_{i=1}^k N_i(\bar\alpha_i(\gamma))=(\sum_{i=1}^k\sum_{\sigma \in \Gamma_{\alpha'}/\Gamma_i}\sigma(\bar\alpha_i))(\gamma)=\bar\alpha'(\gamma)$, the last equality due to \eqref{eq:func}. This completes the construction of $[\gamma,\epsilon,(\delta_{\alpha'})_{\alpha' \in R'_\tx{sym}}]$ and shows that it represents an element $\tilde\gamma' \in S(F)_{R'}$. 
\end{cns}

It is immediately checked that $\tilde\gamma \mapsto \tilde\gamma'$ is a homomorphism, provided one knows that it is well-defined, which we will verify next.	

\begin{lem} The element $\tilde\gamma'$ just constructed depends only on $\tilde\gamma$ and not on the choice of presentation $[\gamma,\epsilon,(\delta_\alpha)_{\alpha \in R_\tx{sym}}]$.
\end{lem}
\begin{proof}
Let $(\eta_\alpha)_{\alpha \in R_\tx{sym}}$, define $\eta=\prod_{\alpha \in R_\tx{sym}/\Gamma} \kappa_\alpha(\eta_\alpha)$, and consider the new presentation $[\gamma,\epsilon\eta,(\delta_\alpha\eta_\alpha)]$ of $\tilde\gamma$. This will affect $\delta'_i$ for those $i$ for which $\alpha_i$ is symmetric, and replace it by $\delta'_i \cdot N_i(\eta_{\alpha_i})$. From Corollary \ref{cor:symres} we know $N_i|_{F_{\pm\alpha}^\times} = N_{F_{\pm\alpha}/F_{\pm\alpha'}}$ and so $N_i(\eta_{\alpha_i}) \in F_{\pm\alpha'}^\times$. The element $\delta_{\alpha'}$ is then replaced by $\delta_{\alpha'} \cdot \eta_{\alpha'}$, where $\eta_{\alpha'}=\prod_i N_i(\eta_{\alpha_i})$ and the product running over those $i$ for which $\alpha_i$ is symmetric. Thus we obtain the presentation $[\gamma,\epsilon\eta,(\delta_{\alpha'}\eta_{\alpha'})_{\alpha'\in R'_\tx{sym}})]$ and need to show that it represents the same element of $S(F)_{R'}$ as the presentation $[\gamma,\epsilon,(\delta_{\alpha'})_{\alpha'\in R'_\tx{sym}})]$. This amounts to showing
\[ \prod_{\alpha' \in R'_\tx{sym}/\Gamma} \kappa_{\alpha'}(\eta_{\alpha'}) = \prod_{\alpha \in R_\tx{sym}/\Gamma} \kappa_\alpha(\eta_\alpha).\]
Composing the quadratic character $\kappa_{\alpha'} : F_{\pm\alpha'}^\times/N_{F_{\alpha'}/F_{\pm\alpha'}}(F_{\alpha'}^\times) \to \{\pm 1\}$ with $N_i$ gives the quadratic character $\kappa_i : F_{\pm i}^\times/N_{F_i/F_{\pm i}}(F_i^\times) \to \{\pm 1\}$, according to Lemma \ref{lem:symres}. Therefore $\kappa_{\alpha'}(\eta_{\alpha'})=\prod_i \kappa_i(\eta_{\alpha_i})$. By Fact \ref{fct:reps} this product runs over representatives for the symmetric $\Gamma$-orbits in $R$ mapping to the $\Gamma$-orbit of $\alpha'$ in $R'$. Letting $\alpha'$ run over $R'_{\tx{sym}}/\Gamma$ and taking the product of the left-hand side of this equality has the effect of taking the product over all symmetric $\Gamma$-orbits $\alpha$ in $R$, as claimed.
\end{proof}

\begin{cns} \label{cns:homlcover}
Next we construct the isomorphism $^LS_{R'} \to {^LS_R}$. Fix a gauge $p'$ on $R'$ and define a gauge $p$ on $R$ by $p(\alpha)=p'(f(\alpha))$. That this is a gauge follows from the equivariance of $f$ with respect to $\{\pm1\} \subset \Sigma$. The value $t_p(\sigma,\tau)$ of the Tits cocycle $t_p \in Z^2(\Gamma,\hat S)$ relative to $R$ and $p$ is given by $\lambda_{p,\sigma,\tau}(-1)$, where $\lambda_{p,\sigma,\tau}$ is the sum of $\bar\alpha$ for $\alpha$ in the following subset of $R$
\[ \Lambda_{p,\sigma,\tau}=\{\alpha>0,\sigma^{-1}\alpha<0,(\sigma\tau)^{-1}\alpha>0\}, \]
where $\alpha>0$ stands for $p(\alpha)>0$. Equation \eqref{eq:func}, the $\Sigma$-equivariance of $f$, and the definition of $p$, imply that this sum equals the sum of $\bar\alpha'$ for $\alpha'$ in the following subset of $R'$
\[ \Lambda_{p',\sigma,\tau}=\{\alpha'>0,\sigma^{-1}\alpha'<0,(\sigma\tau)^{-1}\alpha'>0\}, \]
where now $\alpha'>0$ stands for $p'(\alpha')>0$. We conclude that the Tits cocycle $t_{p'}$  relative to $R'$ and $p'$ equals $t_p$. A similar argument shows that the 1-cochain $s_{p'/q'}$ equals $s_{p/q}$, so the isomorphisms corresponding to different choices of $p'$ splice together to an isomorphism $^LS_{R'} \to {^LS_R}$.	
\end{cns}

We now come to the desired compatibility of the Langlands correspondence. Before we state it, we give an alternative expression for the cochain $r_p$ of Definition \ref{dfn:rp}. For each $\alpha  \in R_\tx{sym}$ fix a section $s : W_{\pm\alpha} \lmod W \to W$ and let $r : W \to W_{\pm\alpha}$ be defined by $w=r(w)s(w)$. Then
\begin{equation} \label{eq:rpalt}
r_p(w) = \prod_{\alpha \in R/\Sigma} \prod_{u \in W_{\pm\alpha} \lmod W} \chi_\alpha(v_0(r(u)^{-1}r(uw)))^{s(u)^{-1}\bar\alpha}.
\end{equation}

\begin{pro} \label{pro:llcfunc}
Let $\chi'$ be a set of $\chi$-data for $R'$ and let $\chi=\tx{inf}\,\chi'$ as in Definition \ref{dfn:inf}. Then
\begin{enumerate}
	\item The pull-back along $S(F)_R \to S(F)_{R'}$ of the character $\chi'_S$ is equal to the character $\chi_S$.
	\item Let $q'$ be a gauge on $R'$ and let $q=q' \circ f$. The 1-cochains $r_{q',\chi'},r_{q,\chi} \in C^1(W_F,\hat S)$ are equal up to 1-coboundary.
\end{enumerate}
\end{pro}
\begin{proof}
To compare the pull-back of the character $\chi'_S$ with the character $\chi_S$, consider a $\Sigma$-orbit $O'$ in $R'$ and choose $\alpha' \in O'$. Let $\alpha_1,\dots,\alpha_k$ be a set of representatives for the $\Gamma_{\alpha'}$-orbits in the preimage of $\alpha'$. Let $\tilde\gamma = [\gamma,\epsilon,(\delta_\alpha)_{\alpha \in R_\tx{res}}]$ be an element of $S(F)_\pm$ with image $\tilde\gamma'=[\gamma,\epsilon,(\delta_{\alpha'})_{\alpha' \in R'_\tx{sym}}]$ as in Construction \ref{cns:homcover}. It is enough to show that the contribution of $O'$ to $\chi'_S(\tilde\gamma')$ equals the product of the contributions to $\chi_S(\tilde\gamma)$ of the $\Sigma$-orbits in $R$ lying in the preimage of the $\Sigma$-orbit of $\alpha'$.

Assume first that $\alpha'$ is asymmetric. Then each $\alpha_i$ is also asymmetric. The contribution of $O'$ to $\chi'_S(\tilde\gamma')$ is $\chi_{\alpha'}(\bar\alpha'(\gamma))$. By \eqref{eq:func} we have
\begin{eqnarray*}
\chi_{\alpha'}(\bar\alpha'(\gamma))&=&\chi_{\alpha'}\left(\Big[\sum_{i=1}^k\sum_{\sigma \in \Gamma_{\alpha'}/\Gamma_i}\sigma(\bar\alpha_i)\Big](\gamma)\right)\\
&=&\chi_{\alpha'}\left(\prod_{i=1}^k\prod_{\sigma \in \Gamma_{\alpha'}/\Gamma_i} \sigma(\bar\alpha_i(\gamma))\right)\\
&=&\prod_{i=1}^k\chi_{\alpha'}(N_i(\bar\alpha_i(\gamma))).
\end{eqnarray*}
Recalling that $\chi_{\alpha_i}=\chi_{\alpha'}\circ N_i$  and using Fact \ref{fct:reps} we see that this is indeed the product of the contributions to $\chi_S(\tilde\gamma)$ of the $\Sigma$-orbits in $R$ that are in the preimage of the $\Sigma$-orbit of $\alpha'$.

Assume now that $\alpha'$ is symmetric. Then the contribution of $O'$ to $\chi'_S(\tilde\gamma')$ is $\chi_{\alpha'}(\delta_{\alpha'})$. Partition $\{1,\dots,k\}$ as $I \sqcup J \sqcup J'$, where $I$ is the set of $i$ s.t. $\alpha_i$ is symmetric, and $J \cup J'$ is the set of $i$ s.t. $\alpha_i$ is asymmetric, and for $i \in J$ there exists $j \in J'$ s.t. $-\Gamma\alpha_i=\Gamma\alpha_j$. By Construction \ref{cns:homcover}
\[ \delta_{\alpha'} = \prod_{i \in I} N_i(\delta_{\alpha_i}) \cdot \prod_{j \in J} N_j(\bar\alpha_j(\gamma)). \]
For $i \in I$ we have $\chi_{\alpha'}(N_i(\delta_{\alpha_i}))=\chi_{\alpha_i}(\delta_{\alpha_i})$, which is the contribution of the $\Sigma$-orbit of $\alpha_i$ to $\chi_S(\tilde\gamma)$. For $j \in J$ we have $\chi_{\alpha'}(N_j(\bar\alpha_j(\gamma)))=\chi_{\alpha_j}(\bar\alpha_j(\gamma))$, which is the contribution of the $\Sigma$-orbit of $\alpha_j$ ot $\chi_S(\tilde\gamma)$. By Fact \ref{fct:reps} the set $\{\alpha_i|i \in I \cup J\}$ represents the $\Sigma$-orbits inside the preimage of the $\Sigma$-orbit of $\alpha'$. The comparison of $\chi'_S$ and $\chi_S$ is thus complete.

We now come to the comparison of $r_{q',\chi'}$ and $r_{q,\chi}$. Again we begin our consideration with an asymmetric $\Sigma$-orbit $O'$ in $R'$ and a representative $\alpha' \in O'$, as well as representatives $\alpha_1,\dots,\alpha_k$ for the $\Gamma_{\alpha'}$-orbits in the preimage of $\alpha'$. To ease notation, we shall write $W_i$ and $W'$ in place of $W_{\alpha_i}$ and $W_{\alpha'}$, and then later in the symmetric case also $W_{\pm i}$ and $W'_{\pm}$ analogously. The same notational conventions will apply to $\Gamma$ as well, so that for example $\Gamma'=\Gamma_{\alpha'}$. Consider the subgroups $W_i \subset W' \subset W$. We choose sections $s_1 : W' \lmod W \to W$ and $s_2 : W_i \lmod W' \to W'$. These determine $r_1 : W \to W'$ and $r_2 : W' \to W_i$ as in \eqref{eq:rpalt}. Then $s_3 : W_i \lmod W \to W$ defined by $s_3(w)=s_2(r_1(w))s_1(w)$ is a section and the corresponding $r_3 : W \to W_i$ is defined by $r_3(w)=r_2(r_1(w))$. Of course $s_2,s_3,r_2,r_3$ depend on $i$.

As in Definition \ref{dfn:rp} these choices determine a gauge $p'$ on $O'$ and a gauge $p$ on $O_1 \cup \dots \cup O_k$, where $O_i$ is the $\Sigma$-orbit of $\alpha_i$. Since all these $\Sigma$-orbits are asymmetric these gauges can be described more simply as follows: $p'(\beta')>0$ if and only if $\beta' \in \Gamma\alpha'$, and $p(\beta)>0$ if and only if $\beta \in \Gamma\alpha_i$ for some $1 \leq i \leq k$. In particular, these gauges are related by $p(\beta)=p'(f(\beta))$ and it follows that $s_{q',p'}=s_{q,p}$. It is thus enough to compare the contribution of $O'$ to $r_{p',\chi'}$ with the contribution of $O_1 \cup \dots \cup O_k$ to $r_{p,\chi}$. 

The contribution of $O'$ to $r_{p',\chi'}$ is
\[ w \mapsto \prod_{u \in W'\lmod W} u^{-1}\bar\alpha'(\chi_{\alpha'}(r_1(u)^{-1}r_1(uw))). \]
According to \eqref{eq:func} we have $\bar\alpha'=\sum_i \sum_{\sigma \in \Gamma_i\lmod \Gamma'} \sigma^{-1}\bar\alpha_i$, the above equals
\begin{equation} \label{eq:rpcomp1} w \mapsto \prod_i \prod_{u \in W_{i}\lmod W} u^{-1}\bar\alpha_i(\chi_{\alpha'}(r_1(u)^{-1}r_1(uw))). \end{equation}
On the other hand, by Fact \ref{fct:reps} the set $\alpha_1,\dots,\alpha_k$ represents the $\Sigma$-orbits in $R$ that constitute the preimage of the $\Sigma$-orbit of $\alpha'$. Their combined contribution to $r_{p,\chi}$ is
\begin{equation} \label{eq:rpcomp2} w \mapsto \prod_i \prod_{u \in W_{i}\lmod W} u^{-1}\bar\alpha_i(\chi_{\alpha_i}(r_3(u)^{-1}r_3(uw))). \end{equation}
Now $\chi_{\alpha_i}$ is simply the restriction of $\chi_{\alpha'}$ to $W_i$. By construction of $r_3$ the argument of $\chi_{\alpha_i}$ is $r_2(r_1(w))^{-1}r_2(r_1(uw))$. We write
\begin{eqnarray*}
r_1(u)^{-1}r_1(uw)&=&(r_2(r_1(u))s_2(r_1(u)))^{-1}(r_2(r_1(uw))s_2(r_1(uw)))\\
&=&s_2(r_1(u))^{-1}(r_2(r_1(u))^{-1}r_2(r_1(uw)))s_2(r_1(uw)).
\end{eqnarray*}
Comparing the contributions of a fixed $i$ to \eqref{eq:rpcomp1} and \eqref{eq:rpcomp2} we see that they differ by
\[ \prod_u u^{-1}\bar\alpha_i\chi_{\alpha_i}\Big( s_2(r_1(u))^{-1} s_2(r_1(uw))\Big). \]
Splitting the product and reindexing this becomes
\[ \prod_u u^{-1}\bar\alpha_i\chi_{\alpha_i}\Big( s_2(r_1(u))\Big)^{-1}  \cdot w\prod_u u^{-1}\bar\alpha_i\chi_{\alpha_i}\Big( s_2(r_1(u))\Big), \]
visibly a coboundary. Thus the contribution to $r_{p',\chi'}$ of $O'$ equals up to coboundary the contribution to $r_{p,\chi}$ of the preimage in $R$ of $O'$.

Consider next a symmetric $\alpha'$. We have $W' \subset W'_{\pm} \subset W$. Choose $s_1 : W'_{\pm} \lmod W \to W$ and obtain $r_1 : W \to W'_{\pm}$. Choose also $v_0 \in W'$ and $v_1 \in W'_{\pm} \sm W'$. The resulting gauge $p'$ on $R'$ declares $s_1(u)^{-1}\alpha'$ positive for all $u$. The contribution of $\alpha'$ to $r_{p',\chi'}$ is
\[ w \mapsto \prod_{u \in W'_{\pm}\lmod W} s_1(u)^{-1}\bar\alpha'(\chi_{\alpha'}(v_0(r_1(u)^{-1}r_1(uw)))). \]
For simplicity we take $v_0=1$.

Again we decompose the preimage of $\alpha'$ into $\Gamma'$-orbits. The map $\alpha \mapsto -v_1\alpha$ acts as an involution on the set of these orbits. An orbit is fixed by this involution if and only if it consists of symmetric elements. We choose representatives $\alpha_1,\dots,\alpha_l$ for the symmetric orbits, as well as representatives $\alpha_{l+1},\dots,\alpha_k$, one for each pair of asymmetric orbits. Then we set $\alpha_{k+j+1}=-v_1\alpha_{k-j}$, so that $\alpha_1,\dots,\alpha_{2k-l}$ are representatives for all orbits. Then \eqref{eq:func} implies that $\bar\alpha' = \sum_{i=1}^{2k-l} \sum_{\sigma \in \Gamma_i \lmod \Gamma'} \sigma^{-1}\bar\alpha_i$. We plug that into the above formula to get
\begin{equation} \label{eq:rpcomp10}
w \mapsto \prod_{i=1}^{2k-l} \prod_{u \in W'_{\pm}\lmod W} \prod_{\sigma \in \Gamma_i\lmod \Gamma'} s_1(u)^{-1}\sigma^{-1}\bar\alpha_i(\chi_{\alpha'}(v_0(r_1(u)^{-1}r_1(uw)))).
\end{equation}

We now consider the contribution of an individual $i$. We begin with a symmetric $\alpha_i$. Choose $s_2 : W_i
\lmod W' \to W'$ and view it, using the identity $\Gamma'_\pm/\Gamma_{\pm i}=\Gamma'/\Gamma_i$ of Lemma \ref{lem:symres}, as $W_{\pm i} \lmod W'_{\pm} \to W'_{\pm}$. Declare as before $s_3 : W_{\pm i} \lmod W \to W$ by $s_3(w)=s_2(r_1(w))s_1(w)$. The $i$-th factor in the above product equals
\begin{equation} \label{eq:rpcomp11}
w \mapsto \prod_{u \in W_{\pm i}\lmod W} s_3(u)^{-1}\bar\alpha_i(\chi_{\alpha'}(v_0(r_1(u)^{-1}r_1(uw)))).	
\end{equation}
We use again
\[r_1(u)^{-1}r_1(uw) = s_2(r_1(u))^{-1}(r_3(u)^{-1}r_3(uw))s_2(r_1(uw)) \]
and write it as
\[ s_2(r_1(u))^{-1} \cdot \tx{Ad}[r_3(u)^{-1}r_3(uw)](s_2(r_1(uw))) \cdot r_3(u)^{-1}r_3(uw). \]
Since the image of $s_2$ lies in $W'$ we have that $v_0(r_1(u)^{-1}r_1(uw))$ is equal to
\[ s_2(r_1(u))^{-1} \cdot \tx{Ad}[r_3(u)^{-1}r_3(uw)](s_2(r_1(uw))) \cdot v_0(r_3(u)^{-1}r_3(uw)). \]
We claim that the first two terms contribute a coboundary. Indeed, the first term contributes
\[ w \mapsto \prod_{u \in W_{\pm i}\lmod W} s_3(u)^{-1}\bar\alpha_i(\chi_{\alpha'}(s_2(r_1(u))))^{-1}, \]
while the contribution of the second is, upon making the substitution $u \mapsto uw^{-1}$, equal to
\[ w \mapsto \prod_{u \in W_{\pm i}\lmod W} s_3(uw^{-1})^{-1}\bar\alpha_i(\chi_{\alpha'}(\tx{Ad}[r_3(uw^{-1})^{-1}r_3(u)]s_2(r_1(u)))). \]
Now $s_3(uw^{-1})=r_3(uw^{-1})^{-1}r_3(u)s_3(u)w^{-1}$, which makes the above equal to
\[ w \mapsto w\sprod{u \in W_{\pm i}\lmod W}s_3(u)^{-1}r_3(u)^{-1}r_3(uw^{-1})\bar\alpha_i(\chi_{\alpha'}(\tx{Ad}[r_3(uw^{-1})^{-1}r_3(u)]s_2(r_1(u)))). \]
But $(r_3(u)^{-1}r_3(uw^{-1})\bar\alpha_i)\circ \chi_{\alpha'}\circ \tx{Ad}[r_3(uw^{-1})^{-1}r_3(u)]=\bar\alpha_i\circ\chi_{\alpha'}$, and we indeed conclude that the first two terms combine to a coboundary.

We turn to the third term, namely $v_0(r_3(u)^{-1}r_3(uw))$. Its contribution to \eqref{eq:rpcomp11} is
\begin{equation} \label{eq:rpcomp12}
w \mapsto \prod_{u \in W_{\pm i}\lmod W} s_3(u)^{-1}\bar\alpha_i(\chi_{\alpha'}(v_0(r_3(u)^{-1}r_3(uw)))).	
\end{equation}
Since $\chi_{\alpha'}|_{W_i}=\chi_{\alpha_i}$, this looks exactly like the contribution to $r_{p,\chi}$ of the $\Sigma$-orbit of $\alpha_i$. But we need to be careful here. The chosen $v_1$ may not lie in $W_{\pm i}$. Choose $v_1' \in W_{\pm i} \sm W_{i}$. There exists $a \in W'$ so that $v_1'=av_1$. We denote by $v_0'(-)$ the version of the term $v_0(-)$ with respect to $v_1'$ instead of $v_1$. The contribution to $r_{p,\chi}$ of the $\Sigma$-orbit of $\alpha_i$ is given by
\begin{equation} \label{eq:rpcomp13}
w \mapsto \prod_{u \in W_{\pm i}\lmod W} s_3(u)^{-1}\bar\alpha_i(\chi_{\alpha'}(v_0'(r_3(u)^{-1}r_3(uw)))).	
\end{equation}
Now \eqref{eq:rpcomp12} is the product of \eqref{eq:rpcomp13} and
\[ w \mapsto \begin{cases}
\prod_{u \in W_{\pm i}\lmod W}s_3(u)^{-1}\bar\alpha_i(\chi_{\alpha'}(a^{-1})),& r_3(u)^{-1}r_3(uw) \notin W_{i}\\
1,&r_3(u)^{-1}r_3(uw) \in W_{i}\\
\end{cases}
\]
This is the image under $\tx{cor}\circ\bar\alpha_i : Z^1(W_{\pm i}/W_i,\C^\times_{(-1)}) \to Z^1(W_F,\hat S)$ of
\[
x \mapsto \begin{cases} \chi_{\alpha'}(a^{-1})&,x \notin W_{i}\\ 1&,x \in W_{i}
\end{cases}
\]
Since $H^1(\Gamma_{\pm i}/\Gamma_i,\C^\times_{(-1)})$ vanishes, this is also a coboundary. We conclude that, up to a coboundary, the 1-cochains \eqref{eq:rpcomp11}, \eqref{eq:rpcomp12}, and \eqref{eq:rpcomp13}, are all equal. Thus, the $i$-th factor in \eqref{eq:rpcomp10} equals up to a coboundary the contribution to $r_{p,\chi}$ of the $\Sigma$-orbit $O_i$ of $\alpha_i$. Here the gauge $p$ on the orbit $O_i$ declares each $s_3(u)^{-1}\alpha_i$ positive. Recalling that $s_3(u)=s_2(r_1(u))s_1(u)$ and that $s_2$ takes values in $W'$ we see that, as in the case of asymmetric $\alpha'$, we have $p(\beta)=p'(f(\beta))$ for all $\beta \in O_i$. Therefore, as in the asymmetric case, the equality of these parts of $r_{p',\chi'}$ and $r_{p,\chi}$ implies the equality of the corresponding parts of $r_{q',\chi'}$ and $r_{q,\chi}$.

We turn again to the contribution of an individual $i$ to \eqref{eq:rpcomp10}, but now we consider the case of an asymmetric $\alpha_{k-j}$. We combine this contribution with that of its partner $\alpha_{k+j+1}=-v_1\alpha_{k-j}$. This joint contribution is
\[ \prod_{u \in W'_{\pm}\lmod W} \prod_{\sigma \in \Gamma_{i}\lmod \Gamma_{\tx{eq}}} s_1(u)^{-1}(\sigma^{-1}\bar\alpha_i -v_1\sigma^{-1}\bar\alpha_i)(\chi_{\alpha'}(v_0(r_1(u)^{-1}r_1(uw)))), \]
where we have set $i=k-j$.

We now encode $v_0=1$ and $v_1$ as a section $s^\circ : W' \lmod W'_{\pm} \to W'_{\pm}$ and choose further a section $s_2 : W_{i} \lmod W' \to W'$. We can combine these into a section $s_3 : W_{i} \lmod W \to W$ defined by $s_3(a)=s_2r^\circ r_1(a) \cdot s^\circ r_1(a) \cdot s_1(a)$. The corresponding $r_3 : W \to W_{i}$ is defined by $r_3(a)=r_2r^\circ r_1(a)$.

For any $a \in W$ we can decompose $r_1(a) \in W'_\pm$ as
\[ r_1(a) = \underbrace{r_2r^\circ r_1(a)}_{=r_3(a)\in W_{i}} \cdot \underbrace{s_2r^\circ r_1(a)}_{\in W'} \cdot \underbrace{s^\circ r_1(a)}_{\in W'_{\pm}}. \]
Using this, we compute $v_0(r_1(a)^{-1}r_1(aw))$ for any $a \in W$ by distinguishing four cases, depending on whether $r_1(a)$ and $r_1(aw)$ belong to $W'$ or not, which we record as a sequence of two signs:
\[
\begin{cases}
[r_3(a)^{-1}r_3(aw)]\cdot [s_2r^\circ r_1(a)^{-1}s_2r^\circ r_1(aw)] \cdot [s^\circ r_1(a)^{-1} s^\circ r_1(aw)],& ++	\\
[r_3(a)^{-1}r_3(aw)]\cdot [s_2r^\circ r_1(a)^{-1}s_2r^\circ r_1(aw)] \cdot [s^\circ r_1(a)^{-1} s^\circ r_1(aw)v_1^{-1}],& +-	\\
[r_3(a)^{-1}r_3(aw)]^{c}\cdot [s_2r^\circ r_1(a)^{-1}s_2r^\circ r_1(aw)]^{c} \cdot [s^\circ r_1(a)^{-1} s^\circ r_1(aw)v_1^{-1}],& -+	\\
[r_3(a)^{-1}r_3(aw)]^{c}\cdot [s_2r^\circ r_1(a)^{-1}s_2r^\circ r_1(aw)]^{c} \cdot [s^\circ r_1(a)^{-1} s^\circ r_1(aw)],& --	\\
\end{cases}
\]
Each row has three terms isolated by square brackets. Each term belongs to $W'$. The superscript $c$ denotes the conjugation action of $v_1$ on $W'$. We consider first the contributions of the third term in each row. These four terms are equal to $1$, $1$, $v_1^{-2}$, and $1$, respectively. Since $\chi_{\alpha'}(v_1^2)=-1$, we see that the contributions of the third terms combines to
\begin{equation} \label{eq:rpcomp4} w \mapsto \prod_{\substack{a \in W_i \lmod W\\ r_1(a) \notin W'\\ r_1(aw) \in W'}} s_3(a)^{-1}\bar\alpha_i(-1). \end{equation}

Next we consider the contributions of the first term in each row. Since $\chi_{\alpha'}(x^c)=\chi_{-\alpha'}(x)=\chi_{\alpha'}(x)^{-1}$ they combine to
\[ \prod_{a \in W_{i} \lmod W} s_3(a)^{-1}\bar\alpha_i(\chi_{\alpha'}(r_3(a)^{-1}r_3(aw))), \]
which is the contribution of $\Sigma$-orbit of the asymmetric $\alpha_{k-j}$ (this orbit contains its partner $\alpha_{k+j+1}$) to $r_{p,\chi}$, where $p$ is the gauge on the $\Sigma$-orbit of $\alpha_{k-j}$ declaring that $s_3(a)^{-1}\alpha_i$ is positive.

Finally we consider the contributions of the second term in each row. They combine to
\[ \prod_{a \in W_{i} \lmod W} s_3(a)^{-1}\bar\alpha_i(\chi_{\alpha'}(s_2r^\circ r_1(a)^{-1}s_2r^\circ r_1(aw))). \]
Now both $s_2r^\circ r_1(a)$ and $s_2r^\circ r_1(aw)$ belong to $W'$ and can be individually plugged into $\chi_{\alpha'}$. Moreover $s_3(a)^{-1}\alpha_i=a^{-1}\alpha_i$. Substituting $a \mapsto aw^{-1}$ in the second term reveals the above contribution to be a coboundary.

We have concluded that the contribution of the pair of indices $i=k-j$ and $i=k+j+1$ to \eqref{eq:rpcomp10} differs from the contribution of the $\Sigma$-orbit of $\alpha_{k-j}$ to $r_{p,\chi}$ by the term \eqref{eq:rpcomp4}. We claim that this difference is accounted for by the fact that the gauges $p$ and $p'$ are not related by $p(\beta)=p'(f(\beta))$. To study this, let $p_0$ be the gauge on $\Sigma\alpha_{k-j}$ defined by $p_0(\beta)=p'(f(\beta))$. Reviewing the definitions of $p$ and $p'$ we see that $p$ declares the entire $\Gamma$-orbit of the asymmetric $\alpha_{k-j}$ positive and the entire $\Gamma$-orbit of $\alpha_{k+j+1}$ negative, while $p_0$ declares half of the $\Gamma$-orbit of both $\alpha_{k-j}$ and $\alpha_{k+j+1}$ positive. To state this more precisely, we recall $s_3(a)=s_2r^\circ r_1(a) \cdot s^\circ r_1(a) \cdot s_1(a)$. Letting $i=k-j$ we have
\begin{eqnarray*}
p_0(a^{-1}\alpha_i)=&=&p_0(s_3(a)^{-1}\alpha_i)\\
&=&p_0(s_1(a)^{-1}(s^\circ r_1(a))^{-1}(s_2r^\circ r_1(a))^{-1}\alpha_i)\\
&=&p'(s_1(a)^{-1}(s^\circ r_1(a))^{-1}\alpha'),
\end{eqnarray*}
where we have used that $s_2$ takes values in $W'$. When $r_1(a) \in W'$ then the value $p_0(s_3(a)^{-1}\alpha_i)$ is positive and hence equal to $p(s_3(a)^{-1}\alpha_i)$, and when $r_1(a) \notin W'$ then the value $p_0(s_3(a)^{-1}\alpha_i)$ is negative and hence unequal to $p(s_3(a)^{-1}\alpha_i)$.

Now $r_{p_0,\chi}=s_{p_0/p}r_{p,\chi}$ and up to coboundaries $s_{p_0/p}=s_{p/p_0}$. Reviewing the definition of $s_{p/p_0}$ we see that the product determined by the conditions $p(\lambda)=1$, $p(\sigma^{-1}\lambda)=-1$, $p_0(\lambda)=1$, $p_0(\sigma^{-1}\lambda)=1$ is empty, because $p$ is constant on $\Gamma$-orbits, while the product over $p(\lambda)=1$, $p(\sigma^{-1}\lambda)=1$ $p_0(\lambda)=-1$, $p_0(\sigma^{-1}\lambda)=1$ runs over members $a^{-1}\alpha_i$ of the $\Gamma$-orbit of $\alpha_i$ satisfying $r_1(a) \notin W'$ and $r_1(aw) \in W'$, where $w \in W$ represents $\sigma \in \Gamma$, and hence equals precisely \eqref{eq:rpcomp4}. We conclude that $r_{p',\chi'}=r_{p_0,\chi}$, and hence $r_{q',\chi'}=r_{q,\chi}$.
\end{proof}

\begin{cor} \label{cor:llcfunc1}
If $\zeta'$ is a collection of $\zeta$-data for $R'$ and we define $\zeta=\tx{inf}\,\zeta'$ as in Definition \ref{dfn:inf}, then the pull-back to $S(F)$ of the character $\zeta'_S$ equals the character $\zeta_S$.	
\end{cor}
\begin{proof}
Apply Proposition \ref{pro:llcfunc} to $(\chi_{\alpha'})$ and $(\chi_{\alpha'}\zeta_{\alpha'})$ for an arbitrary choice of $\chi$-data $(\chi_{\alpha'})$.
\end{proof}

We now return to the case of two tori $S$ and $T$ and two $\Sigma$-equivariant maps $R_S \to X^*(S)$, $R_T \to X^*(T)$, as well as a surjective $\Sigma$-equivariant map $f : R_T \to R_S$ satisfying \eqref{eq:func_st}.

\begin{cor} \label{cor:llcfunc} If $\varphi : W_F \to {^LT_\pm}$ is a Langlands parameter and $\chi : T_\pm(F) \to \C^\times$ is the associated genuine character, composing $\varphi$ with the natural $L$\-homomorphism ${^LT_\pm} \to {^LS_\pm}$ is a parameter whose corresponding character is the pull-back of $\chi$ along the natural lift $S_\pm(F) \to T_\pm(F)$ of $f$.
\end{cor}
\begin{proof}
This is immediate from Proposition \ref{pro:llcfunc} and Construction \ref{cns:func_t}.
\end{proof}

\subsection{A comparison of minimally ramified $\chi$-data for $R$ and $R'$}

Consider a non-archimedean local field $F$ whose residual characteristic $p$ is not equal to $2$, a torus $S$ defined over $F$, two finite admissible $\Sigma$-sets $R$ and $R'$, a $\Sigma$-equivariant $R \to X^*(S)$, and a surjective $\Sigma$-equivariant map $f : R \to R'$. We do not assume the existence of a $\Sigma$-equivariant map $R' \to X^*(S)$ or the validity of \eqref{eq:func}.

Consider minimally ramified $\chi$-data for $R$ in the sense of \cite[Definition 4.6.1]{KalRSP}, which we recall means that, for $\alpha \in R$,
\begin{enumerate}
	\item $\chi_\alpha=1$ if $\alpha$ is asymmetric;
	\item $\chi_\alpha$ is unramified if $\alpha$ is symmetric unramified;
	\item $\chi_\alpha$ is tamely ramified if $\alpha$ is symmetric ramified.
\end{enumerate}

Given tame $\chi$-data that is not necessarily minimally ramified we can associate canonically minimally ramified $\chi$-data in the following straightforward way:

\begin{dfn} \label{dfn:min}
Define $\tx{min}\,\chi$ by
\begin{enumerate}
	\item $(\tx{min}\,\chi)_\alpha=1$ if $\alpha$ is asymmetric;
	\item $(\tx{min}\,\chi)_\alpha$ is the unique unramified extension of $\kappa_\alpha$ if $\alpha$ is symmetric unramified;
	\item $(\tx{min}\,\chi)_\alpha=\chi_\alpha$ if $\alpha$ is symmetric ramified.
\end{enumerate}	
\end{dfn}

If $\chi'$ is minimally ramified $\chi$-data for $R'$, then the $\chi$-data $\tx{inf}\,\chi'$ of Definition \ref{dfn:inf} is tame, but need not be minimal. Indeed, this happens when $\alpha$ is asymmetric but $\alpha'$ is symmetric, or when $\alpha$ is symmetric unramified but $\alpha'$ is symmetric ramified. Both of these situations do occur, as we shall see in the examples provided in \S\ref{sub:exa}.

Consider the torus $T$ whose character module is the $\Z$-span of $R$ in $X^*(S)$. We have the natural map $S \to T$ dual to the inclusion $X^*(T) \to X^*(S)$. Let $S(F)_c$ be the preimage of the bounded subgroup $T(F)_b$ of $T(F)$. If $R$ is elliptic, in the sense that the $\Z$-span of $R$ in $X^*(S)$ does not have non-zero $\Gamma$-fixed elements, then $T$ is anisotropic and hence $S(F)_c=S(F)$.

Recall the character $\zeta_S : S(F) \to \C^\times$ associated to a set of $\zeta$-data in \S\ref{sub:chichar}.

\begin{pro} \label{pro:mindiff}
Let $\chi'$ be minimally ramified and define $\zeta_\alpha = (\tx{min}\,\tx{inf}\,\chi')_\alpha \cdot (\tx{inf}\,\chi')_\alpha^{-1}$. The restriction of $\zeta_S : S(F) \to \C^\times$ to $S(F)_c$ is the following sign character
\[ \gamma \mapsto \prod_{\substack{\alpha \in R_\tx{asym}/\Sigma\\ \alpha' \in R'_\tx{sym,ram}}} \tx{sgn}_{k_\alpha^\times}(\bar\alpha(\gamma))^{e(\alpha/\alpha')} \cdot \prod_{\substack{\alpha \in R_\tx{sym,unram}/\Gamma\\ \alpha' \in R'_\tx{sym,ram}}} \tx{sgn}_{k_\alpha^1}(\bar\alpha(\gamma))^{e(\alpha/\alpha')}. \]
In particular, this character is independent of $\chi'$.
\end{pro}

Let us explain the notation. The first product runs over asymmetric $\Sigma$-orbits in $R$ whose images in $R'$ are symmetric ramified, while the second product runs over symmetric unramified orbits in $R$ whose images in $R'$ are symmetric ramified. Since $\gamma$ maps to the bounded subgroup of $T(F)_b$, $\bar\alpha(\gamma)$ belongs to $O_{F_\alpha}^\times$ and we can project to $k_\alpha^\times$. When $\alpha$ is symmetric unramified, the image of $\bar\alpha(\gamma)$ in $k_\alpha^\times$ belongs to the subgroup $k_\alpha^1$ of elements whose norm to $k_{\pm\alpha}^\times$ is $1$. Both groups $k_\alpha^\times$ and $k_\alpha^1$ are cyclic of even order, hence have a canonical sign character, which we denote by $\tx{sgn}$. Finally, $e(\alpha/\alpha')$ is the ramification degree of the field extension $F_\alpha/F_{\alpha'}$.

\begin{proof}
Write $\chi=\tx{inf}\,\chi'$ and $\chi^m=\tx{min}\,\tx{inf}\,\chi'$. Given $\alpha \in R$, we have the following cases in which $\zeta_\alpha$ is automatically trivial, because $\chi_\alpha=\chi^m_\alpha$:
\begin{enumerate}
	\item $\alpha$ is ramified symmetric (then so is also $\alpha'$).
	\item $\alpha$ and $\alpha'$ are both unramified symmetric. 
	\item $\alpha$ and $\alpha'$ are both asymmetric. 
\end{enumerate}

This leaves the following cases, in wich $\zeta_\alpha$ is non-trivial:
\begin{enumerate}
	\item $\alpha$ is unramified symmetric, but $\alpha'$ is ramified symmetric.
	\item $\alpha$ is asymmetric, but $\alpha'$ is unramified symmetric.
	\item $\alpha$ is asymmetric, but $\alpha'$ is ramified symmetric.
\end{enumerate}

Writing a,u,r, for asymmetric, unramified symmetric, ramified symetric, the three cases we need to consider can be abbreviated as (u,r), (a,u), (a,r).

In the settings (a,u) and (a,r) we have $\zeta_\alpha=\chi_\alpha=\chi_{\alpha'}\circ N_{F_\alpha/F_{\alpha'}}$ and the contribution of the $\Sigma$-orbit of $\alpha$ is $\zeta_\alpha \circ \alpha$. In the setting (a,u) the character $\zeta_\alpha$ is unramified and $\bar\alpha(\gamma) \in O_{\alpha}^\times$ implies that $\zeta_\alpha \circ \alpha$ is trivial. In the setting (a,r) the restriction of $\chi_{\alpha'}$ to $O_{\alpha'}^\times$ factors to the sign character of $k_{\alpha'}^\times$, and the norm map $N_{F_\alpha/F_{\alpha'}}$ factors as $[N_{k_\alpha/k_{\alpha'}}]^{e(\alpha/\alpha')}$. Composing the sign character of $k_{\alpha'}^\times$  with $N_{k_\alpha/k_{\alpha'}}$ gives the sign character of $k_\alpha^\times$. So the contribution to $\zeta : S(F) \to \C^\times$ of $\alpha$ of type (a,r) is $\tx{sgn}_{k_\alpha^\times}(\bar\alpha(\gamma))^{e(\alpha/\alpha')}$.

In the setting (u,r) we have $\zeta_\alpha = \chi_\alpha \cdot (\chi^m_\alpha)^{-1}$ and the contribution of $\alpha$ is $\chi_\alpha(\delta_\alpha) \cdot \chi^m_\alpha(\delta_\alpha)^{-1}$, where $\delta_\alpha \in F_\alpha^\times$ is chosen so that $\delta_\alpha/\tau_\alpha(\delta_\alpha)=\bar\alpha(\gamma)$. Since $F_\alpha/F_{\pm\alpha}$ is unramified, $\delta_\alpha$ can be chosen in $O_\alpha^\times$. Now $\chi^m_\alpha$ is the unramified quadratic character of $F_\alpha^\times$, so its restriction to $O_\alpha^\times$ is trivial, while $\chi_\alpha$ is as the same as in type (a,r), namely its restriction to $O_\alpha^\times$ factors through the quadratic character of $k_\alpha^\times$ taken to the power $e(\alpha/\alpha')$.  Thus the contribution of $\alpha$ to $\zeta_S$ becomes $\tx{sgn}_{k_\alpha^\times}(\delta_\alpha)^{e(\alpha/\alpha')}$. The quadratic character of $k_\alpha^\times$ kills $k_{\pm\alpha}^\times$ and thus factors through the map $k_\alpha^\times \to k_\alpha^1$ and gives on $k_\alpha^1$ the quadratic character there. Since the image of $\delta_\alpha$ under this map is $\bar\alpha(\gamma)$, the contribution to $\zeta_S$ of $\alpha$ is $\tx{sgn}_{k_\alpha^1}(\bar\alpha(\gamma))^{e(\alpha/\alpha')}$.
\end{proof}

We now consider the setting of Corollary \ref{cor:llcfunc1}. In particular, we now assume given a $\Sigma$-equivariant map $R' \to X^*(S)$ and the validity of \eqref{eq:func}.

\begin{cor} \label{cor:llcfunc1b}
Let $\chi'_1$ and $\chi'_2$ be two minimal $\chi$-data for $R'$ and let $\zeta'=\chi'_2 \cdot (\chi'_1)^{-1}$. Let $\zeta = \tx{min}\ \tx{inf}\chi'_2 \cdot (\tx{min}\ \tx{inf}\chi'_1)^{-1}$. Then $\zeta_S$ is the pull-back of $\zeta'_S$.
\end{cor}
Notice that the difference between Corollaries \ref{cor:llcfunc1} and \ref{cor:llcfunc1b} is that in the latter case we have passed to minimal $\chi$-data after inflation.
\begin{proof}
We have $\zeta=[\tx{min}\ \tx{inf}\chi'_2 \cdot (\tx{inf}\chi'_2)^{-1}] \cdot \tx{inf}\zeta' \cdot [(\tx{inf}\chi'_1) \cdot (\tx{min}\ \tx{inf}\chi'_1)^{-1}]^{-1}$
and hence $\zeta_S = [\tx{min}\ \tx{inf}\chi'_2 \cdot (\tx{inf}\chi'_2)^{-1}]_S \cdot [\tx{inf}\zeta']_S \cdot [(\tx{inf}\chi'_1) \cdot (\tx{min}\ \tx{inf}\chi'_1)^{-1}]_S^{-1}$. By Proposition \ref{pro:mindiff} the first and third terms in square brackets are the same and thus cancel out, while Corollary \ref{cor:llcfunc1} tells us that the middle term is the pull-back of $\zeta'_S$.
\end{proof}

\section{Twisted Levi subgroups} \label{sec:tl}

Let $G$ be a connected reductive group defined over $F$. A twisted Levi subgroup of $G$ is a connected reductive subgroup $M \subset G$ defined over $F$ that is the Levi factor of a parabolic subgroup defined over a field extension $E/F$. When $E=F$ then $M$ is called a (actual) Levi subgroup. In the latter case it is well known that there is a canonical $\hat G$-conjugacy class of $L$-embeddings $^LM \to {^LG}$. For a twisted Levi subgroup this is no longer the case. Indeed, any maximal torus is a twisted Levi subgroup. We will show presently that the Langlands-Shelstad construction lends itself easily to the construction of $L$-embeddings $^LM \to {^LG}$ when $M$ is a twisted Levi subgroup. We will then apply the results of \S\ref{sub:func} to discuss compositions of the form $^LT \to {^LM} \to {^LG}$ for a maximal torus $T \subset M \subset G$. We will also reinterpret these matters in terms of double covers, and in particular introduce a double cover $M(F)_\pm$ of $M(F)$ and an associate $L$-group.

\subsection{Embedding the $L$-group of a twisted Levi subgroup} \label{sub:tlem}

Let $G$ be a connected reductive group defined over $F$ and $M \subset G$ a twisted Levi subgroup. Let $M_\tx{ab}=M/M_\tx{der}$ be the maximal abelian quotient of $M$. It is a torus defined over $F$. We begin by defining a finite admissible $\Sigma$-set $R(M_\tx{ab},G)$ with a $\Sigma$-equivariant map $R(M_\tx{ab},G) \to X^*(M_\tx{ab})$.

\begin{lem} \label{lem:rmab}
Let $S \subset M$ be a maximal torus. Let $R_S$ be the set of $\Omega(S,M)$-orbits in $R(S,G) \sm R(S,M)$.
\begin{enumerate}
	\item The set $R_S$ is a finite admissible $\Sigma$-set.
	\item The map $R_S \to X^*(S)$ that sends an $\Omega(S,M)$-orbit in $R(S,G) \sm R(S,M)$ to the sum of its elements takes image in $X^*(M_\tx{ab})$ and is $\Sigma$-equivariant.
	\item If $S_1,S_2 \subset M$ are two maximal tori, then any element $m \in M$ such that $S_2=mS_1m^{-1}$ induces the same bijection $R_{S_2} \to R_{S_1}$. In particular, this bijection is $\Sigma$-equivariant. Moreover, it respects the maps $R_{S_1} \to X^*(M_\tx{ab})$ and $R_{S_2} \to X^*(M_\tx{ab})$.
	\item Let $\Delta(S,G)$ be a set of simple roots for $G$ that contains a set of simple roots for $M$. The gauge $p : R(S,G) \sm R(S,M) \to \{\pm 1\}$ that sends a positive root to $+1$ is constant on $\Omega(S,M)$-orbits and hence induces a gauge on $R_S$.
\end{enumerate} 
\end{lem}
\begin{proof}
It is clear that $R_S$ is finite. The action of $\Sigma$ on $R(S,G)$ preserves $R(S,M)$ and respects the orbits of $\Omega(S,M)$, since $M$ is defined over $F$. Therefore $R_S$ is a finite $\Sigma$-set. Its admissibility will follow from 4.

The map $R(S,G) \sm R(S,M) \to X^*(S)$ is obviously $\Sigma$-equivariant and descends to $R_S$, and its image lies in $X^*(S)^{\Omega(S,M)}=X^*(M_\tx{ab})$.

The isomorphism $\tx{Ad}(m) : X^*(S_2) \to X^*(S_1)$ identifies $R(S_2,G) \sm R(S_2,M)$ with $R(S_1,G) \sm R(S_1,M)$ and is $\Omega(S_2,M)-\Omega(S_1,M)$ equivariant. Moreover, it is the identity on the embedded copy of $X^*(M_\tx{ab})$ into $X^*(S_2)$ and $X^*(S_1)$. Therefore it induces a bijection $R_{S_2} \to R_{S_1}$ that respects the maps to $X^*(M_\tx{ab})$. Another choice of $m$ is of the form $nm$ with $n \in N(S_1,M)$, so induces the same bijection $R_{S_2} \to R_{S_1}$. 

Consider a set of simple roots $\Delta(S,G)$ for $G$ that contains a set of simple roots for $M$. Let $\eta \in X_*(S/Z(G))\otimes\R$ denote a regular element in the corresponding acute Weyl chamber. Then the gauge $p$ is given by $p(\alpha)=\tx{sgn}(\<\eta,\alpha\>)$. The action of $\Omega(S,M)$ preserves the set of positive roots in $R(S,G) \sm R(S,M)$ and therefore for any $\alpha \in R(S,G) \sm R(S,M)$ the sign of the non-zero real number $\<\eta,w\alpha\>$ is independent of $w \in \Omega(S,M)$. Thus $p$ induces a gauge on $R_S$. This gauge shows that no element of $R_S$ is fixed by $-1 \in \Sigma$, so $R_S$ is indeed admissible as claimed by 1. 
\end{proof}

\begin{dfn} \label{dfn:rmab}
We define the set $R(M_\tx{ab},G)$ to be the limit over all maximal tori $S \subset M$ of the set $R_S$ of Lemma \ref{lem:rmab}, with transition maps given by $\tx{Ad}(m)$ for $m \in M$. This set comes equipped with a map to $X^*(M_\tx{ab})$.
\end{dfn}

\begin{rem} The map $R(M_\tx{ab},G) \to X^*(M_\tx{ab})$ need not be injectve, as we will see in \S\ref{sub:exa}.
\end{rem}

\begin{lem} \label{lem:tlqs}
Let $\xi : G^* \to G$ be an inner twist with $G^*$ quasi-split and let $M \subset G$ be a twisted Levi subgroup. After conjugation by $G^*(F^s)$ we can arrange that the subgroup $M^* = \xi^{-1}(M)$ is defined over $F$ (hence is a twisted Levi subgroup of $G^*$) and quasi-split. The restriction $\xi : M^* \to M$ is an inner twist.
\end{lem}
\begin{proof}
Let $T_M \subset M$ be a maximal torus, and let $Z_M \subset T_M$ be the connected center of $M$. Then $Z_M$ is a subtorus of $T_M$ defined over $F$ and $M$ is the centralizer of $Z_M$. The torus $T_M$ transfers to $G^*$, i.e. after conjugating $\xi$ by an element of $G^*(F^s)$ the preimage $T_M^1 = \xi^{-1}(T_M)$ is defined over $F$ and $\xi$ restricts to an isomorphism $T_M^1 \to T_M$ defined over $F$. In particular $\xi^{-1}\sigma(\xi)$ is realized by conjugation by an element of $T_M^1$. Transport $Z_M$ by $\xi$ and take its centralizer in $G^*$, calling it $M^1$. Then $M^1$ is defined over $F$ and $\xi : M^1 \to M$ is an inner twist. Let $t^1 \in Z_{M^1}(F)$ be a regular element, so that $M^1$ is the centralizer of $t^1$. According to \cite[Lemma 3.3]{Kot82} there exists a stable conjugate $t^*$ of $t^1$ whose centralizer in $G^*$ is quasi-split. Let $M^*$ be the centralizer of $t^*$ and apply \cite[Lemma 3.2]{Kot82}.
\end{proof}

\begin{fct} In the setting of above lemma, $\xi$ induces an isomorphism $M^*_\tx{ab} \to M_\tx{ab}$ defined over $F$ that identifies $R(M^*_\tx{ab},G^*)$ with $R(M_\tx{ab},G)$. 
\end{fct}

We now proceed to study $L$-embeddings $^LM \to {^LG}$. Lemma \ref{lem:tlqs} allows us to assume that $G$ and $M$ are quasi-split. Let $(T_M,B_M)$ and $(T_G,B_G)$ be Borel pairs in $M$ and $G$ respectively. Choose $g \in G$ s.t. $T_M = gT_Gg^{-1}$ and $B_M \subset gB_Gg^{-1}$. Then the simple roots $\Delta(T_M,B_M)$ are a subset of $g\Delta(T_G,B_G)g^{-1}$. Let $\sigma_{M,G} \in \Omega(T_G,G)$ be the element $g^{-1}\sigma(g)$.

Consider a pinned dual group $(\hat G,\hat T,\hat B,\{X_{\alpha^\vee}\})$ for $G$. There exists a unique standard Levi subgroup $\hat M$ of $\hat G$ in the conjugacy class determined by $M$. Its set of simple roots $\Delta(\hat T,\hat M)$ is the subset dual to $g^{-1}\Delta(T_M,B_M)g \subset \Delta(T_G,G)=\Delta(\hat T,\hat G)^\vee$. For $\sigma \in \Gamma$ let $\sigma_G$ denote the pinned action of $\sigma$ on $\hat G$. The set $\Delta(\hat T,\hat M)$, and hence also the group $\hat M$, will in general not be $\sigma_G$-stable. Using the identification $\Omega(\hat T,\hat G)=\Omega(T_G,G)$ we obtain the element $\sigma_{M,G} \in \Omega(\hat T,\hat G)$. The set of simple roots $\Delta(\hat T,\hat G)$ is no longer stable under the twisted action $\sigma_{M,G} \rtimes \sigma_G$, but its subset $\Delta(\hat T,\hat M)$ is. Thus $\hat M$ is stable under the twisted action $n(\sigma_{M,G}) \rtimes \sigma_G$, where $n(\dots)$ is the Tits lift relative to the chosen pinning. By \cite[11.2.11]{Spr81} this action preserves the pinning of $\hat M$ inherited from that of $\hat G$ and hence would give an $L$-action if the map $\sigma \mapsto n(\sigma_{M,G}) \rtimes \sigma_G$ were multiplicative. Thus we are looking for a 1-cochain $q$ of $\Gamma$ or $W$, valued in $Z(\hat M)$, so that $q(\sigma)n(\sigma_{M,G}) \rtimes \sigma$ is multiplicative. Since $q$ takes values in $Z(\hat M)$ it acts trivially on the pinning of $\hat M$, so $q(\sigma)n(\sigma_{M,G}) \rtimes \sigma$ still preserves that pinning. Being in addition multiplicative, we see that
\begin{equation} \label{eq:tlemb}
\hat M \rtimes W_F \to \hat G \rtimes W_F,\qquad (m,w) \mapsto m \cdot q(w)n(\sigma_{M,G}) \rtimes \sigma_G 	
\end{equation}
is an $L$-embedding, where $\sigma \in \Gamma$ is the image of $w$. 

We will now show that $q$ can be constructed explicitly by choosing $\chi$-data for $R(M_\tx{ab},G)$. The torus dual to $M_\tx{ab}$ is the complex torus $Z(\hat M)^\circ$ equipped with the action of $\Gamma$ where $\sigma \in \Gamma$ acts as $\sigma_{M,G} \rtimes \sigma_G$. Via the isomorphism dual to $\tx{Ad}(g) : T_G \to T_M$ and the identification $\hat T_G = \hat T$, the set $R(M_\tx{ab},G)$ is identified with the set of $\Omega(\hat T,\hat M)$-orbits in $R^\vee(\hat T,\hat G) \sm R^\vee(\hat T,\hat M)$.

The failure of multiplicativity of $\sigma \mapsto n(\sigma_{M,G}) \rtimes \sigma_G$ is measured by the Tits cocycle $t_p(\sigma,\tau) \in \hat T$:
\[ n(\sigma_{M,G}) \rtimes \sigma_G \cdot n(\tau_{M,G}) \rtimes \tau_G = t_p(\sigma,\tau)\cdot n((\sigma\tau)_{M,G}) \rtimes (\sigma\tau)_G, \] 
given by $t_p(\sigma,\tau) = \lambda_{\sigma,\tau}(-1)$, where $\lambda_{\sigma,\tau}$ is the sum of elements of  $R^\vee(\hat T,\hat G)=R(T_G,G)$ in the set
\[ \Lambda_{\sigma,\tau} = \{ \alpha>0,(\sigma_{M,G}\sigma_G)^{-1}\alpha<0,(\sigma_{M,G}\sigma_G\tau_{M,G}\tau_G)^{-1}\alpha>0\}, \]
where $\alpha>0$ is taken with respect to the Borel subgroup $\hat B$. Let $p : R^\vee(\hat T,\hat G) \sm R^\vee(\hat T,\hat M) \to \{\pm 1\}$ denote this gauge, and let $p' : R(M_\tx{ab},G) \to \{\pm 1\}$ be the induced gauge as in Lemma \ref{lem:rmab}, part 4.

\begin{lem} \label{lem:lambdalevi}
$\lambda_{\sigma,\tau}$ is equal to the sum of $\bar\alpha'$ for all $\alpha'$ belonging to the following subset
\[ \Lambda'_{\sigma,\mu} = \{ \alpha'>0,(\sigma_{M,G}\sigma_G)^{-1}\alpha'<0,(\sigma_{M,G}\sigma_G\tau_{M,G}\tau_G)^{-1}\alpha'>0\} \]
of $R(M_\tx{ab},G)$, where $\alpha'>0$ means $p'(\alpha')>0$, and $\bar\alpha' \in X^*(M_\tx{ab})=X_*(Z(\hat M)^\circ)$ denotes the image of $\alpha'$ under the map $R(M_\tx{ab},G) \to X^*(M_\tx{ab})$ of Definition \ref{dfn:rmab}.
\end{lem}
\begin{proof}
The action of $\sigma_{M,G}\sigma_G$ preserves the root system $R(\hat T,\hat M)$ and the set of positive roots in it. Therefore the set $\Lambda_{\sigma,\tau}$ lies in $R(\hat T,\hat G) \rm R(\hat T,\hat M)$ and is invariant under $\Omega(\hat T,\hat M)$. This means that the sum of the elements of $\Lambda_{\sigma,\tau}$ can be computed by first summing the elements of each individual $\Omega(\hat T,\hat M)$-orbit, and then summing the contributions of those orbits.
\end{proof}

\begin{cor} \label{cor:lambdalevi}
The Tits cocycle $t_p$ takes values in $Z(\hat M)^\circ$. Given $\chi$-data $\chi'$ for $R(M_\tx{ab},G)$, the differential of the 1-cochain $r_{p',\chi'}$ equals $t_p$. Therefore, $q=r_{p',\chi'}$ makes \eqref{eq:tlemb} into an $L$-embedding.
\end{cor}
\begin{proof}
Immediate from Lemmas \ref{lem:lambdalevi} and \ref{lem:lparpm}.
\end{proof}

\begin{rem} \label{rem:inf1}
Let $R(Z(M)^\circ,G) \subset X^*(Z(M)^\circ)$ be the set of non-zero weights for the action of $Z(M)^\circ$ on the Lie-algebra of $G$. It is a finite admissible $\Sigma$-set and equals, for any maximal torus $S \subset M$, the image of $R(S,G) \sm R(S,M) \subset X^*(S)$ under the restriction map $X^*(S) \to X^*(Z(M)^\circ)$. This restriction map is $\Omega(S,M)$-invariant and therefore we obtain a surjective $\Sigma$-equivariant map $R(M_\tx{ab},G) \to R(Z(M)^\circ,G)$. This map need not be injective, as we will see in \S\ref{sub:exa}. A set of $\chi$-data for $R(Z(M)^\circ,G)$ can be inflated to a set of $\chi$-data for $R(M_\tx{ab},G)$ as in Definition \ref{dfn:inf}. In this way, a set of $\chi$-data for $R(Z(M)^\circ,G)$ leads, via Corollary \ref{cor:lambdalevi}, to an $L$-embedding $^LM \to {^LG}$.
\end{rem}

\subsection{Factoring the $L$-embedding of a maximal torus through a twisted Levi subgroup}

We continue with a twisted Levi subgroup $M \subset G$ and a maximal torus $S \subset M$. Consider given a set of $\chi$-data $\chi'$ for $R(M_\tx{ab},G)$. We inflate it, as in Definition \ref{dfn:inf}, to a set of $\chi$-data $\chi=\tx{inf}\,\chi'$ for $R(S,G) \sm R(S,M)$ via the natural surjection $R(S,G) \sm R(S,M) \to R(M_\tx{ab},G)$. Choose a set of $\chi$-data for $R(S,M)$, thereby obtaining a set of $\chi$-data for $R(S,G)$. Let $^Lj_{S,G} : {^LS} \to {^LG}$, $^Lj_{S,M} : {^LS} \to {^LM}$, and $^Lj_{M,G} : {^LM} \to {^LG}$ be the $L$-embeddings obtained from these sets of $\chi$-data. Each $L$-embedding is well-defined up to conjugation (by $\hat M$ or $\hat G$ respectively, depending on its target).

\begin{pro} \label{pro:lembfact}
Up to conjugation by $\hat G$ we have
\[ {^Lj_{M,G}} \circ {^Lj_{S,M}} = {^Lj_{S,G}}. \]
\end{pro}
\begin{proof}
Since $S$ transfers to the quasi-split form of $M$ we reduce by Lemma \ref{lem:tlqs} to the case that both $G$ and $M$ are quasi-split. Choose Borel pairs $(T_M,B_M)$ and $(T_G,B_G)$, an element $g \in G$ with $\tx{Ad}(g)T_G = T_M$ and $\tx{Ad}(g)B_G \supset B_M$, and an element $m \in M$ with $\tx{Ad}(m)T_M = S$. This gives elements $\sigma_{S,G}=g^{-1}m^{-1}\sigma(m)\sigma(g) \in \Omega(T_G,G)$,  $\sigma_{S,M}=m^{-1}\sigma(m) \in \Omega(T_M,M)$, and $\sigma_{M,G}=g^{-1}\sigma(g) \in \Omega(T_G,G)$. Thus $\sigma_{S,G}=g^{-1}\sigma_{S,M}g \cdot \sigma_{M,G}$. We identify $\sigma_{S,M}$ with $g^{-1}\sigma_{S,M}g \in \Omega(T_G,G)$ and then have $\sigma_{S,G}=\sigma_{S,M}\cdot \sigma_{M,G}$. As in \S\ref{sub:tlem} we fix a pinning $(\hat T,\hat B,\{X_{\alpha^\vee}\})$ for $\hat G$ and let $\hat M$ be the unique standard Levi subgroup in the $\hat G$-conjugacy class determined by $M$.

Write ${^Lj_{M,G}}(w)=r_{M,G}(w)n(\sigma_{M,G}) \rtimes \sigma_G$, ${^Lj_{S,G}}(w)=r_{S,G}(w)n(\sigma_{S,G}) \rtimes \sigma_G$, and ${^Lj_{S,M}}(w)=r_{S,M}(w)n(\sigma_{S,M}) \rtimes \sigma_M$. Here $r_{S,G}$ and $r_{M,G}$ are the 1-cochains $r_{p,\chi}$ of Lemma \ref{lem:lparpm} for the gauge $p$ given by $\hat B$ and the fixed $\chi$-data for $R(S,G)$ and $R(S,M)$, respectively, while $r_{M,G}$ is given by the same Lemma, but for the $\chi$-data for $R(M_\tx{ab},G)$, the map $R(M_\tx{ab},G) \to X^*(M_\tx{ab})$, and the gauge $p'$ on $R(M_\tx{ab},G)$ induced by $p$ as in Lemma \ref{lem:rmab}.

We have ${^Lj_{M,G}}\circ {^Lj_{S,M}}(w)=r_{S,M}(w)n(\sigma_{S,M})r_{M,G}(w)n(\sigma_{M,G})\rtimes\sigma_G$. Recalling that $r_{M,G}$ takes values in $Z(\hat M)^\circ$ we can move it past $n(\sigma_{S,M})$. Furthermore, since $\sigma_{S,M}$ preserves the set of positive roots in $R(T_G,G) \sm R(T_G,M^g)$, while $\sigma_{M,G}\sigma_G$ preserves the set of positive roots in $R(T_G,M^g)$, with $M^g=g^{-1}Mg$, \cite[Lemma 2.1.A]{LS87} implies that $n(\sigma_{S,M})n(\sigma_{M,G})=n(\sigma_{S,M}\sigma_{M,G})=n(\sigma_{S,G})$. We conclude ${^Lj_{M,G}}\circ {^Lj_{S,M}}(w)=r_{S,M}(w)r_{M,G}(w)n(\sigma_{S,G})\rtimes\sigma_G$. The two $L$-embeddings ${^Lj_{M,G}}\circ {^Lj_{S,M}}(w)$ and $^Lj_{S,G}$ thus differ by the 1-cocycle $r_{S,M}(w)r_{M,G}(w)r_{S,G}(w)^{-1}$ and our goal is to show that it is cohomologically trivial.

Recall that $r_{S,G}(w)$ is a product taken over the $\Sigma$-orbits in $R(S,G)$. The 1-cochain $r_{S,M}(w)$ is a sub-product, taken over those $\Sigma$-orbits that lie in $R(S,M)$. Therefore $r_{S,G}(w)r_{S,M}(w)^{-1}$ equals the 1-cochain of Lemma \ref{lem:lparpm} for the torus $S$, the finite admissible $\Sigma$-set $R(S,G) \sm R(S,M) \subset X^*(S)$, and the $\chi$-data that by construction is inflated from the $\chi$-data for $R(M_\tx{ab},G)$. On the other hand, $r_{M,G}(w)$ is the analogous 1-cochain for the torus $M_\tx{ab}$, the finite admissible $\Sigma$-set $R(M_\tx{ab},G)$, the $\Sigma$-equivariant map $R(M_\tx{ab},G) \to X^*(M_\tx{ab})$, and the chosen $\chi$-data for $R(M_\tx{ab},G)$. The gauges on $R(M_\tx{ab},G)$ and $R(S,G) \sm R(S,M)$ are compatible with the natural projection map $f : R(S,G) \sm R(S,M) \to R(M_\tx{ab},G)$. This natural projection map satisfies \eqref{eq:func_st}. Proposition \ref{pro:llcfunc} applied to the epimorphism $S \to M_\tx{ab}$ asserts that $r_{M,G}$ is cohomologous to $r_{S,G}\cdot r_{S,M}^{-1}$.
\end{proof}

\subsection{The canonical double cover of a twisted Levi subgroup}

Let $M \subset G$ be a twisted Levi subgroup. The finite admissible $\Sigma$-set $R(M_\tx{ab},G)$ and the $\Sigma$-equivariant map $R(M_\tx{ab},G) \to X^*(M_\tx{ab})$ of Definition \ref{dfn:rmab} lead to the double cover $M_\tx{ab}(F)_\pm$ of $M_\tx{ab}(F)$.

\begin{dfn}
\begin{enumerate}
	\item The double cover $M(F)_\pm$ is the pull-back of $M(F) \to M_\tx{ab}(F) \from M_\tx{ab}(F)_\pm$.
	\item The $L$-group $^LM_\pm$ of $M(F)_\pm$ is the push-out of $\hat M \from Z(\hat M)^\circ \to {^LM}_{\tx{ab},\pm}$.
\end{enumerate}
\end{dfn}

\begin{lem} Let $S \subset M$ be a maximal torus. Let $S(F)_\pm$ be the double cover for the finite set $R(S,G) \sm R(S,M) \subset X^*(S)$. The inclusion $S(F) \to M(F)$ lifts naturally to a homomorphism $S(F)_\pm \to M(F)_\pm$.
\end{lem}
\begin{proof}
By definition the diagram
\[ \xymatrix{
	R(M_\tx{ab},G)\ar[r]\ar[d]&R(S,G) \sm R(S,M)\ar[d]\\
	X^*(M_\tx{ab})\ar[r]&X^*(S)
	}
\]
satisfies \eqref{eq:func_st}. Constructions \ref{cns:func_t} and \ref{cns:homcover} produces a lifting $S(F)_\pm \to M_\tx{ab}(F)_\pm$ of the natural map $S(F) \to M_\tx{ab}(F)$. The map $S(F)_\pm \to M_\tx{ab}(F)_\pm \to M_\tx{ab}(F)$ is equal to the map $S(F)_\pm \to S(F) \to M(F) \to M_\tx{ab}(F)$, so these two maps induce the desired map $S(F)_\pm \to M(F)_\pm$.
\end{proof}

\begin{lem} \label{lem:lembfact}
Let $S \subset M$ be a maximal torus. Let $S(F)_\pm$ be the double cover for the finite set $R(S,G) \subset X^*(S)$.
\begin{enumerate}
	\item The canonical $\hat G$-conjugacy class of embeddings $\hat M \to \hat G$ extends canonically to $\hat G$-conjugacy class of $L$-embeddings $^LM_\pm \to {^LG}$. 
	\item The canonical $\hat M$-conjugacy class of embeddings $\hat S \to \hat M$ extends canonically to a $\hat M$-conjugacy class of $L$-embeddings $^LS_\pm \to {^LM_\pm}$. 
	\item The $\hat G$-conjugacy class of $L$-embeddings $^LS_\pm \to {^LG}$ factors as the composition $^LS_\pm \to {^LM_\pm} \to {^LG}$ of the two $L$-embeddings described in the preceding two points.
\end{enumerate}
\end{lem}
\begin{proof}
In this proof we will use the same set-up as in the proof of Proposition \ref{pro:lembfact}, so we refer the reader to the first paragraph of that proof rather than repeating it.

Let $t_{M,G}$, $t_{S,M}$, and $t_{S,G}$ be the Tits cocycles for the maps sending $\sigma$ to $n(\sigma_{M,G})$, $n(\sigma_{S,M})$, and $n(\sigma_{S,G})$, respectively. The embeddings $^LM_\pm \to {^LG}$, $^LS_\pm \to {^LM_\pm}$, and $^LS_\pm \to {^LG}$ are then given by 
\[ \hat M \boxtimes_{t_{M,G}} \Gamma \to \hat G \rtimes \Gamma, m \boxtimes \sigma \mapsto m \cdot n(\sigma_{M,G}) \rtimes \sigma_G, \]
\[ \hat S \boxtimes_{t_{S,G}} \Gamma \to \hat M \boxtimes_{t_{M,G}} \Gamma, s \boxtimes \sigma \mapsto s \cdot n(\sigma_{S,M}) \boxtimes \sigma_G, \]
\[ \hat S \boxtimes_{t_{S,G}} \Gamma \to \hat G \rtimes \Gamma, s \boxtimes \sigma \mapsto s \cdot n(\sigma_{S,G}) \rtimes \sigma_G. \]
The first formula is the same as \eqref{eq:tlemb}, but with the 1-cochain $q$ removed. The other two are as in \eqref{eq:lembt}. 

The fact that the first of these is indeed a group homomorphism follows from the computation $t_{M,G}(\sigma,\tau)=(-1)^{\lambda_{\sigma,\tau}}$ with $\lambda_{\sigma,\tau}$ as in \S\ref{sub:tlem} and Lemma \ref{lem:lambdalevi}. This embedding depends only on the choice of pinning of $\hat G$, and since two pinnings are $\hat G$-conjugate, so are the resulting $L$-embeddings. The fact that the third is a group homomorphism was already discussed in \S\ref{sub:canemb}. The fact that the second is a group homomorphism follows from the same argument combined with the observation $t_{S,G}=t_{S,M} \cdot t_{M,G}$. This observation follows from the equation $n(\sigma_{S,M}) \cdot n(\sigma_{M,G}) = n(\sigma_{S,G})$ also recorded in the proof of Proposition \ref{pro:lembfact}. That latter equation also implies that the composition of the first two embeddings equals the third embedding. 
\end{proof}

\begin{fct} \label{fct:levisp}
The Weil forms of $^LM_\pm$ and $^LM$ are non-canonically isomorphic. More precisely, let $(\chi_{\alpha'})$ be a set of $\chi$-data for $R(M_\tx{ab},G)$ and let $r_{M,G}$ be the corresponding $1$-cochain as in \S\ref{sub:tlem}. Then
\[ \hat M \rtimes W \to \hat M \boxtimes_{t_{M,G}} W,\qquad m \rtimes w \mapsto m \cdot r_{M,G}(w) \boxtimes w \]
is an $L$-isomorphism.
\end{fct}

\begin{rem} Assume given a bijection between $L$\-parameters $\varphi : W' \to {^LM}$ and $L$-packets for $M(F)$. Assume further this bijection is compatible with twists by co-central characters. Then one obtains a bijection between $L$\-parameters $\varphi : W' \to {^LM_\pm}$ and $L$-packets of genuine representations of $M(F)_\pm$ as follows: Choose a set of $\chi$-data for $R(M_\tx{ab},G)$. We obtain a genuine character $\chi_M$ of $M_\tx{ab}(F)_\pm$ which we can pull back to a genuine character of $M(F)_\pm$. Via the isomorphism of Fact \ref{fct:levisp} we obtain from a given $\varphi : W' \to {^LM_\pm}$ an $L$-packet $\Pi_\varphi$ on $M(F)$. Tensoring each member of $\Pi_\varphi$ with $\chi_M$ we obtain a genuine representation of $M(F)_\pm$. The set of these forms the $L$-packet $\Pi_\varphi^\pm$ on $M(F)_\pm$. The compatibility of the assumed correspondence for $M$ with co-central characters implies that $\Pi_\varphi^\pm$ depends only on $\varphi$ and not on the chosen $\chi$-data.
\end{rem}

\begin{rem} Assume that there is a correspondence assigning to a pair $(S,\theta)$ of a maximal torus $S \subset G$ and a genuine character $\theta : S(F)_\pm \to \C^\times$, subject to certain conditions, a representation $\pi$ of $G(F)$, as suggested in \S\ref{sub:char}. Assume further that the same holds for $M$. The factorization of $L$-embeddings of Lemma \ref{lem:lembfact} suggests that there should be a parallel correspondence that assings to a pair $(S,\theta)$, where $S \subset G$ is a maximal torus that lies in $M$ and $\theta$ is a genuine character of $S(F)_\pm$, a genuine representation of the double dover $M(F)_\pm$.

This is easy to see formally. Since we are now dealing with different $\Sigma$-invariant finite sets we will use them as subscripts to distinguish the various double covers. Thus we have a maximal torus $S \subset M$ and a genuine character $\theta : S(F)_{R(S,G)} \to \C^\times$ and want to assign to them a genuine representation $\pi_\pm$ of $M(F)_{R(M_\tx{ab},G)}$. Such representations are classified as follows. Choose a genuine character $\chi : M_\tx{ab}(F)_{R(M_\tx{ab},G)} \to \C^\times$ and inflate it to $M(F)_{R(M_\tx{ab},G)}$. Then $\pi := \pi_\pm \otimes \chi^{-1}$ is a representation inflated from $M(F)$. The pair $(\pi,\chi)$ determines $\pi_\pm$, but is not determined by it. Another pair determining the same $\pi_\pm$ is of the form $(\pi\otimes\delta,\chi\delta^{-1})$ for a character $\delta : M_\tx{ab}(F) \to \C^\times$. Now say that $\pi$ corresponds to $(S,\theta_0)$ with $\theta_0 : S(F)_{R(S,M)} \to \C^\times$ genuine. Then $\pi_\pm$ is determined, and conversely determines, the torus $S$ and a geniune character of the pull-back of
\[ S(F)_{R(S,M)} \to S(F) \to M_\tx{ab}(F) \from M_\tx{ab}(F)_{R(M_\tx{ab},G)}. \]
This pull-back is an extension of $S(F)$ by $\{\pm 1\}^2$ and a genuine character is one that maps both $(-1,1)$ and $(1,-1)$ to $-1$. According to Fact \ref{fct:pullbackfunc} this pull-back is canonically isomorphic to the pull-back of
\[ S(F)_{R(S,M)} \to S(F) \from S(F)_{R(S,G) \sm R(S,M)}, \]
whose quotient by the diagonal copy of $\{\pm 1\}$ equals $S(F)_{R(S,G)}$. 
\end{rem}

\subsection{Examples} \label{sub:exa}

We will now show by example that the maps $R(M_\tx{ab},G) \to X^*(M_\tx{ab})$ and $R(M_\tx{ab},G) \to R(Z(M)^\circ,G)$ need not be injective.

Consider the group $G=\tx{PSp}_4$. Note that $Z(M)$ is connected. The root system $R(S,G) \subset X^*(S)=\Z^2$ consists of $\{e_1,e_2,e_2-e_1,e_2+e_1\}$ and their negatives. The simple roots for this choice of positive roots are $\{e_1,e_2-e_1\}$.

Let first $R(S,M)=\{e_1,-e_1\}$. Then $X^*(Z(M))=\Z$ and the map $X^*(S) \to X^*(Z(M))$ sends $e_1$ to $0$ and $e_2$ to $1$. In particular, $R(Z(M))=\{1,-1\}$. The action of $\Omega(S,M)$ on the positive part of $R(S,G) \sm R(S,M)$ has the orbits $\{e_2\}$ and $\{e_2-e_1,e_2+e_1\}$. These orbits and their negatives are the elements of $R(M_\tx{ab},G)$ and we see that the map $R(M_\tx{ab},G) \to R(Z(M),G)$ sends both positive orbits to $1$, hence is not injective. On the other hand, $X^*(M_\tx{ab})=\<e_2\>$ and the map $R(M_\tx{ab},G) \to X^*(M_\tx{ab})$ sends $\{e_2\}$ to $e_2$ and $\{e_2-e_1,e_2+e_1\}$ to $2e_2$, and is thus injective.

Let now $R(S,M)=\{e_2-e_1,e_1-e_2\}$. Then $X^*(Z(M))=\Z$ and the map $X^*(S) \to X^*(Z(M))$ sends both $e_1$ and $e_2$ to $1$. Now $R(Z(M))=\{1,2,-1,-2\}$. The action of $\Omega(S,M)$ on the positive part of $R(S,G) \sm R(S,M)$ has the orbits $\{e_1,e_2\}$ and $\{e_1+e_2\}$. These orbits and their negatives are the elements of $R(M_\tx{ab},G)$ and we see that the map $R(M_\tx{ab},G) \to R(Z(M),G)$ sends $\{e_1,e_2\}$ to $1$ and $\{e_1+e_2\}$ to $2$, thus it is injective. On the other hand, $X^*(M_\tx{ab})=\<e_1+e_2\>$ and the map $R(M_\tx{ab},G) \to X^*(M_\tx{ab})$ sends both $\{e_1,e_2\}$ and $\{e_1+e_2\}$ to $e_1+e_2$, and is thus not injective.

\begin{appendices}

\section{Frobenius reciprocity} \label{app:frob}

Let $\Delta \subset \Gamma$ be two finite groups, $X$ a $\Z[\Delta]$-module, and $Y$ and $\Z[\Gamma]$-module. We consider
\[ \tx{Ind}_\Delta^\Gamma X = \{f : \Gamma \to X| f(\delta\gamma)=\delta f(\gamma)\}. \]
We have the $\Delta$-equivariant maps
\[ X \to \tx{Ind}_\Delta^\Gamma X \to X, \]
the first sending $x \in X$ to
\[ f_x(\gamma) = \begin{cases}
\gamma x&,\gamma \in\Delta\\
0&,\gamma \notin \Delta
\end{cases}\]
and the second sending $f$ to $f(1)$. Their composition is the identity. In particular, the first is injective and the second surjective.

The first Frobenius reciprocity isomorphism is
\[ \tx{Hom}_\Gamma(Y,\tx{Ind}_\Delta^\Gamma X) \leftrightarrow \tx{Hom}_\Delta(Y,X). \]
An element $\Phi$ on the left is related to an element $\varphi$ on the right by $\varphi(y)=\Phi(y)(1)$ and $\Phi(y)(\gamma)=\varphi(\gamma y)$.

The second Frobenius reciprocity isomorphism is
\[ \tx{Hom}_\Gamma(\tx{Ind}_\Delta^\Gamma X,Y) \leftrightarrow \tx{Hom}_\Delta(X,Y). \]
An element $\Phi$ on the left is related to an element $\varphi$ on the right by $\varphi(x)=\Phi(f_x)$ and $\Phi(f)=\sum_{\tau \in \Delta \lmod \Gamma} \tau^{-1}\varphi(f(\tau))$. Note that each summand depends only on the coset $\tau$, not an individual representative of it, and is thus well-defined.

In checking that the maps of the second Frobenius reciprocity isomorphism are inverse to each other, the following identity is used: $\gamma^{-1}f_{f(\gamma)} = f \cdot \tb{1}_{\Delta  \gamma}$.

There is a different presentation for $\tx{Ind}_\Delta^\Gamma X$. Choose a set of representatives $(\gamma_1,\dots,\gamma_k)$ for $\Delta \lmod \Gamma$. Then
\[ \tx{Ind}_\Delta^\Gamma X \to X^{[\Gamma:\Delta]},\qquad f \mapsto (f(\gamma_1),\dots,f(\gamma_k)) \]
is an isomorphism of $\Z$-modules. To see what $\Gamma$-action on the target makes it equivariant, introduce the function
\[ \{1,\dots,k\} \times \Gamma \to \Delta \times \{1,\dots,k\},\qquad (i,\gamma) \mapsto (\gamma_i(\gamma),i_\gamma),\qquad \gamma_i \cdot \gamma = \gamma_i(\gamma)\cdot \gamma_{i_{\gamma}} \]
and then see that the action of $\Gamma$ on the right is by
\[ \gamma \cdot (x_1,\dots,x_k) = (\gamma_1(\gamma) \cdot x_{1_\gamma},\dots,\gamma_k(\gamma) \cdot x_{k_\gamma}).\]

\section{Comparison with groups of type $L$} \label{app:glt}

In this appendix we review and reinterpret material from a set of informal notes called ``groups of type $L$'' by Gross, that were created in 2008 and circulated privately. The aim is to generalize the construction of Adams-Vogan of the ``universal extension'' of a real torus, its various $E$-groups and the associated local Langlands correspondences \cite[\S5]{AV92}. We will see that this works under a rather strict assumption on the torus and that some important special cases which satisfy this assumption can be used to shed new light on the computations of Langlands-Shelstad that we have used in this paper.

\subsection{Groups of type $L$} \label{sub:gtl}

We begin by reviewing the material from Gross' notes. This material has also been reviewed in \cite[\S5]{AdRo16}. 

Consider an algebraic torus $S$ defined over $F$ with complex dual torus $\hat S$. Let $E/F$ be the splitting extension of $S$. In this subsection we make the following 
\begin{asm} \label{asm:norm}
The norm map $S(E) \to S(F)$ is surjective.	
\end{asm}
This leads to the exact sequence
\begin{equation} \label{eq:gtlcover}
1 \to \hat H^{-1}(\Gamma_{E/F},S(E)) \to \frac{S(E)}{IS(E)} \to S(F) \to 1, 	
\end{equation}
where $IS(E)=\<x^{-1}\sigma(x)|x \in S(E),\sigma \in \Gamma_{E/F}\>$ is the augmentation submodule of $S(E)$ for the action of $\Gamma_{E/F}$. 

Gross' observation is that characters of $S(E)/IS(E)$ are parameterized by ``Langlands parameters'' valued in the various extensions of $\Gamma_{E/F}$ by $\hat S$. More precisely, we have the following

\begin{pro} \label{pro:gtl}
\begin{enumerate}
	\item There is a natural bijection between $H^2(\Gamma_{E/F},\hat S)$ and the Pontryagin dual of the finite abelian group $\hat H^{-1}(\Gamma_{E/F},S(E))$.
	\item Let $c : \hat H^{-1}(\Gamma_{E/F},S(E)) \to \C^\times$ be a character corresponding to an element of $H^2(\Gamma_{E/F},\hat S)$ and let 
	\begin{equation} \label{eq:gtl}
	1 \to \hat S \to \mc{S} \to \Gamma_{E/F} \to 1.
	\end{equation}
	be an extension corresponding to that element. There is a natural bijection between the set of $\hat S$-conjugacy classes of continuous $L$\-homomorphisms $W \to \mc{S}$ and the set of $c$-genuine continuous characters of $S(E)/IS(E)$.
\end{enumerate}
\end{pro}

\begin{rem} The extension \eqref{eq:gtl} is is what Gross calls a \emph{group of type $L$}, and what Adams-Vogan call an $E$-group in \cite[Definition 5.9]{AV92}. 
\end{rem}

\begin{proof}
We first claim that the sequence
\begin{equation} \label{eq:gtlir}
0 \to H^1(W_F,\hat S) \to H^1(W_E,\hat S)^{\Gamma_{E/F}} \to H^2(\Gamma_{E/F},\hat S) \to 0	
\end{equation}
is exact, where the first map is restriction, and the second is transgression. Indeed, this is a piece of the inflation\-restriction exact sequence for $W_E \subset W_F$, which gives exactness in the middle. Exactness on the right is due to \cite[Lemma 4]{Lan79}. Exactness on the left is due inflation being dual to the norm $S(E) \to S(F)$, which is assumed surjective.

Since the first map of \eqref{eq:gtlir} is dual to the second map in \eqref{eq:gtlcover}, we conclude that the third term in \eqref{eq:gtlir} is dual to the first term in \eqref{eq:gtlcover}. This proves the first part of the proposition.

The second part of the proposition follows immediately from the fact that the fiber of the second map in \eqref{eq:gtlir} over a given class in $H^2(\Gamma_{E/F},\hat S)$ is in natural bijection with the set of $L$\-homomorphisms $W_F \to \mc{S}$, where $\mc{S}$ is the extension \eqref{eq:gtl} corresponding to that class. We can be more explicit: Given a continuous $\Gamma_{E/F}$-equivariant homomorphism $\varphi_E : W_E \to \hat S$ choose arbitrarily an extension of $\varphi_E$ to a continuous map $\varphi_F : W_F \to \hat S$ s.t. $\varphi_F(ww')=\varphi_E(w)\varphi_F(w')$ for $w \in W_E$ and $w' \in W_F$.  Let $z \in Z^2(\Gamma_{E/F},\hat S)$ be its differential. It represents the image of $\varphi_E$ under the transgression map. Then $w \mapsto \varphi_F(w) \boxtimes w$ is a homomorphism from $W_F$ to the twisted product $\hat S \boxtimes_z \Gamma_{E/F}$, which is an extension of $\Gamma_{E/F}$ by $\hat S$ in the isomorphism class specified by the cohomology class of $z$. Regarding well-definedness, please note the following remark.
\end{proof}

\begin{rem} A subtle but important point emphasized in \cite{AV92} is that the $L$-group of a torus $S$ is not simply a split extension of $\hat S$ by $\Gamma$, but rather such an extension equipped with an $\hat S$-conjugacy class of splittings. Similarly, an $E$-group in their terminology is not just an extension of $\hat S$ by $\Gamma$, but one equipped with a $\hat S$-conjugacy class of set-theoretic splittings. Our assumption on the surjectivity of the norm $S(E) \to S(F)$ is equivalent to the vanishing of $\hat H^0(\Gamma_{E/F},S(E))$, which by \cite{Tate66} is isomorphic to $\hat H^{-2}(\Gamma_{E/F},X_*(S))$, which is dual to $H^2(\Gamma_{E/F},X_*(\hat S))$, which is isomorphic to $H^1(\Gamma_{E/F},\hat S)$. Therefore, there is a unique $\hat S$-conjugacy class of set-theoretic splittings for any extension of $\Gamma_{E/F}$ by $\hat S$. Equivalently, specifying an isomorphism class of an extension is the same as specifying the extension itself up to $\hat S$-conjugacy.
\end{rem}

\begin{rem}
There is an alternative way to specify an isomorphism between $H^2(\Gamma_{E/F},\hat S)$ and the dual of $\hat H^{-1}(\Gamma_{E/F},S(E))$, namely by composing the isomorphism $H^2(\Gamma_{E/F},\hat S) \to H^3(\Gamma_{E/F},X_*(\hat S))$ coming from the exponential sequence with the duality between $H^3(\Gamma_{E/F},X_*(\hat S))$ and $\hat H^{-3}(\Gamma_{E/F},X_*(S))$ and the isomorphism $\hat H^{-3}(\Gamma_{E/F},X_*(S)) \to \hat H^{-1}(\Gamma_{E/F},S(E))$ due to \cite{Tate66}. We leave it to the reader to compare this identification with the one of Proposition \ref{pro:gtl}.
\end{rem}

\begin{rem} Consider the case $F=\R$. Given a torus $S$ we have the exponential exact sequence
\[ 0 \to X_*(S) \to X_*(S)\otimes_\Z\C \to X_*(S) \otimes_\Z \C^\times \to 1 \]
which realizes $X_*(S) \otimes_\Z \C=\tx{Lie}(S)$ as the universal cover of $S(\C)$ and $X_*(S)$ as the fundamental group. Adams and Vogan consider \cite[(5.3b)]{AV92}
\[ S(\R)^\sim = \{x \in X_*(S) \otimes_\Z\C| \exp(2\pi i x) \in S(\R) \} / (1-\sigma)X_*(S) \]
which fits into the exact sequence
\[ 0 \to \frac{X_*(S)}{(1-\sigma)X_*(S)} \to S(\R)^\sim \to S(\R) \to 1 \]
that they call ``the universal extension'' of $S(\R)$. Note that we are using the map $\exp(2\pi ix)$, while they are using the map $\exp(x)$, which accounts for the slight disagreement of notation. They show that an extension  
\[ 1 \to \hat S \to \mc{S} \to \Gamma \to 1 \]
equipped with a set-theoretic splitting $s$ gives a character $c_s$ of $X_*(S)/(1-\sigma)X_*(S)$, and $c_s$-genuine characters of $S(\R)^\sim$ are in 1-1 correspondence with $L$-parameters valued in $\mc{S}$.

If $S$ is an anisotropic torus then the norm map $S(\C) \to S(\R)$ is surjective, so Proposition \ref{pro:gtl} can be applied. Furthermore $\hat H^{-1}(\Gamma,S)=H^1(\Gamma,S)=H^{-1}(\Gamma,X_*(S))=H_0(\Gamma,X_*(S))=X_*(S)/(1-\sigma)X_*(S)$, where we have used the periodicity of cohomology for cyclic groups, the Tate isomorphism, and the anisotropicity of $S$. In this way Proposition \ref{pro:gtl} recovers the construction of Adams-Vogan.

If $S$ is split then the norm map $S(\C) \to S(\R)$ is no longer surjective, so Proposition \ref{pro:gtl} does not apply. We have $H^2(\Gamma,\hat S)=0$, so up to equivalence there is no non-trivial extension of $\Gamma$ by $\hat S$. But if we consider an actual such extension equipped with a set-theoretic splitting we obtain not just an element of $H^2(\Gamma,\hat S)$, but an element of $Z^2(\Gamma,\hat S)$. A very special feature of $F=\R$ is that $\Gamma$ is finite cyclic with canonical generator, so there is a canonical element of $Z^2(\Gamma,\Z)$ and cup product with this element gives an isomorphism $Z^0(\Gamma,\hat S) \to Z^2(\Gamma,\hat S)$. Now $Z^0(\Gamma,\hat S)=H^0(\Gamma,\hat S)$ is dual to $H_0(\Gamma,X_*(S))=X_*(S)$ and hence we obtain a character of $X_*(S)$. The universal extension is
\[ 0 \to X_*(S) \to S(\R)^\sim \to S(\R) \to 1. \]
Thus the construction of Adams-Vogan and the construction presented here diverge. 

It is not clear if the construction of Adams-Vogan can be performed in the $p$-adic case. The crucial isomorphism $Z^2(\Gamma,\hat S) \to Z^0(\Gamma,\hat S)$, which relied on the fact that $\Gamma$ is cyclic with canonical generator, is no longer given.
\end{rem}

\begin{exa} \label{exa:nns}
We give here an example of an anisotropic torus over $F=\Q_p$ for which the norm map $S(E) \to S(F)$ is not surjective, where $E/F$ is the splitting extension. 

Consider $F=\Q_p$ with $4|(p-3)$. The unramified quadratic extension of $F$ is generated by a fourth root of unity $\zeta_4$ . Choose a fourth root $\varpi$ of $p$ and let $E_\pm=\Q_p(\sqrt{p})$, $E_+=\Q_p(\varpi)$ and $E=\Q_p(\varpi,\zeta_4)$. Then $E_\pm/F$ is ramified quadratic, $E_+/E_\pm$ is also ramified quadratic, and $E$ is the normal closure of $E_+$. Let $S$ be the restriction of scalars from $E_\pm$ to $F$ of the one-dimensional anisotropic torus defined over $E_\pm$ and split over $E_+$. Then $S$ is anisotropic and its splitting extension is $E/F$.

The map $x/\bar x$ from $E_+$ to $E_+^1$ is surjective and the preimage of $-1$ consists of elements of $E_+$ of odd valuation. No such element is in the image of the norm map $E \to E_+$. Since $S(E)$ is a product of copies of $E^\times$ and the norm map $S(E) \to S(F)$ is made of compositions of the norm map $E \to E_+$ and the map $E_+ \to E_+^1$, we see that $-1 \in E_+^1=S(F)$ is not in the image of the norm map $S(E) \to S(F)$.
\end{exa}

\subsection{One-dimensional anisotropic tori} \label{sub:1d}

In this subsection we consider the one-dimensional anisotropic torus $S$ defined over $F$ and split over a separable quadratic extension $E/F$, i.e. $S=\tx{Res}^1_{E/F}\mb{G}_m$. We have $\hat S=\C^\times_{(-1)}$. The Tate periodicity isomorphism
\[ H^2(\Gamma_{E/F},\C^\times_{(-1)}) = \hat H^0(\Gamma_{E/F},\C^\times_{(-1)}) = \{\pm 1\} \]
shows that there are exactly two isomorphism classes of extensions of $\Gamma_{E/F}$ -- the split extension, i.e. the $L$-group of $S$, and a second extension. In fact, the second extension has a canonical 2-cocycle, namely $z \in Z^2(\Gamma_{E/F},\C^\times_{(-1)})$ specified by $z(\sigma,\sigma)=-1$, where $\sigma \in \Gamma_{E/F}$ is the non-trivial element. Therefore we have the canonical representative $\C^\times_{(-1)} \boxtimes_z \Gamma_{E/F}$ of that isomorphism class.

The torus $S$ satisfies $\hat H^0(\Gamma_{E/F},S(E))=0$, i.e. Assumption \ref{asm:norm}. At the same time we have $\hat H^{-1}(\Gamma_{E/F},S(E))=F^\times/N_{E/F}(E^\times)$. Therefore the extension \eqref{eq:gtlcover} becomes, via push-out under $\kappa : F^\times/N_{E/F}(E^\times) \to \{\pm 1\}$, the double cover
\begin{equation} \label{eq:dcelem}
1 \to \{\pm 1\} \to E^\times/N_{E/F}(E^\times) \to E^1 \to 1, 	
\end{equation}
where $E^1 = \tx{Ker}(N_{E/F})$ and the map $E^\times \to E^1$ is given by $x \mapsto x/\sigma x$. According to Proposition \ref{pro:gtl} there is a canonical bijection between $L$-parameters valued in $\C^\times_{(-1)} \boxtimes_z \Gamma_{E/F}$ and genuine characters of this double cover.

Let us examine this bijection more carefully. A genuine character of this double cover is the same as a character $\theta$ of $E^\times$ that extends $\kappa$. In order to obtain the corresponding $L$\-parameter we follow the proof of Proposition \ref{pro:gtl}. Choose an element $v_1 \in W_{E/F} \sm E^\times$. Note that $v_1^2 \in F^\times \sm N_{E/F}(E^\times)$. Extend $\theta$ to a map $W_{E/F} \to \C^\times$ by $\theta(ev_1)=\theta(e)$ for $e \in E^\times$. The differential of this map sends $(\sigma,\sigma) \in \Gamma_{E/F}$ to $\theta(v_1^2)^{-1}=\kappa(v_1^2)^{-1}=-1$. Therefore the $L$\-parameter for $\theta$ is the map
\begin{equation} \label{eq:par1}
W_{E/F} \to \C^\times \boxtimes_z \Gamma_{E/F},\quad e \mapsto \theta(e)\boxtimes 1,\quad ev_1 \mapsto \theta(e)\boxtimes \sigma.
\end{equation}

\subsection{Induced one-dimensional ansitropic tori} \label{sub:i1dat}

In this subsection we consider the restriction of scalars of a one-dimensional anisotropic torus. Thus let $E_\pm/F$ be a finite separable extension and let $E_+/E_\pm$ be a quadratic separable extension. Let $S=\tx{Res}_{E_\pm/F}\tx{Res}^1_{E_+/E_\pm}\mb{G}_m$. We have $\hat S = \tx{Ind}_{\Gamma_{E_\pm}}^{\Gamma_F} \C^\times_{(-1)}$ and $X^*(S)=\tx{Ind}_{\Gamma_{E_\pm}}^{\Gamma_F}\Z_{(-1)}$. 

The splitting field of $S$ is the normal closure $E$ of $E_+$. The norm map $S(E) \to S(F)$ need not be surjective any more as we saw in Example \ref{exa:nns} so this situation is not covered by \S\ref{sub:gtl}. But we can extrapolate the results of \S\ref{sub:1d}. First, we still have the canonical double cover of $S(F)=E_+^1$, namely $E_+^\times/N_{E_+/E_\pm}(E_+^\times)$, but it is not related any more to $S(E)$. Second, we can define a suitable $L$-group for that double cover and obtain a local correspondence. This is conceptually rather simple -- it just involves Shapiro's lemma, but since we will be dealing with co-chains rather than cohomology classes we must be more careful.

We begin by noting that we still have a canonical class in $H^2(\Gamma,\hat S)$, becase the latter is canonically isomorphic to $H^2(\Gamma_{E_\pm},\C^\times_{(-1)})$ which contains the class inflated from the unique non-trivial class in $H^2(\Gamma_{E_\pm/E_+},\C^\times_{(-1)})$. But there is no canonical 2-cocycle representing it. Indeed, while the Shapiro isomorphism $H^2(\Gamma_{E_\pm},\C^\times_{(-1)}) \to H^2(\Gamma,\hat S)$ is canonical, a presentation for it $Z^2(\Gamma_{E_\pm},\C^\times_{(-1)}) \to Z^2(\Gamma,\hat S)$ on the level of cocycles is not. It depends on a choice, which we now discuss.

There is a canonical element $e \in X^*(S)$: it is the unique element supported on $\Gamma_{E_\pm}$ and sending the identity in the group to $1 \in \Z$. Its $\Gamma$-orbit $R$ is symmetric. The following choices are equivalent:
\begin{enumerate}
	\item A section $s : \Gamma_{E_\pm} \lmod \Gamma_F \to \Gamma_{E_+} \lmod \Gamma_F$ of the natural projection;
	\item a gauge $p : R \to \{\pm 1\}$.
\end{enumerate}

The equivalence is by $p(s(\sigma)^{-1}e)=+1$. Once this choice is made, we compose $s$ by an arbitrary section $\Gamma_{E_+}\lmod \Gamma_F \to \Gamma_F$, define $r : \Gamma_F \to \Gamma_{E_\pm}$ by $\sigma=r(\sigma)s(\sigma)$, and obtain according to \cite[Lemma B.3]{KalGRI} the Shapiro map $Z^2(\Gamma_{E_\pm},\C^\times_{(-1)}) \to Z^2(\Gamma,\hat S)$ as
\[ \tx{Sh}\,z(\sigma,\tau)(a) = r(a)z(r(a)^{-1}r(a\sigma),r(a\sigma)^{-1}r(a\sigma\tau)). \]
Its restriction to $Z^2(\Gamma_{E_\pm/E_+},\C^\times_{(-1)})$ does not depend on the choice of section $\Gamma_{E_+} \lmod \Gamma_F \to \Gamma_F$.

We are particularly interested in the image of the canonical element $z \in Z^2(\Gamma_{E_\pm/E_+},\C^\times_{(-1)})$. It is given by a function $\Gamma_{E_\pm} \lmod \Gamma_F \to \{\pm 1\}$. Such a function is determined by its values at elements $a$ with $r(a) \in \Gamma_{E_+}$, equivalently $p(a^{-1}e) =+1$. Let us adopt the notation $\alpha>0$ for $p(\alpha)=+1$. From above formula we see that $\tx{Sh}\,z(\sigma,\tau)(a)=-1$ if and only if $r(a\sigma) \in \Gamma_{E_\pm} \sm \Gamma_{E_+}$ and $r(a\sigma\tau) \in \Gamma_{E_+}$, equivalently  $(a\sigma)^{-1}e <0$ and $(a\sigma\tau)^{-1}e>0$.

We can rewrite $\tx{Sh}\,z(\sigma,\tau)$ by using the element $(-1)^e = e \otimes (-1) \in X^*(S) \otimes_\Z \C^\times = \hat S$, i.e. the function $\Gamma_{E_\pm} \lmod \Gamma_F \to \{\pm 1\}$ sending the trivial coset to $-1$ and all other cosets to $1$. In these terms the above formula evaluates to 
\begin{equation} \label{eq:tpi1d}
\tx{Sh}\,z(\sigma,\tau) = \prod_{\substack{a \in \Gamma_{E_\pm}\lmod \Gamma_F\\ \tx{Sh}\,z(\sigma,\tau)(a)=-1}} (-1)^{\delta_a} = \prod_{\substack{a \in \Gamma_{E_\pm}\lmod \Gamma_F\\a^{-1}e>0\\ (a\sigma)^{-1}e<0\\ (a\sigma\tau)^{-1}e >0}} (-1)^{a^{-1}e} = \prod_{\substack{\alpha \in R\\\alpha>0\\ \sigma^{-1}\alpha<0\\ (\sigma\tau)^{-1}\alpha >0}} (-1)^{\alpha}.	
\end{equation}
We now recognize this as the Tits cocycle of Definition \ref{dfn:tp}. 
\begin{rem} \label{rem:spq}
If we change the choice of gauge, the new version of $\tx{Sh}\,z(\sigma,\tau)$ will of course be cohomologous to the old version. But in fact a stronger statement is true -- there is a natural element of $C^1(\Gamma,\hat S)$ whose coboundary relates the two versions, as observed in \cite[Lemma B.4]{KalGRI}. Specifying the formula of that Lemma to the case at hand we obtain the cochain $s_{p/q}$ of \cite[\S2.4]{LS87}.	
\end{rem}

Keeping a gauge $p$ fixed we define the $L$-group of the double cover of $S(F)$ to be $^LS_\pm = \hat S \boxtimes_{t_p} \Gamma_{E/F}$, where $t_p$ is the image of $z$ under the Shapiro map determined by the gauge $p$. Thus $^LS_\pm$ is an extension of $\Gamma_{E/F}$ by $\hat S$ equipped with a set-theoretic splitting, i.e. an $E$-group for $S$ in the sense of \cite{AV92}. Since the group $H^1(\Gamma_{E/F},\hat S)$ need not vanish any more, specifying this splitting gives additional information. A parameter for $^LS_\pm$ is now the same as an element of $C^1(W_F,\hat S)/B^1(W_F,\hat S)$ whose differential is $t_p$. The Shapiro map in degree $1$, again given by the gauge $p$, identifies these with the elements of $C^1(W_{E_\pm},\C^\times_{(-1)})/B^1(W_{E_\pm},\C^\times_{(-1)})$ whose differential is the canonical element $z \in Z^2(\Gamma_{E_\pm/E_+},\C^\times_{(-1)})$. If $\theta$ is a genuine character of $E_+^\times/N_{E_+/E_\pm}(E_+^\times)$ then the image of its parameter \eqref{eq:par1} under the Shapiro map in degree $1$ is given by 
\[ \prod_{a \in \Gamma_{E_\pm} \lmod \Gamma_F} \theta_0(r(a)^{-1}r(a\sigma))^{\delta_{s(a)}}, \]
where $\theta_0 : W_{E_\pm} \to \C^\times$ is an arbitrary extension of $\theta : W_{E_+} \to \C^\times$ satisfying $\theta_0(ww')=\theta(w)\theta_0(w')$ for $w \in W_{E_+}$ and $w' \in W_{E_\pm}$. Of course we have arranged here that $r$ and $s$ take values in $W$ in place of $\Gamma$. 

This formula can be rewritten in another form: If $w_1,\dots,w_n$ are the images of the section $s$, $v_0=1$, and $v_1 \in W_{E_\pm} \sm W_{E_+}$ arbitrary, and we follow the notation of Lemma \ref{lem:zsp}, then we obtain
\begin{equation} \label{eq:par2}
\prod_{i=1}^n \theta(v_0(u_i(w)))^{\delta_{w_i}},	
\end{equation}
which is precisely Definition \ref{dfn:rp} in this case.

\subsection{General tori} \label{sub:gentori}

Let us now collect the preceding observations and see how they, applied to a general torus $S$ equipped with a finite $\Sigma$-invariant subset $0 \notin R \subset X^*(S)$, motivate the definition of the $L$-group of its double cover and the $L$-parameter of the genuine character $\chi_S$ corresponding to a set of $\chi$-data.

The double cover $S(F)_\pm$ of $S(F)$ was obtained as the pull-back of $S(F) \to \prod J_O(F) \from \left(\prod J_O(F)\right)_\pm$ and it makes sense to define its $L$-group as the push\-out of $\hat S \from (\prod \hat J_O) \to {^L(\prod J_O)_\pm}$. Each $J_O$ is an induced one-dimensional anisotropic torus. The $L$-group of its canonical double cover $J_O(F)_\pm$ was discussed in \S\ref{sub:i1dat} -- it is the product of $\hat J_O$ and $\Gamma$ twisted by the 2-cocycle \eqref{eq:tpi1d}. Under the map $\prod \hat J_O \to \hat S$ dual to $S \to \prod J_O$ the image of the product of these cocycles is precisely the Tits cocycle of Definition \ref{dfn:tp}. Thus Definition \ref{dfn:lgrp} of the $L$-group of $S(F)_\pm$ is indeed this pushout.

Consider a set of $\chi$-data $(\chi_\alpha)$. Each $\chi_\alpha$ is a genuine character of the double cover $J_\alpha(F)_\pm$ and its $L$-parameter is given by \eqref{eq:par2}. We now recognize Definition \ref{dfn:rp} as image under the pushout map $^L(\prod J_O)_\pm \to {^LS_\pm}$ of the product over $O$ of the parameters of these characters.

\end{appendices}

\bibliographystyle{amsalpha}
\bibliography{/Users/kaletha/Dropbox/Work/TexMain/bibliography.bib}

\end{document}

%% file: macros.tex
\def\mc#1{\mathcal{#1}}
\def\mb#1{\mathbb{#1}}
\def\tx#1{\textrm{#1}}
\def\tb#1{\textbf{#1}}

\def\R{\mathbb{R}}

\def\C{\mathbb{C}}
\def\Q{\mathbb{Q}}

\def\Z{\mathbb{Z}}

\def\lmod{\setminus}

\def\hat{\widehat}

\def\rw{\rightarrow}

\def\from{\leftarrow}

\def\lrw{\longrightarrow}
\def\llw{\longleftarrow}

\def\sm{\smallsetminus}

\def\<{\langle}
\def\>{\rangle}

\newenvironment{mytitle}
{\begin{center}\large\sc}
{\end{center}}